\documentclass{cpamart1}     %% Input Standard CPAM class.
\received{April 2011}       \revised{November 2011}
\yearofpublication{2012}
\volume{000}
\startingpage{1}                      

\authorheadline{F.\ Rassoul-Agha, T.\ Sepp\"al\"ainen and A.\ Yilmaz}
\titleheadline{Quenched Free Energy and LDP for RWRP}

 \usepackage{amssymb}
\usepackage{mathptmx}
\newtheorem{theorem}{Theorem}[section]
\newtheorem{lemma}[theorem]{Lemma}
\newtheorem{corollary}[theorem]{Corollary}

\theoremstyle{definition}
\newtheorem{definition}[theorem]{Definition}
\theoremstyle{remark}
\newtheorem{remark}[theorem]{Remark}
\newtheorem{example}[theorem]{Example}

\font \mymathbb = bbold10 at 11pt

\newcommand{\one}{\mbox{\mymathbb{1}}}   %Indicator functions

\newcommand{\beq}{\begin{equation}}
\newcommand{\eeq}{\end{equation}}

\RequirePackage{datetime} %comment out when arXiving
\numberwithin{equation}{section}
\allowdisplaybreaks[1]
\usepackage{mathrsfs}
\usepackage{color,colordvi}
\definecolor{my-blue}{rgb}{0.0,0.0,0.6}
\definecolor{my-red}{rgb}{0.5,0.0,0.0}
\definecolor{my-green}{rgb}{0.0,0.5,0.0}
\RequirePackage[colorlinks,urlcolor=my-blue,linkcolor=my-red,citecolor=my-green]{hyperref}
\RequirePackage{hypernat}
\usepackage{bm}

\newcommand*\X{\cX}

\newcommand\unK{K} %The lower bound K
\newcommand\Iqr{I_{q,3}}  %level 3 quenched rate function (forward paths.. main result) r for R_n
\newcommand\Iqell{I_{q,2,\ell}}  %\ell-variate level 2 rate 
\newcommand\pell{\hat p_\ell} %uniform transitions on \bigom_\ell
\newcommand\ratell{H_{\ell,\P}}  %entropy set to \infty at bad measures, ell-variate case 
\newcommand\bigom{\bm\Omega} %bold Omega: space of paths and environments
\newcommand\range{{\sR}} % range of first step
\newcommand\wz{\eta} %environment and path
\newcommand\MC{\cQ} %set of Markov chains
\newcommand\measures{\cM_1} %set of probability measures
\newcommand\Sopr{S^+}

\newcommand\pr{p^+}  %transition kernels:

\newcommand\K{\cK_\ell} %class K  (\bigom_\ell\times\range)
\newcommand\Ll{\cL} %class L  (\bigom_\ell\times\range)
\newcommand\Uset{\cU} %convex hull of \range
\newcommand\unif{\kU_b} %bounded uniformly continuous functions
\newcommand\pres{\Lambda}  %pressure
\newcommand\oneell{\ell} %R_n^\oneell
\newcommand\oneinfty{\infty} %R_n^\oneinfty
\newcommand\gr{{\mathcal G}}%\!\!\!\scriptr{r}(\range)}  %Group generated by \range

\newcommand*{\Z}{{\mathbb Z}} %whole numbers
\newcommand*{\kS}{{\mathfrak S}}

\newcommand*{\cP}{{\mathcal P}}
\newcommand*{\w}{\omega}
\newcommand*{\sR}{{\mathscr R}}
\newcommand*{\E}{{\mathbb E}} %expectation under \P
\newcommand*{\cX}{{\mathcal X}}
\newcommand*{\sB}{{\mathscr B}}
\newcommand*{\cM}{{\mathcal M}}
\newcommand*{\N}{{\mathbb N}} %natural numbers
\newcommand*{\R}{{\mathbb R}} %real numbers
\newcommand*{\cQ}{{\mathcal Q}}
\providecommand{\abs}[1]{\left\vert#1\right\vert}
\newcommand*{\zbar}{{\bar z}}

\newcommand*{\cU}{{\mathcal U}}
\newcommand*{\kU}{{\mathfrak U}}

\newcommand*{\Q}{{\mathbb Q}}
\renewcommand*{\P}{{\mathbb P}} %\P already exists
\newcommand*{\xhat}{{\hat x}}  %used to define [n\xi]

\newcommand*{\cK}{{\mathcal K}}
\newcommand*{\cL}{{\mathcal L}}
\newcommand*{\esssup}{\mathop{{\rm ess~sup}}}
\newcommand*{\e}{\varepsilon}
\newcommand*{\Fhat}{{\widehat F}}
\newcommand*{\Ghat}{{\widehat G}}
\newcommand*{\zhat}{{\hat z}}

\newcommand*{\bfu}{{\mathbf u}}

\newcommand*{\ce}[1]{\left\lceil{#1}\right\rceil}
\newcommand*{\xtil}{{\tilde x}}
\newcommand*{\ztil}{{\tilde z}}
\newcommand*{\xbar}{{\bar x}}
\newcommand*{\abar}{{\bar a}} 
\newcommand*{\uhat}{{\hat u}} 
\newcommand*{\Ftil}{{\widetilde F}}
\newcommand*{\Rtil}{{\widetilde R}}
\newcommand*{\Rhat}{{\widehat R}}
\newcommand*{\nn}{\nonumber}
\newcommand*{\fhat}{{\hat f}}
\newcommand*{\Rbar}{{\bar R}}
\newcommand*{\sF}{{\mathscr F}}
\newcommand\cA{{\mathcal A}}

\newcommand{\refp}{\hat p}    %%%% reference kernel 
 
\DeclareMathOperator\ri{ri} 
\newcommand\Iqella{I_{q,2,\ell+1}}  %\ell+1 -variate level 2 rate 
\providecommand{\absa}[1]{\lvert#1\rvert}
 \def\wt{\widetilde}   
 
\allowdisplaybreaks[1]

\begin{document}                        %% Standard LaTeX command

%%      ---------------------------------------------------------------------
%%      -------------------------------- TITLE -----------------------------
%%      ---------------------------------------------------------------------

\title{Quenched Free Energy and Large Deviations for Random Walks in Random Potentials}

%%      ---------------------------------------------------------------------
%%      ------------------------------- AUTHORS -----------------------------
%%      ---------------------------------------------------------------------
\author{Firas Rassoul-Agha}{University of Utah}
\author{Timo Sepp\"al\"ainen}{University of Wisconsin-Madison}
\author{Atilla Yilmaz}{University of California-Berkeley}

% EXAMPLE: \author{Paris Hilton}{Universit‰ Paris-Sorbonne (Paris IV)}
% Uncomment the following lines as needed
%\author{*** SECOND AUTHOR'S NAME ***}{*** SECOND AUTHOR'S AFFILIATION WHEN ARTICLE WAS WRITTEN ***}
%\author{*** THIRD AUTHOR'S NAME ***}{*** THIRD AUTHOR'S AFFILIATION WHEN ARTICLE WAS WRITTEN ***}
%\author{*** FOURTH AUTHOR'S NAME ***}{*** FOURTH AUTHOR'S AFFILIATION WHEN ARTICLE WAS WRITTEN ***}
%\author{*** FIFTH AUTHOR'S NAME ***}{*** FIFTH AUTHOR'S AFFILIATION WHEN ARTICLE WAS WRITTEN ***}
% Add additional names and affiliations as necessary using above format
%%      ---------------------------------------------------------------------
%%      --------------------------- DEDICATION  (OPTIONAL)------------------- 
%%      ---------------------------------------------------------------------

%       Uncomment the following line to insert a dedication.

%\dedication{ *** DEDICATION *** }        %% Enter dedication between braces.

%%      ---------------------------------------------------------------------
%%      --------------------------- ABSTRACT (OPTIONAL)----------------------
%%      ---------------------------------------------------------------------

%% ***** UNCOMMENT THE FOLLOWING TO INSERT AN ABSTRACT

\begin{abstract}
We study   quenched distributions  on random walks  in a random potential
on integer lattices  of arbitrary dimension and 
with an arbitrary finite set of admissible steps. 
 The potential can be unbounded and can depend on a few steps of the walk.  
Directed,  undirected and stretched  polymers, as well as random walk in random environment, 
  are   covered.
The restriction needed is on the moment
of the potential, in relation to the degree of mixing of the ergodic environment. 
We derive  two variational formulas for the limiting quenched free energy and 
%A variational formula for the point-to-point free energy is also given.
prove  a process-level quenched  large deviation principle for the empirical measure.   
 As a corollary we obtain LDPs for types of random walk in random environment
 not covered by earlier results. 
\end{abstract}

% With AMS-LaTeX, \maketitle follows the abstract
\maketitle   

%%      ---------------------------------------------------------------------
%%      ------------------- TABLE OF CONTENTS (OPTIONAL) --------------------
%%      ---------------------------------------------------------------------

%% ***** IF YOUR PAPER IS OVER 40 PAGES AND YOU WISH TO HAVE A TABLE
%% ***** OF CONTENTS, PLEASE UNCOMMENT THE FOLLOWING LINE

% \tableofcontents

%%      ---------------------------------------------------------------------
%%      ---------------------------- BODY OF PAPER --------------------------
%%      ---------------------------------------------------------------------

%%      Please input or insert the body of your paper here.

\section{Introduction}
This paper investigates the limiting free energy and  large deviations 
  for  several  much-studied lattice models
of random motion in a random  medium.  These include walks in random
potentials, also called polymer models, and the standard 
random walk in random environment (RWRE).  We 
derive variational formulas for the free energy   
and process-level large deviations for the empirical measure.  

\subsection{Walks in random potentials and environments}  \label{sec-model}
We call our basic model  {\sl random walk in a random potential} (RWRP).
A special case is  random walk in random environment (RWRE).   Fix a dimension $d\in\N$. 
There are three ingredients to the model: (i) a reference random walk on $\Z^d$, (ii) an environment, and
(iii)  a potential.

(i) Fix a finite subset $\range\subset\Z^d$.   Let  $P_x$ denote the distribution of the  
discrete time random walk on $\Z^d$ that starts at $x$ and has jump probability 
$\refp(z)=1/|\range|$ for $z\in\range$ and $\refp(z)=0$ otherwise. $E_x$  is expectation under $P_x$.   The walk is denoted by $X_{0,\infty}=(X_n)_{n\ge 0}$. 
Let $\gr$ be the additive subgroup of $\Z^d$  generated by $\range$.

(ii) An {\sl environment} $\w$ is a sample point from
 a probability space $(\Omega, \kS, \P)$.
$\Omega$ comes  equipped with a group   
  $\{T_z:{z\in\gr}\}$   of  measurable commuting  bijections that satisfy 
$T_{x+y}=T_xT_y$ and $T_0$ is the identity.
  $\P$ is a $\{T_z:z\in\gr\}$-invariant probability measure on $(\Omega,\kS)$ 
  that is ergodic under this group.
In other words, 
if $A\in\kS$ satisfies $T_zA=A$ for all $z\in\gr$ then  $\P(A)=0$ or $1$. $\E$ will denote expectation relative to $\P$.  We call  Ê$(\Omega,\kS,\P,\{T_z:z\in\gr\})$Ê
  a {\sl measurable ergodic dynamical system}.   

(iii) A  {\sl potential}  is   a  measurable  function $V:\Omega\times\range^\ell\to\R$, for
some integer $\ell\ge0$.

Given an environment $\w$ and  a starting point $x\in\Z^d$,  for    $n\ge1$   define the 
 {\sl quenched polymer measures} 
\beq	\begin{aligned}
	Q^{V,\w}_{n,x}\{X_{0,\infty}\in A\}
	&=\frac1{Z_{n,x}^{V,\w}}E_x\big[e^{-\sum_{k=0}^{n-1}V(T_{X_k}\w,Z_{k+1,k+\ell})}
	\one_A(X_{0,\infty}) \big]\\
%	&=\frac1{ Z_{n,x}^{V,\w} } |\range|^{-n}\,e^{-\sum_{k=0}^{n-1}V(T_{x_k}\w,z_{k+1,k+\ell})} 
	\end{aligned}  \label{q-meas}\eeq
  normalized by  the {\sl quenched partition function} 
	\[Z_{n,x}^{V,\w}=E_x\big[e^{-\sum_{k=0}^{n-1}V(T_{X_k}\w,Z_{k+1,k+\ell})}\big]
	=\sum_{z_{1,n+\ell-1}\in\range^{n+\ell-1}}\ {|\range|^{-n-\ell+1}}\,e^{-\sum_{k=0}^{n-1}V(T_{x_k}\w,z_{k+1,k+\ell})}. 
	\]
  $Z_k=X_k-X_{k-1}$
is a step of the walk and vectors are    
$X_{i,j}=(X_i,X_{i+1},\dotsc,X_j)$.  
$Q_{n,x}^{V,\w}$  represents the evolution of the polymer in a ``frozen''  environment $\w$. (The picture is that of a heated
sword quenched in water.)
 Let us mention two models of special importance.
%Including the environment gives rise to measures $Q_{n,x}^{V,\w}(dx_{0,n})\P(d\w)$ that  we call the {\sl averaged joint polymer measures} or the {\sl quenched joint polymer measures}. The marginals of these joint measures onto paths $X_{0,n}$ are then referred to as  the {\sl averaged polymer measures}.

% In the simplest case   $V(\w,z_{1,\ell})=\Psi(\w_0,z_{1,\ell})$ depends only on the  $\w$-coordinate at the origin. 

 \begin{example}[$k+1$ dimensional   directed polymer in a random
 environment]  
 Take the  canonical setting:    product space 
 $\Omega=\Gamma^{\Z^d}$  with generic points $\w=(\w_x)_{x\in\Z^d}$ 
 and translations $(T_x\w)_y=\w_{x+y}$.   
Then let    $d=k+1$,  $V(\w)=-\beta \w_0$ with   inverse temperature parameter
$\beta$,   $\range=\{ e_i+e_{k+1}: 1\le i\le k\}$, 
 and the coordinates $\{\w_x\}$   i.i.d.\ under $\P$.  Thus the
 projection of the walk on $\Z^k$ is simple
 random walk, and at every step the walk
 sees a fresh environment.   
\label{dpre} \end{example} 

\begin{example}[Random walk in random environment]\label{ex:RWRE}
RWRE is a Markov chain $X_n$ 
on $\Z^d$ whose transition probabilities are determined by
an environment $\w\in\Omega$.  
Let  $\cP=\{(\rho_z)_{z\in\range} \in[0,1]^\range:\sum_z \rho_z=1\}$ be the set of
probability distributions on $\range$ and  
%  $\cP$ is endowed with  the   topology of weak convergence of probability measures or, equivalently,  with  the restriction of the product topology.
 $p:\Omega\to\cP$ a measurable function with   $p(\w)= (p_z(\w))_{z\in\range}$.   
A transition probability matrix is   defined   
  by 
\[  \pi_{x,y}(\w)= \begin{cases}   p_{y-x}(T_x\w)   &y-x\in\range\\
0 &y-x\notin\range   \end{cases} \qquad  \text{  for    $x,y\in\Z^d$.} \] 
% We call $\w$ and also $(\pi_{x,y}(\w))_{x,y\in\Z^d}$ an environment because it determines the transition  probabilities of a Markov chain.
%We call $\w_x=(\w_{x,z})_{z\in\Z^d}\in\cP$ the environment at point $x\in\Z^d$. 

%The natural ``canonical'' choice for this setting is to take  $\Omega=\cP^{\Z^d}$ with generic points $\w=(\w_x)_{x\in\Z^d}$, translations $(T_x\w)_y=\w_{x+y}$    and $p_z(\w)=(\w_0)_z$.  In this setting, finiteness of $\range$ implies compactness of $\Omega$.  

Given  $\w$ and  $x\in\Z^d$,  $P_x^\w$ is the law of the Markov chain $X_{0,\infty}=(X_n)_{n\ge0}$ on $\Z^d$ with initial point $X_0=x$ and  
transition probabilities $\pi_{y,z}(\w)$. That is,  $P_x^\w$ satisfies 
$P_x^\w\{X_0=x\}=1$ and 
	\[P_x^\w\{X_{n+1}=z\,|\,X_n=y\}=\pi_{y,z}(\w) \ \text{ for all $y,z\in\Z^d.$}\] 
  $P_x^\w$ is called the {\sl quenched} distribution of the walk $X_n$. 
The {\sl averaged}   (or {\sl annealed}) 
distribution    is the path marginal   $P_x(\cdot)=\int P_x^\w(\cdot )\,\P(d\w)$    
of the {\sl joint} distribution  
  $P_x(dx_{0,\infty},d\w)=P_x^\w(dx_{0,\infty})\P(d\w)$.  
%Its marginal on the path space $(\Z^d)^{\Z_+}$ is also denoted by $P_x$ and   called  since $\w$ is averaged out:    \[P_x(\cdot)=\int P_x^\w(\cdot )\,\P(d\w) \quad\text{  for a  measurable $A\subset(\Z^d)^{\Z_+}$.}\]
%We say $\P$ is i.i.d.\ when $\{(\pi_{x,x+z})_{z\in\Z^d}:x\in\Z^d\}$ is an i.i.d.\ collection of random variables.

RWRE is  a special case of \eqref{q-meas} with     $V(\w,z_{1,\ell})=-\log\pi_{0,z_1}(\w)$.
(Note the abuse of  notation: for RWRE $P_0$  is the averaged measure while in  RWRP $P_0$ is  the reference random walk. This should cause no confusion.)

Of particular interest are RWREs where $0$ lies outside  the convex hull
of $\range$.  These are {\sl strictly directed} in the sense that for some $\uhat\in\R^d$,
$z\cdot\uhat>0$ for each admissible step $z\in\range$.   General large deviation theory
for  these walks is covered for the first time in the present paper.  
\end{example}

%We close this section with a remark concerning the connection between the RWRE and RWRP terminology.

%\begin{remark}   The notion of quenched measure is of course the same for RWRE and RWRP.   On the other hand, annealed and averaged measures for RWRP coincide in RWRE.  However, because there is no notion of a path evolving alongside the environment in RWRE,  we choose in this case to use the term averaged rather than annealed.  \end{remark}

\subsection{Results}\label{intro-results}
We have two types of results.   First we prove the $\P$-a.s.\ existence  of  the {\sl quenched  free energy}
	\begin{align}\label{intro:free energy}
	\lim_{n\to\infty}n^{-1}\log Z_{n,0}^{V,\w}=\lim_{n\to\infty} n^{-1} \log E_0\big[e^{-\sum_{k=0}^{n-1}V(T_{X_k}\w,Z_{k+1,k+\ell})}\big] 
	\end{align}
and derive  two variational formulas for the limit.   The   assumption we need
combines  moment bounds on  $V$  with the degree of mixing in $\P$: if $\P$
is merely ergodic we require a bounded $V$,  while with independence 
or exponential mixing 
$L^p$ for $p>d$ is sufficient.   The existence of the limit is not 
entirely new because in some cases  it  
follows from subadditive methods and concentration inequalities. In   
 Example  \ref{dpre}  \cite{Com-Shi-Yos-03} proved the limit under an 
exponential moment assumption  and \cite{Var-07}  with the tail assumption
under which greedy lattice animals are known to have linear growth.  
Our  variational descriptions of the free energy  are new.  
 
The second results are large deviation principles (LDPs) for  
 the quenched  distributions  $Q^{V,\w}_{n,0}\{R_n^{\oneinfty}\in \cdot\,\}$ of the    {\sl empirical process}  
 \[R_n^{\oneinfty}=n^{-1}\sum_{k=0}^{n-1}\delta_{T_{X_k}\w,\,Z_{k+1,\infty}}. \]  
   $T_{X_k}\w$ is the environment   seen from the current position 
 of the walk and   $Z_{k+1,\infty}=(Z_i)_{k+1\le i<\infty}$ is the entire sequence of future steps.    
We assume $\Omega$ separable  metric with  Borel $\sigma$-algebra $\kS$.
Distributions  $Q^{V,\w}_{n,0}\{R_n^{\oneinfty}\in \cdot\,\}$ are   probability measures on  
  $\measures(\Omega\times\range^\N)$,   the space of Borel probability measures on
  $\Omega\times\range^\N$ endowed with the weak topology generated by bounded
  continuous functions.  

The LDP  takes this standard form.   There is a 
 lower semicontinuous  convex rate function
  $\Iqr^V:\measures(\Omega\times\range^\N)\to[0,\infty]$ such that these bounds hold: 
 \begin{equation*}
  \begin{aligned}
	&\varlimsup_{n\to\infty}n^{-1}\log  Q_{n,0}^{V,\w}\{R_n^{\oneinfty}\in C\}
	\le-\inf_{\mu\in C} \Iqr^V(\mu)\ \text{ for all compact sets }C \\
\text{and}	\quad
 &\varliminf_{n\to\infty}n^{-1}\log Q_{n,0}^{V,\w}\{R_n^{\oneinfty}\in O\}
 \ge-\inf_{\mu\in O} \Iqr^V(\mu)\ \text{ for all open sets }O.
\end{aligned}%\label{LDP}
\end{equation*}
 Large deviations of  $R_n^{\oneinfty}$ are  
called   {\sl level 3} or {\sl process level} large deviations.  
For basic   large deviation theory we refer the reader to 
\cite{Dem-Zei-98}, \cite{Hol-00},  
\cite{Deu-Str-89}, \cite{Ras-Sep-10-ldp-}, and \cite{Var-84}. 

 Since we   prove the upper bound only for compact sets the result 
 is technically known as a weak LDP.  In   the   important special
 case of strictly directed walk in an i.i.d.\ environment we   strengthen
 the result to a full LDP where the upper bound is valid for all closed sets.  
 Often  $\Omega$ is compact and then  this issue vanishes.     
As a corollary we obtain large deviations
 for RWRE.  
%The  rate function is then   given by a specific relative entropy; see Lemma 2.2 in \cite{Ras-Sep-10-}.   
 
%The rate function $\Iqr^V$ is given in terms of relative entropy but again, describing it requires further   notation; see Theorem \ref{th:ldp RWRP} and Remark \ref{level 3 RWRP}. 

This paper does   not investigate models that allow 
  $V=\infty$.  
  An example in 
  RWRE would be a walk on a supercritical percolation cluster.

\subsection{Overview of literature and predecessors of this work}
Random walk in random environment was   introduced by Chernov \cite{Che-67-eng} in 1967 and Temkin \cite{Tem-72-eng}
in 1972 as a  model for DNA replication.  
Random walk in random potential appeared  in the work of Huse and Henley \cite{Hus-Hen-85} in 1985
on impurity-induced domain-wall roughening in the two-dimensional Ising model.
The seminal mathematical work on RWRE was Solomon  1975  \cite{Sol-75} 
and on RWRP  Imbrie and Spencer 1988 \cite{Imb-Spe-88} and Bolthausen 1989 \cite{Bol-89}. 
Despite a few decades of effort  many basic questions 
on   (i)  recurrence, transience
 and zero-one laws, (ii) fluctuation 
behavior 
and (iii) large deviations remain only partially answered.  
   Accounts of parts of the state of the art can be found in the lectures 
\cite{Bol-Szn-02-dmv},  \cite{Hav-Ben-02}, \cite{Szn-04} and  \cite{Zei-04}  %\cite{Hau-Keh-87}  
on RWRE, and in  \cite{Com-Shi-Yos-04}, \cite{Hol-09}, \cite{Gia-07},  \cite{Rub-Col-03} %\cite{deG-79},
and \cite{Szn-98} on RWRP.

 Our LDP  Theorem \ref{th:ldp RWRP} specialized to RWRE  covers the quenched 
  level 1  LDPs for RWRE  that have   been established over the last  two
decades.  In the one-dimensional case 
 Greven and den Hollander \cite{Gre-Hol-94} considered the i.i.d.\ nearest-neighbor case, 
 Comets, Gantert, and Zeitouni \cite{Com-Gan-Zei-00} the ergodic   nearest-neighbor case,
 and 
Y\i lmaz \cite{Yil-09-cpam} the ergodic  case with $\range=\{z:|z|\le M\}$ for some $M$.  
In the multidimensional setting  
Zerner \cite{Zer-98-aop}  looked at 
the i.i.d.\ nearest-neighbor nestling case, and Varadhan \cite{Var-03-cpam}  the general ergodic  case with
bounded step size and $\{z:|z|=1\}\subset\range$.  
All these works, with the exception of \cite{Zer-98-aop},  required uniform ellipticity at least on part of $\range$,  
i.e.\ $\pi_{0,z}\ge\kappa$ for a fixed $\kappa>0$ and all $z$ with $|z|=1$.  % and all $z\in\range$. 
\cite{Zer-98-aop} needs $\E[|\log\pi_{0,z}|^d]<\infty$, still for all $|z|=1$. Rosenbluth \cite{Ros-06} gave a variational formula
for the rate function in \cite{Var-03-cpam} under an assumption of 
$p>d$ moments on $\log\pi_{0,z}$, $|z|=1$.

Article  \cite{Ras-Sep-10-}   proved a quenched level 3  LDP 
for RWRE under a  general ergodic environment, subject to bounded steps, 
  $p>d$ moments on $\log\pi_{0,z}$,  and  an  irreducibility assumption
that required the origin   to be accessible from every  $x\in\Z^d$. 
(See Remark \ref{more on loops}  for   more technical explanation of the 
scope of \cite{Ras-Sep-10-}  compared to the present paper.)  
Level 3 large deviations for RWRE have  not appeared in other works.  
  \cite{Yil-09-cpam} gave a 
quenched  univariate  level 2 LDP.  This means that the path component 
in the empirical measure 
has only one step:  $n^{-1}\sum_{k=0}^{n-1}\delta_{T_{X_k}\w,Z_{k+1}}$. 

One goal of the present paper is to eliminate the unsatisfactory 
  irreducibility assumption of \cite{Ras-Sep-10-}.   
This  is important because the  irreducibility assumption
   excluded several  basic and fruitful  models, such as directed polymers,  
   RWRE in a
    space-time, or dynamical, environment
(the case  $\range\subset\{x:x\cdot e_1=1\}$)
  and RWRE with a forbidden direction
  (the case  $\range\subset\{x:x\cdot \uhat\ge0\}$ for some $\uhat\ne0$). 
    Corollary \ref{cor:nestling} shows that 
forbidden direction is  the only case not covered by  \cite{Ras-Sep-10-}, but
 \cite{Ras-Sep-10-} did not address the more general
 polymer model.

Our results   cover the quenched level 1 LDPs for space-time RWRE derived in  
  \cite{Yil-09-aop}  for i.i.d.\ environment in a neighborhood of the asymptotic velocity
  and by 
Avena, den Hollander, and Redig \cite{Ave-Hol-Red-10} for a space-time 
random environment given by a mixing attractive spin flip particle system.
Our results can also be adapted to continuous time to cover the quenched level 1 LDP by Drewitz {\it et al.}\ \cite{Dre-etal-10-} 
for a random walk among a Poisson system of moving traps.

On the RWRP side, Theorems \ref{Lambda=Hstar} and \ref{th:ldp RWRP} cover, respectively, the existence of free energy and the quenched level 1 LDPs
for simple random walk in %$L^p$
random potential proved by Zerner \cite{Zer-98-aap} and the corresponding results for directed simple random walk in 
random potential %with exponential moments
proved by Carmona and Hu \cite{Car-Hu-04} and Comets, Shiga, and Yoshida \cite{Com-Shi-Yos-03}.  (See  \cite{Szn-94} 
for an earlier  continuous counterpart  of \cite{Zer-98-aap}.)  
We also give an entropy interpretation for the rate function and two variational formulas for the free energy,  
while earlier descriptions
of  these objects came in terms of Lyapunov functions and  subadditivity arguments. 
As far as we know,   level 2 or 3 large deviations have not been established
in the past  for RWRP.

%{\bf ** Zerner uses subadditivity to get the Lyapunov functions and the free energy, then uses them to get the qLDP. Comets and also Carmona use subadditivity
%and concentration to get the free energy. Carmona uses subadditivity to get the qLDP (\`a la Varadhan cpam result) **}

%{\bf ** Should we mention that our variational formula has appeared in Kosygina etal 1\&2, Lions and Souganidis, and Schroeder?  
%We do say this in Remark \ref{K-var}. maybe that's where it belongs and here it is kind of out of context? **}

%{\bf ** Should we mention the annealed counterparts, done by Flury for SRWRP (nothing for space-time, other than
%Frank's and Ignatiouk-Robert's) **}

%Turning to the proofs, 
The technical heart of  \cite{Ras-Sep-10-}   was a multidimensional 
extension of a homogenization argument that goes back to Kosygina, Rezakhanlou, and Varadhan \cite{Kos-Rez-Var-06} in the context of diffusion in time-independent 
random potential. This argument was used by Rosenbluth \cite{Ros-06} and Y\i lmaz \cite{Yil-09-cpam} 
to prove  LDPs  for RWRE.  
%It seems reasonable that this version of the method can be made to work for space-time RWRE.
%It is also not quite clear how the method would extend beyond space-time RWRE to the case
%of a forbidden direction, i.e.\ when $\range\subset\{x:x\cdot\uhat\ge0\}$ for some vector $\uhat\in\R^d\smallsetminus\{0\}$.

The main technical  contribution of the current work is a new approach to the homogenization 
argument % in \cite{Kos-Var-08} 
that  allows us to drop the aforementioned irreducibility 
requirement.   One comment to make is
that this construction  that we undertake in Appendix  \ref{app:pf lm class K} 
 does use the invertibility of the transformations $T_z$
assumed in Section \ref{sec-model}.  This is the only place where that
is needed.  

The homogenization method of \cite{Kos-Rez-Var-06} was sharpened by Kosygina
and Varadhan \cite{Kos-Var-08} to handle time-dependent but bounded random potentials.
The results in \cite{Kos-Rez-Var-06} and \cite{Kos-Var-08} concerned  homogenization of stochastic Hamilton-Jacobi-Bellman equations and yielded  variational formulas
for the effective Hamiltonian. For a special case of the random Hamiltonian one can convert these results into quenched large deviations for the velocity of a
diffusion in a random potential, with variational formulas for the quenched free energy.  
Using different methods, \cite{Lio-Sou-05}  and \cite{Arm-Sou-11-} obtain homogenization results similar to \cite{Kos-Rez-Var-06} and \cite{Kos-Var-08}, respectively.
%\cite{Nol-Xin-09} diffusion with random drift, not potential
Furthermore, \cite{Arm-Sou-11-} allows   unbounded potentials and  requires mixing  to compensate for the unboundedness; 
compare with part  (d) of our Lemma \ref{lm-L}.
It is noteworthy that when $d=1$, an ergodic $L^1$ potential is in fact enough; see \cite{E-Weh-Xin-08} and compare with part  (b) of our Lemma \ref{lm-L}.
%  {\bf ** the idea of mentioning these two references is to say that there are other methods out there.. **}
%The recent \cite{Arm-Sou-11-} has results and assumptions somewhat similar to some of the ones in the present manuscript.
%They consider degenerate second and first order Hamilton-Jacobi-Bellman equations,
%%which include as a special case the equation you obtain after a logarithmic transformation from the heat equation with a potential) 
%in stationary ergodic but unbounded environments. Like us they need mixing in order to compensate for the unboundedness. 
%In the particular case of a diffusion with random drift, their condition on the mixing appears to be similar but not exactly the same as the one we assume.
%The methods are again different.

\medskip

We end this section with   some conventions for easy reference. 
%  $\Z_+$, $\Z_-$, and $\N$ denote, respectively, the set of non-negative,  non-positive, and positive integers. For $\range\subset\Z^d$, $\Uset$ denotes its convex hull and $\gr$ denotes the additive group generated by   $\range$.
%  $|\cdot|$ denotes the $\ell^\infty$-norm on $\R^d$.  $\{e_1,\dotsc,e_d\}$ is the canonical basis of $\R^d$. 
For a measurable space $(\X,{\sB})$, %$b\X$ denotes the space of bounded measurable functions.  
$\measures(\X)$  is the space  of probability measures on $\X$ and   $\MC(\X)$    the set  of 
Markov transition kernels on $\X$.  
Given   $\mu\in\measures(\X)$ and  
$q\in\MC(\X)$,
 $\mu\times q$ is the probability measure on 
$\X\times\X$ defined by 
$\mu\times q(A\times B)=\int \one_A(x)q(x,B)\,\mu(dx)$
and $\mu q$ is  its second marginal.
$E^\mu[f]$ denotes  expectation of $f$ under   probability measure $\mu$.
%, occasionally  simplified to  $\mu(f)$.   
  The  increments 
 of   a   path $(x_i)$ in $\Z^d$  
are denoted by $z_i=x_i-x_{i-1}$. 
%The  sequences $(x_i)$ and $(z_i)$  are in 1-1 correspondence. 
Segments of sequences are denoted by     
$z_{i,j}=(z_i,z_{i+1},\dotsc,z_j)$,   also for  $j=\infty$.
%, and also    for   random variables:    $Z_{i,j}=(Z_i,Z_{i+1},\cdots,Z_j)$. $z_{i,i-1}$ is the empty sequence and simply means we are not looking at increments.

\section{Variational representations for free energy}\label{sec:var rep}

Standing assumptions in this section are that Ê$(\Omega,\kS,\P,\{T_z:z\in\gr\})$Ê
is a measurable ergodic dynamical system and,  as throughout the paper,  $\range$ is
an arbitrary finite subset of $\Z^d$ that generates the additive  group $\gr$.
These will not be repeated in the statements of lemmas and theorems.   
Most of the time we also assume that  $\kS$ is countably generated, this will be
mentioned.  The relevant Markov process for this analysis is 
$(T_{X_n}\w, Z_{n+1,n+\ell})$ with state space $\bigom_\ell=\Omega\times\range^\ell$.
The evolution goes via the transformations 
$\Sopr_z(\w,z_{1,\ell})=(T_{z_1}\w, (z_{2,\ell},z))$ 
on $\bigom_\ell$ where the step $z$  is chosen randomly from $\range$ 
as stipulated by the kernel $\refp$.  Elements of $\bigom_\ell$ are 
abbreviated $\wz=(\w,\,z_{1,\ell})$.   

We first look at the  limiting logarithmic moment generating function \eqref{intro:free energy}, also
	called the {\sl pressure} or the {\sl free energy}. 
%	due to  statistical mechanical connections.
To cover much-studied directed polymer models it is  important to go beyond
bounded continuous  potentials.  To achieve this, and at the same time  provide  
a succinct statement of a key hypothesis 
  for Lemma \ref{lm-pressure2}  below, we 
introduce  class $\Ll$ in  the next definition.   
   Let
  \beq D_n=\{z_1+\cdots+z_n\in\Z^d:z_{1,n}\in\range^n\} \label{defDn}\eeq
   denote the set of points accessible from the origin 
  in exactly $n$ steps from $\range$.  
  	\begin{definition}\label{cL-def}
	A function $g:\Omega\to\R$ is in 
	class $\Ll$ if $g\in L^1(\P)$ and for any nonzero $z\in\range$
	\[\varlimsup_{\e\to0}\;\varlimsup_{n\to\infty} \;\max_{x\in \cup_{k=0}^n D_k}\;\frac1n \sum_{0\le i\le\e n} |g\circ T_{x+iz}|=0\quad\P\text{-a.s.}\]
	Similarly,  a function $g$ on $\bigom_\ell$ is a member of $\Ll$ if $g(\cdot\,,z_{1,\ell})\in\Ll$ for each $z_{1,\ell}\in\range^\ell$.
	\end{definition}
	
A bounded $g$ is  in   $\Ll$ under an arbitrary ergodic $\P$, and so is any $g\in L^1(\P)$ if $d=1$. In general  there is a trade-off between the degree of  mixing in $\P$ 
and the moment of $g$  required.    For example, if sufficiently separated  shifts of 
$g$ are i.i.d.\ or there is exponential mixing, then $g\in L^p(\P)$ for some $p>d$  guarantees  $g\in\Ll$.   Under polynomial mixing a higher moment is needed.   
  Lemma \ref{lm-L} in Appendix \ref{app-pressure} collects sufficient 
  conditions for membership in $\Ll$.

We have two variational formulas for the free energy.  One is  
duality in terms of entropy.  The other  involves 
a functional $\unK_\ell(g)$  
 defined by  a minimization over gradient-like  auxiliary functions. 
Class $\K$ below is a generalization of a class of functions previously introduced by \cite{Ros-06}.
%  First we develop $\unK_\ell(g)$.  
  
\begin{definition}
\label{cK-def}
A measurable function $F:\bigom_\ell\times\range\to\R$ is in 
class $\K$  
if it satisfies the following three conditions.
\begin{itemize}
\item[(i)] Integrability: for each $z_{1,\ell}\in\range^\ell$ and $z\in\range$, 
$\E[|F(\w,z_{1,\ell},z)|]<\infty$.
\item[(ii)] Mean zero: for all $n\ge\ell$ and $\{a_i\}_{i=1}^n\in \range^n$ the following holds.
If $\wz_0=(\w, a_{n-\ell+1,n})$ and $\wz_i=\Sopr_{a_i}\wz_{i-1}$ for $i=1,\dotsc, n$, then
\begin{align*}
\E\Big[\sum_{i=0}^{n-1} F(\wz_i,a_{i+1})\Big]=0.
\end{align*}
%This implies that $\E[f(\w,\zzl,\zzl,x)]=0$ for all $\zzl$ and $x$.
In other words, expectation vanishes
whenever the sequence of moves $\Sopr_{a_1},\dotsc,\Sopr_{a_n}$
 takes $(\w, z_{1,\ell})$ to $(T_x\w, z_{1,\ell})$
for all $\w$, for fixed $x$ and $z_{1,\ell}$. 
\item[(iii)] Closed loop: for $\P$-a.e.\ $\w$ and 
any two paths $\{\wz_i\}_{i=0}^n$ and 
$\{\bar\wz_j\}_{j=0}^m$ with 
$\wz_0=\bar\wz_0=(\w,z_{1,\ell})$, $\wz_n=\bar\wz_m$,  
$\wz_i=\Sopr_{a_i}\wz_{i-1}$,
and $\bar\wz_j=\Sopr_{\bar a_j}\bar\wz_{j-1}$, for $i,j>0$ and some 
$\{a_i\}_{i=1}^n\in\range^n$ and $\{\bar a_j\}_{j=1}^m\in\range^m$, 
we have
\begin{align*}
&\sum_{i=0}^{n-1} F(\wz_i,a_{i+1})
=\sum_{j=0}^{m-1} F(\bar\wz_j,\bar a_{j+1}).
\end{align*}
\end{itemize}
\end{definition}

 In case of a loop ($\wz_0=\wz_n$) in  (iii) above  one can take $m=0$ and the 
right-hand side in the  display vanishes.   The simplest members of $\K$ are gradients
$F(\wz,z)=h(\Sopr_z\wz)-h(\wz)$ with bounded measurable
  $h:\bigom_\ell\to\R$.    Lemma \ref{L1 limit} in the appendix shows that 
   $\K$  is the  $L^1(\P)$-closure of   such gradients.  

For $F\in\K$ and $g:\bigom_\ell\to\R$  such that $g(\cdot\,,z_{1,\ell})\in L^1(\P)$ for all $z_{1,\ell}\in\range^\ell$ 
define
	\[K_{\ell,F}(g)=\P\text{-}\esssup_\w\; \max_{z_{1,\ell}} \, \log 
	\sum_{z\in\range} \frac1{|\range|}e^{g(\wz)+F(\wz,z)}\]
and then 
	\begin{align*}%\label{K-def}
	\unK_\ell(g)=\inf_{F\in\K} K_{\ell,F}(g).
	\end{align*}

The reference walk $\refp$ with uniform steps from $\range$ defines 
a Markov kernel 
  $\pell$ on $\bigom_\ell$   by 
  	\begin{align}\label{pell-def}
		&\pell(\wz,\Sopr_z\wz)=\tfrac1{|\range|} \, \text{ for }z\in\range\text{ and }\wz=(\w,z_{1,\ell})\in\bigom_\ell. 
	\end{align}
%The shift transformation $\Sopr_z$ was defined in \eqref{shifts}.
  Let $\mu_0$ denote the $\Omega$-marginal of a measure $\mu\in\measures(\bigom_\ell)$.
  Define an entropy $\ratell$ on $\measures(\bigom_\ell)$
	by
	\begin{align}\label{Helldef}
		\ratell(\mu)=
			\begin{cases}
				\inf\{H(\mu\times q\,|\,\mu\times \pell):q\in\MC(\bigom_\ell)\text{ with }\mu q=\mu\}&\text{if }\mu_0\ll\P,\\
				\infty&\text{otherwise.}
			\end{cases}
	\end{align}
Inside the braces the familiar relative entropy is  
\begin{align*}
	H(\mu\times q\,|\,\mu\times\pell)
	=\int\sum_{z\in\range}q(\wz,\Sopr_z\wz)\,\log\frac{q(\wz,\Sopr_z\wz)}{\pell(\wz,\Sopr_z\wz)}\,\mu(d\wz).
\end{align*}
$\ratell: \measures(\bigom_\ell)\to [0,\infty]$ is convex. (The  argument for this
can be found at the end of Section 4 in \cite{Ras-Sep-10-}.)   
For measurable functions $g$ on $\bigom_\ell$ define 
		\begin{align}\label{Hsharp}
		\ratell^\#(g)=\sup_{\mu\in\measures(\bigom_\ell),\,c>0}\{E^\mu[\min(g,c)]-\ratell(\mu)\}.
		\end{align}
For  $g$   from the space of bounded measurable functions (or bounded continuous functions
if $\Omega$ comes with a metric)   
$\ratell^\#(g)$ is the convex dual of $\ratell$, and then we 
  write  $\ratell^*(g)$.    The constant   $R=\max\{|z|:z\in\range\}$ appears also
  frequently in the results. 
  
For the rest of the section we fix $\ell\ge 0$ and consider measurable functions
$g: \bigom_\ell\to\R$.    		

	\begin{theorem}\label{Lambda=Hstar}
%	 Let $(\Omega,\kS,\P,\{T_z:z\in\gr\})$ be a measurable ergodic system and assume $\range$ is finite and 
	 Assume $\kS$ is countably generated. 
Let  $g\in\Ll$.
	 Then, for $\P$-a.e.\ $\w$ the limit 	
	 		\begin{align*}%\label{pres-limit}
		\pres_\ell(g)=\lim_{n\to\infty}n^{-1}\log E_0\big[e^{\sum_{k=0}^{n-1}g(T_{X_k}\w,\,Z_{k+1,k+\ell})}\big]
		\end{align*}
	exists, is deterministic, and satisfies $ \pres_\ell(g)=\unK_\ell(g)=\ratell^\#(g). $ 
	\end{theorem}
	
%Note that if $g$ is bounded and continuous then $\ratell^\#(g)=\ratell^*(g)$.

 \begin{remark}  The limit  $ \pres_\ell(g)$ satisfies these bounds:
	\begin{align}\label{La-lb}
	\E\Big[\min_{z_{1,\ell}\in\range^\ell} g(\w,z_{1,\ell})\Big]\le\pres_\ell(g)\le
	\varlimsup_{n\to\infty}\;
	\max_{\substack{x_i-x_{i-1}\in\range\\ 1\le i\le n}}
	\; n^{-1}  \sum_{k=0}^{n-1}  \max_{\wt z_{1,\ell}\in\range^\ell} g(T_{x_k}\w,\wt z_{1,\ell}). 
	\end{align}  
The upper bound is nonrandom by invariance.  
The lower bound  comes from  
  ergodicity of the Markov chain $T_{X_n}\w$   (Lemma 4.1 in \cite{Ras-Sep-10-})
 and Jensen's inequality: 
\[n^{-1}\log E_0\big[e^{\sum_{k=0}^{n-1}g(T_{X_k}\w,Z_{k+1,k+\ell})}\big]\ge n^{-1}\sum_{k=0}^{n-1} E_0\Big[\min_{z_{1,\ell}}g(T_{X_k}\w,z_{1,\ell})\Big].\]
If $g$ is unbounded from above and $\range$ allows the walk to revisit
sites then a situation where $\pres_\ell(g)=\infty$ can be easily created. 
%The point of stating the   upper bound is that 
Under some
independence and moment assumptions the limit on the right  in \eqref{La-lb}
is known to be
a.s.\ finite.   
 \end{remark}

\begin{remark}
Suppose $\Omega$ is a product space with i.i.d.\ coordinates $\{\w_x\}$ under
$\P$,   the walk is   strictly directed ($0$ does not lie in the convex hull of
$\range$), and    $g(\cdot\,, z_{1,\ell})$ is a local function on $\Omega$.  
Then the assumption  $\E\abs{g(\cdot\,, z_{1,\ell})}^p<\infty$,  for some $p>d$ 
and all $z_{1,\ell}\in\range^\ell$,   
is sufficient for the above Theorem \ref{Lambda=Hstar} and the finiteness
of the limit $ \pres_\ell(g)$.  That such $g\in\cL$ is proved 
in  Lemma \ref{lm-L} in Appendix \ref{app-pressure}. 
Under this moment bound, finiteness of the upper bound in \eqref{La-lb} follows
from lattice animal bounds  
  \cite{Cox-Gan-Gri-93, Gan-Kes-94, Mar-02}.   \end{remark}

\begin{remark}\label{K-var}
If $\range=\{\pm e_1,\dotsc,\pm e_d\}$ or $\range=\{e_1\pm e_2,\dotsc,e_1\pm e_d\}$ and 
if we take $g$ to be a function of $\w$ only, then $\pres_0(g)=\unK_0(g)=H^\#_{0,\P}(g)$ 
corresponds to a discretization of the 
variational formula for the effective Hamiltonian $\overline H$ of the homogenized stochastic 
Hamilton-Jacobi-Bellman equation considered in \cite{Kos-Rez-Var-06}, \cite{Kos-Var-08}, and \cite{Lio-Sou-05}. 
%See page 1497 of \cite{Kos-Rez-Var-06} and ?? of \cite{Kos-Var-08}.
It is also related to the variational formula for the exponential decay rate of the Green's function of Brownian 
motion in a periodic potential; see (1.1) in \cite{Sch-88}. 
\end{remark}
 
Here is an outline of the proof of  Theorem \ref{Lambda=Hstar}.  
Introduce the empirical measure
	$R_n^{\oneell}=n^{-1}\sum_{k=0}^{n-1}\delta_{T_{X_k}\w,\,Z_{k+1,k+\ell}}$ 
so that  $nR_n^{\oneell}(g)=\sum_{k=0}^{n-1}g(T_{X_k}\w,Z_{k+1,k+\ell})$ gives 
	convenient compact notation for the sum in the exponent.  
Let 
\[  \overline\pres_\ell(g,\w)=\varlimsup_{n\to\infty}n^{-1}\log E_0\Big[e^{n R_n^{\oneell}(g)}\Big]
\quad\text{and}\quad 
 \underline\pres_\ell(g,\w)=\varliminf_{n\to\infty}n^{-1}\log E_0\Big[e^{n R_n^{\oneell}(g)}\Big]. 
\]
 The existence of  $\pres_\ell(g)$ 
and the  variational formulas are
established  through the inequalities 
 	\begin{align}\label{inequalities1}
	\overline\pres_\ell(g)\mathop{\le}^{\rm(i)}\unK_\ell(g)\mathop{\le}^{\rm(ii)} \ratell^\#(g)\mathop{\le}^{\rm(iii)}\underline\pres_\ell(g). 
	\end{align}
 Inequality (\ref{inequalities1}{\color{my-red}.i}) is proved in Lemma \ref{lm-new-upper}.
 This is the only step that requires  
  $g\in\Ll$ 
rather than just   $L^1(\P)$. %see also Lemma \ref{lm-pressure2} for the exact place where this requirement is used. 
Inequality (\ref{inequalities1}{\color{my-red}.ii}) is proved in Lemma \ref{lm:iv}. This is where the main technical
effort of the paper lies, in order to  relax the  irreducibility assumption on $\range$ used in \cite{Ras-Sep-10-}. 
Bound  (\ref{inequalities1}{\color{my-red}.iii}) is proved with 
the usual   change of measure argument.  It follows as a special
case from  Lemma \ref{lower-bound-lemma} below.
The proof of Theorem \ref{Lambda=Hstar}   comes at the end of this section after the  
lemmas.  
To improve the readability of this section    some lemmas
are proved in an appendix   at the end  of the paper.  

\begin{remark}\label{more on loops} 
Suppose  $0$ lies  in the relative interior of the convex hull of $\range$.
Then for every $x\in\gr$ there exists $z_{1,n}\in\range^n$ with $x_n=x$ 
 (Corollary \ref{cor:nestling}).
Under this irreducibility 
the approach of \cite{Ras-Sep-10-} becomes available and can be used
to prove our results under the assumption 
 that   $g(\,\cdot\,,\,z_{1,\ell})\in L^p(\P)$ for some $p>d$ and all $z_{1,\ell}\in\range^\ell$.  
   In this case (\ref{inequalities1}{\color{my-red}.i})  is proved via a slight variation of
 Lemma 5.2 of \cite{Ras-Sep-10-} rather than our Lemma \ref{lm-new-upper}. This  relies crucially on 
 Lemma 5.1 of \cite{Ras-Sep-10-}  which is where $p>d$ moments are required. We replace this with the much weaker Lemma 
 \ref{CLASS K} which only requires one moment, but then we need Lemma \ref{lm-new-upper} which requires $g\in\Ll$. 
 \end{remark}

%Recall the definition \eqref{xhat-def} of $\xhat_n(\zeta)$ for $\zeta\in\Uset$.
%For $\xi\in\Q^d\cap\Uset$ we can write $\xi=\sum_{z\in\range}\alpha_z z$ with rational coefficients $\alpha_z\ge 0$ such that $\sum\alpha_z=1$; 
%see Lemma \ref{lm-convex} in the appendix.
%Let $k\in\N$ be such that $k\alpha_z\in\Z_+$ for all $z\in\range$. Then there is a path
%$\xhat_{0,k}(\xi)$ with steps $\xhat_i(\xi) - \xhat_{i-1}(\xi)\in\range$ and such that $\xhat_k(\xi)=k\xi$. (The path has exactly $k$ steps because $\sum k\alpha_z=k$.)
%Let $\xhat_n(\xi)=mk\xi+\xhat_{n-mk}(\xi)$ for $mk\le n<(m+1)k$.  

%The first  issue in proving Theorem \ref{Lambda=Hstar} is to tackle is the existence of the limiting free energy and its
%variational representations.   

We turn to developing inequalities \eqref{inequalities1}.  
Decomposing the free energy according to 
asymptotic directions $\xi$ turns out useful. 
Let $\Uset$ be the (compact) convex hull of $\range$ in $\R^d$. 
For each  rational  point $\xi\in\Uset$ fix a positive integer  $b(\xi)$ such that 
$b(\xi)\xi\in D_{b(\xi)}$ (recall definition \eqref{defDn} of $D_n$). 
 The existence of    $b(\xi)$ follows from Lemma 
\ref{lm-convex} in Appendix \ref{app-pressure}.    Then fix a path $\{\xhat_n(\xi)\}_{n\in\Z_+}$, starting at $\xhat_0(\xi)=0$, 
with  admissible steps $\xhat_n(\xi)-\xhat_{n-1}(\xi)\in\range$ and such that 
$\xhat_{jb(\xi)}(\xi)=jb(\xi)\xi$ for all $j\in\Z_+$.   Even though stationarity and ergodicity are standing assumptions in this section,
the next lemma actually needs no assumptions on $\P$.
 
\begin{lemma}\label{lm-pressure2}  
%Assume $\P$ is $\{T_z:z\in\gr\}$-invariant. % and $\range$ is finite. %, and $\Omega$ is compact.
Let $g\in\Ll$. Then  for $\P$-a.e.\ $\w$ 
 	\begin{align}\label{pres-a}
	%\begin{split}
	&\varlimsup_{n\to\infty}n^{-1}\log E_0\Big[e^{n R_n^{\oneell}(g)}\Big]%\\
	%&\qquad\qquad
	\le\sup_{\xi\in \Uset\cap\Q^d}\varlimsup_{n\to\infty}n^{-1}\log E_0\Big[e^{n R_n^{\oneell}(g)}\one\{X_n=\xhat_n(\xi)\}\Big].
	%\end{split}
	\end{align}
\end{lemma}

\begin{proof}%[Proof of Lemma \ref{lm-pressure2}]
Fix a small $\e>0$, an integer $k\ge|\range|\e^{-1}$, and a nonzero  $\zhat\in\range$. 
For $x\in D_n$ write $x=\sum_{z\in\range} a_z z$ with $a_z\in\Z_+$ and $\sum_{z\in\range}a_z=n$. 
Let $m_n=\ce{n/(k(1-2\e))}$ and $s_z^{(n)}=\ce{k(1-2\e) a_z/n}$. %\ge a_z/m_n$.
Then 
%	\[(1-\tfrac1{1+\e})n^{-1}a_z-\tfrac1k\le n^{-1}a_z - k^{-1}s_z^{(n)}\le(1-\tfrac1{1+\e})n^{-1}a_z.\]
%This implies that \[0\le 1-k^{-1}\sum_z s^{(n)}_z\le1-\tfrac1{1+\e}<\tfrac\delta{2R}\]
$k^{-1}\sum_z s^{(n)}_z\le1-\e$ and $m_n s^{(n)}_z\ge a_z$ for each $z\in\range$.
% and \[\Big| k^{-1}\sum_{z\in\range} s^{(n)}_z z -x/n\Big|\le R \sum_{z\in\range} |s^{(n)}_z/ k -a_z/n|\le R(1-\tfrac1{1+\e})<\tfrac\delta2.\]
Let 
  \beq \xi(n,x)=k^{-1}\sum_{z\in\range} s^{(n)}_z z+\Big(1-k^{-1}\sum_{z\in\range} s^{(n)}_z\Big)\zhat.\label{xink}\eeq 
 Then, $\xi(n,x)\in k^{-1}D_k$.  With $\e$ fixed small enough 
 and considering $n>k/\e$, 
  % One can now get from $x$ to $m_n k\xi(n,x)$ by taking $(m_n b_z-a_z)$ $z$-steps, for 
% each $z\in\range$, followed by $m_n(k-\sum_z b_z)$ $\zhat$-steps, with a total of $m_n k-n$ steps. 
  we   constructed %a $\xi(n,x)\in k^{-1}D_k$ and 
an admissible path  of $  m_n k-n\le 4n\e$ steps    from $x$ to  
$m_n k\xi(n,x)$.  This path has at least 
$m_n(k-\sum s^{(n)}_z) \ge m_n k\e \ge n\e/(1-2\e) $  $\zhat$-steps.  
Consequently at least a fixed fraction $\delta$ of the steps of the path 
are $\zhat$-steps, for all $x\in D_n$ and all $n$. 

Let $b$ be the least common multiple of the (finitely many) 
integers $\{b(\xi): \xi\in k^{-1}D_k\}$.  Now  we take 
another   bounded number of additional  steps to get from $m_n k\xi(n,x)$ to the path $\xhat_\centerdot(\xi(n,x))$.    
Pick $\ell_n$ such that $(\ell_n-1)b<m_n\le \ell_nb$.   Then by repeating the steps
of $k\xi(n,x)$ 
in \eqref{xink} $\ell_nb-m_n \le b$ times, we go from  $m_n k\xi(n,x)$ 
to $\ell_n kb\xi(n,x)=\xhat_{\ell_nkb}(\xi(n,x))$.   The duration of  this last leg is bounded
independently of $n$ and $x\in D_n$ because  $k$ 
was fixed at the outset  and $b$ is determined by $k$. 
   Thus the total number of steps
from $x\in D_n$ to $\xhat_{\ell_nkb}(\xi(n,x))$  is   
$r_n= \ell_n kb-n\le 5n\e$, for large enough $n$. 
Let    $\bfu(n,x)=(u_1,\dotsc,u_{r_n})$ denote this sequence of steps.  Again
we note that at least a fixed fraction $\delta$ of the $u_i$'s 
are $\zhat$-steps. 

Develop an estimate: 
\begin{align*}%\label{error-computation}
%\begin{split}
&\frac1n\log E_0\Big[e^{n R^{\oneell}_n(g)}\Big] 
= \frac1n\log \sum_{x\in D_n} E_0\Big[e^{n R^{\oneell}_n(g)}, X_n=x\Big]\\
&\le \max_{x\in D_n}
\frac1n\log   E_0\Big[e^{(n-\ell) R^{\oneell}_{n-\ell}(g)}, X_n=x\Big]  + 
\max_{w\in D_{n-\ell}} \max_{y\in \cup_{s=0}^\ell D_s}\frac {\ell} n  \bar g(T_{w+y}\w) + \frac{C\log n}{n}\\
&\le \max_{x\in D_n}
\frac1n\log   E_0\Big[e^{\ell_n bk R^{\oneell}_{\ell_n bk}(g)}, X_{\ell_n bk}=\xhat_{\ell_n bk}(\xi(n,x))\Big]+ \frac{C\log n}{n}\\
&\ 
+\max_{w\in D_{n-\ell}}\max_{y\in \cup_{s=0}^\ell D_s} \frac {2\ell} n  \bar g(T_{w+y}\w)+\max_{x\in D_n} \frac1n \sum_{i=1}^{r_n} \bar g(T_{x+u_1+\cdots+u_i}\w) +\frac{r_n}n\log|\range|.
%\end{split}
\end{align*}
Above, $\bar g(\w)=\max_{z_{1,\ell}\in\range^\ell} |g(\w,z_{1,\ell})|$.
The second-to-last line of the above display is bounded above by 
\[
 \max_{\xi\in k^{-1}D_k}
\frac1n\log   E_0\Big[e^{\ell_n bk R^{\oneell}_{\ell_n bk}(g)}, X_{\ell_n bk}=\xhat_{\ell_n bk}(\xi)\Big]+ \frac{C\log n}{n}
\]
and so its   limsup  is almost surely at most 
\[(1+5\e)\sup_{\xi\in \Uset\cap\Q^d}\varlimsup_{n\to\infty}n^{-1}\log E_0\Big[e^{n R_n^{\oneell}(g)}\one\{X_n=\xhat_n(\xi)\}\Big].\]
The proof of \eqref{pres-a} is complete  once we show that a.s.
\begin{align}\label{to-show}
\begin{split}
&\varlimsup_{\e\to0}\varlimsup_{n\to\infty} \max_{x\in D_n} \frac1n\sum_{i=1}^{r_n} \bar g(T_{x+u_1+\cdots+u_i}\w)=0\\
\text{and }\quad&\varlimsup_{\e\to0}\varlimsup_{n\to\infty} \max_{w\in D_{n-\ell}}\max_{y\in \cup_{s=0}^\ell D_s} \frac 1 n  \bar g(T_{w+y}\w)=0.
\end{split}
\end{align} 

To this end, observe that the ordering of the steps of $\bfu(n,x)$ 
was so far immaterial.  Because Definition \ref{cL-def} cannot handle 
zero steps, we need to be careful about them.   The ratio of 
 zero steps to  $\zhat$-steps is at most 
$t=\ce{\delta^{-1}}$.  We begin 
$\bfu(n,x)$ by alternating  $\zhat$-steps with   blocks of at most $t$ zero steps,  until the $\zhat$ steps and the
zero steps are exhausted.   After that order the remaining
nonzero steps of     $\range$  in any fashion $z_1,z_2,\dotsc$, and 
have  $\bfu(n,x)$  take first all its $z_1$ steps, then all its $z_2$ steps,
and so on.  Since zero steps do not shift $\w$ but simply repeat the same 
$\bar g$-value at most $t$ times,  we get the bound  
	\[\sum_{i=1}^{r_n} \bar g(T_{x+u_1+\cdots+u_i}\w)\le t |\range| \max_{y\in x+\bfu(n,x)}\;\max_{z\in\range\setminus\{0\}}\sum_{i=0}^{r_n} \bar g(T_{y+iz}\w).\] 
By $y\in x+\bfu(n,x)$ we mean $y$ is on the path starting from $x$
and taking steps in $\bfu(n,x)$.   A similar bound develops for the second line of 
\eqref{to-show}, and the limits in \eqref{to-show} follow from membership in $\Ll$.
%Then \eqref{to-show} follows from showing that for each $z\in\range$ and $C>0$ fixed,
%\begin{align*} \varlimsup_{\e\to0}\varlimsup_{n\to\infty} \max_{|x|\le Cn} \frac1n \sum_{i=0}^{\e n} \bar g(T_{x+iz}\w)=0.  \end{align*} 
%This in turn follows from membership in $\Ll$. The lemma is proved.
\end{proof}

The next step is to  show (\ref{inequalities1}{\color{my-red}.i}): for $g\in\Ll$ and $\P$-a.e.\ $\w$, $\overline\pres_\ell(g,\w)\le \unK_\ell(g)$.
The following ergodic property is crucial.  Recall the definition of the
path $\xhat_\cdot(\xi)$  above Lemma \ref{lm-pressure2}.  
For   $\xi\in\Q^d\cap\Uset$ and $z_{1,\ell}$, $\zbar_{1,\ell}\in\range^\ell$ define 
\begin{align*}
 \cA_n(\xi,z_{1,\ell},\zbar_{1,\ell}) = \{  (a_1,\dotsc,a_n)\in\range^n: \,
&z_1+\dotsm+z_\ell+a_1+\dotsm+a_{n-\ell}=\xhat_n(\xi), \\   
&\qquad   \, a_{n-\ell+1,n}=\zbar_{1,\ell}\}.
\end{align*}
This is the set of 
steps $(a_1,\dotsc,a_n)$ 
that take  $\wz_0=(\w,z_{1,\ell})$ to $\wz_n=(T_{\xhat_n(\xi)}\w,\zbar_{1,\ell})$
 via  $\wz_i=\Sopr_{a_i}\wz_{i-1}$, 
  $1\le i\le n$.  

\begin{lemma}
\label{CLASS K}
%Let $(\Omega,\kS,\P,\{T_z:z\in\gr\})$ be a measurable ergodic system and
%assume $\range$ is finite. % and $\Omega$ is compact. 
%Fix $\ell\ge0$. 
Let $F\in\K$. 
Then, for each $\xi\in\Q^d\cap\Uset$ and $z_{1,\ell}$, $\zbar_{1,\ell}\in\range^\ell$, 
\begin{align*}\lim_{n\to\infty}\;
\max_{(a_1,\dotsc,a_n)\in\cA_n(\xi,z_{1,\ell},\zbar_{1,\ell})}  
\Bigl\lvert \,  \frac1n  \sum_{i=0}^{n-1} F(\wz_i,a_{i+1})\Bigr\rvert =0
\quad \text{in $L^1(\P)$ and  for $\P$-a.e. $\w$. }
\end{align*}
\end{lemma}

\begin{remark}
Due to the closed loop property (iii) in Definition \ref{cK-def}, the sum   above  
is independent of   $(a_1,\dotsc,a_n)\in \cA_n(\xi,z_{1,\ell},\zbar_{1,\ell})$.
In other words, there actually is no maximum. 
Also, Lemma \ref{CLASS K} holds regardless of the choices made  
 in the definition of $\xhat_\cdot(\xi)$. 
\end{remark}

We postpone the proof of Lemma \ref{CLASS K}  to Appendix \ref{app:pf lm class K}. 

\begin{lemma}\label{lm-new-upper}
%Let $(\Omega,\kS,\P,\{T_z:z\in\gr\})$ be a measurable ergodic system and assume $\range$ is finite. % and $\Omega$ is compact.
%Fix $\ell\ge0$ and 
Let $g\in\Ll$.
Then $\overline\pres_\ell(g,\w)\le \unK_\ell(g)$ 
for $\P$-a.e.\ $\w$.
\end{lemma}

\begin{proof}
  By Lemma \ref{lm-pressure2} it is enough to show that  
\[   \varlimsup_{n\to\infty}n^{-1}\log E_0\Big[e^{n R_n^{\oneell}(g)}
\one\{X_n=\xhat_n(\xi)\}\Big]   \le K_{\ell,F}(g)  \quad\text{    $\P$-a.s.} \] for  fixed  
$\xi\in\Q^d\cap\Uset$ and $F\in\K$.
Abbreviate $\wz_k=(T_{X_k}\w,Z_{k+1,k+\ell})$. Fix $\e>0$. Lemma \ref{CLASS K}
implies that for $\P$-a.e.\ $\w$ there exists a finite $c_\e(\w)$ such that for all $n$, on
the event  $\{X_n=\xhat_n(\xi)\}$, 
\[\sum_{k=0}^{n-1}F(\wz_k,Z_{k+\ell+1})\ge-c_\e-n\e.  %,\quad P_0\text{-a.s.}
\]
Therefore, for $\P$-a.e.\ $\w$
{\allowdisplaybreaks
\begin{align*}
		&-n^{-1}c_\e-\e+n^{-1}\log E_0\Big[e^{nR_n^{\oneell}(g)}\one\{X_n=\xhat_n(\xi)\}\Big]\\
		&\le n^{-1}\log E_0\Big[\exp\Big\{{\sum_{k=0}^{n-1}\Big(g(\wz_k)+F(\wz_k,Z_{k+\ell+1})\Big)}\Big\}\one\{X_n=\xhat_n(\xi)\}\Big]\\
		&\le n^{-1}\log E_0\Big[\exp\Big\{{\sum_{k=0}^{n-1}\Big(g(\wz_k)+F(\wz_k,Z_{k+\ell+1})\Big)}\Big\}\Big]\\
		&=n^{-1}\log E_0\Big[\exp\Big\{{\sum_{k=0}^{n-2}\Big(g(\wz_k)+F(\wz_k,Z_{k+\ell+1})\Big)}\Big\}
		E_0\big[e^{g(\wz_0)+F(\wz_0,Z_{\ell+1})}\,\big|\,\wz_{n-1}\big]\Big]\\
		&\le n^{-1} K_{\ell,F}(g) +n^{-1}\log E_0\Big[\exp\Big\{{\sum_{k=0}^{n-2}\Big(g(\wz_k)+F(\wz_k,Z_{k+\ell+1})\Big)}\Big\}\Big]\\
		&\le\cdots\le K_{\ell,F}(g).
\end{align*}}
The claim  follows by taking $n\nearrow\infty$  and then $\e\searrow0$.
\end{proof}

We have shown (\ref{inequalities1}{\color{my-red}.i}) and 
next  in line is   (\ref{inequalities1}{\color{my-red}.ii}).   The following lemma
is the most laborious step in the paper.  

\begin{lemma}\label{lm:iv}   In  addition to ergodicity assume now 
%Let $(\Omega,\kS,\P,\{T_z:z\in\gr\})$ be a measurable ergodic system.
%Assume $\range$ is finite and 
that $\kS$ is countably generated. % and $\Omega$ is compact.
%Fix $\ell\ge0$.
Assume $g(\cdot\,,z_{1,\ell})\in L^1(\P)$ is bounded above.  Then
\begin{align*}
\unK_\ell(g)\le \ratell^\#(g)=\sup_{\mu\in \measures(\bigom_\ell)}\{E^\mu[g]-\ratell(\mu)\}.
\end{align*}
\end{lemma}

\begin{proof}%[Proof of Lemma \ref{lm:iv}]
We can assume  $\ratell^\#(g)<\infty$.  The first technical issue is to find 
some compactness to control the supremum on the right.  
Assuming $\Omega$ compact would not be helpful because the problem is the absolute
continuity condition  in the definition of $\ratell(\mu)$.

Fix a sequence of increasing finite algebras $\kS_k$ on $\Omega$ that  
satisfy $T_{\pm z}\kS_{k-1}\subset\kS_k$ for all $k\in\N$ and $z\in\range$, and whose union generates $\kS$.
Let $\measures^k=\measures^k(\bigom_\ell)$ be the set of probability measures $\mu$ on $\bigom_\ell$ for which there exist $\kS_k$-measurable Radon-Nikodym
derivatives  $\phi_{z_{1,\ell}}$  on $\Omega$  (with respect to $\P$)
such that for bounded measurable $G$
\[  \int_{\bigom_\ell} G\,d\mu =  \sum_{z_{1,\ell}\in\range^\ell} \int_\Omega \phi_{z_{1,\ell}}(\w) G(\w, z_{1,\ell}) \, \P(d\w). \] 
Such $\mu$ satisfy $\mu_0\ll\P$ and
so 
\begin{align*}%\label{minimax:page}
	\ratell^\#(g)&=\sup_{\mu:\mu_0\ll\P}\{E^\mu[g]-\ratell(\mu)\}
	\ge\sup_{\mu\in\measures^k}\{E^\mu[g]-\ratell(\mu)\}.
\end{align*}
Abbreviate $A=\ratell^\#(g)$.  
The proof is completed by verifying this statement:  
\beq \text{if $A\ge\sup_{\mu\in\measures^k}\{E^\mu[g]-\ratell(\mu)\}$    
for all $k\ge1$,  then $A\ge\unK_\ell(g)$.} 
\label{mmx-goal}\eeq
 
 Let $\alpha$ denote a generic probability measure  on $\bigom_\ell^2$ with marginals 
$\alpha_1$ and $\alpha_2$  and let $b\bigom_\ell$ denote the space of bounded
measurable functions on $\bigom_\ell$.  
\begin{align}%\label{minimax:page}
	A&\ge\sup_{\mu\in\cM_1^k, \; q: \mu q=\mu}
	\bigl\{E^\mu[g]- H(\mu\times q\,|\,\mu\times\hat p_\ell)\bigr\}\nn\\
		&=\sup_{\alpha :\alpha_1\in\cM_1^{k},\alpha_1=\alpha_2 }\big\{E^{\alpha_1}[g]-H(\alpha\,|\,\alpha_1\times\hat p_\ell)\big\}\nn\\
		%&=\sup_{\alpha:(\alpha_1)_0\ll\P}\inf_h\Big\{E^{\alpha_1}[f]+E^{\alpha_2}[h]-E^{\alpha_1}[h]-H(\alpha\,|\,\alpha_1\times\pr)\Big\}\nn\\
		&= \sup_{ \alpha :\alpha_1\in\cM_1^{k}  }\;\inf_{h\in b\bigom_\ell} \big\{  E^{\alpha_1}[g]+E^{\alpha_2}[h]-E^{\alpha_1}[h]
		-H(\alpha\,\vert\,\alpha_1\times \hat p_\ell)\big\}.\label{mmx-id0}
\end{align}
Let $F$ denote a bounded measurable test function  on $\bigom_\ell^2$.  
$\cM_2^k$ is  the set of probability measures $\alpha$  on $\bigom_\ell^2$
of the form 
\[   \int_{\bigom_\ell^2} F\,d\alpha = \sum_{z\in\range}  \int_{\bigom_\ell}
 \alpha_1(d\wz) \, q(\wz, \Sopr_z\wz) \, F(\wz, \Sopr_z\wz)  \]
 where $\alpha_1\in\cM_1^{k} $ and kernel 
 $q(\wz, \Sopr_z\wz) =q( (\w, z_{1,\ell}), (T_{z_1}\w, z_{2,\ell}z))$ is $\kS_k$-measurable
  as a  function of $\w$ for each fixed $(z_{1,\ell},z)$. 
A measure $\alpha\in\cM_2^k$ is  uniquely represented by a finite sequence 
$(\phi_{i,z_{1,\ell}}, q_{i,z_{1,\ell},z})$ via the identity
\beq   \int_{\bigom_\ell^2} F\,d\alpha = \sum_{i, z_{1,\ell}, z }  
\phi_{i,z_{1,\ell}} q_{i,z_{1,\ell},z}
\int_{A_i}  F( (\w, z_{1,\ell}), (T_{z_1}\w, z_{2,\ell}z))\,\P(d\w)
\label{mmx-id1} \eeq
where $\{A_i\}$ is the finite set of atoms of $\kS_k$ such that $\P(A_i)>0$,   
  $\phi_{i,z_{1,\ell}}$ is the value of $\phi_{z_{1,\ell}}(\w)$ for $\w\in A_i$,  and 
$q_{i,z_{1,\ell},z}$ is the value of   $q(\wz, \Sopr_z\wz)$  for $\w\in A_i$. 
Thus $\cM_2^k$ is in bijective 
correspondence with a   compact  subset of a Euclidean space,  and 
  \eqref{mmx-id1} shows  that via this identification the integral
is continuous in $\alpha$ for any $F$ that is suitably integrable under $\P$. 
Similarly the entropy  
\[ H(\alpha\,\vert\,\alpha_1\times \hat p_\ell)
= \sum_{i, z_{1,\ell}, z }  
\phi_{i,z_{1,\ell}} q_{i,z_{1,\ell},z}   \P(A_i)  \log (\abs{\range}q_{i,z_{1,\ell},z})
\] 
is continuous and convex in $\alpha$.  

Turning our attention back to \eqref{mmx-id0}.
Once we  restrict $\alpha$ to the compact Hausdorff space 
$\cM^k_2$ the expression in braces is 
upper semicontinuous and concave in $\alpha$ and convex in $h$.    
We can apply K\"onig's minimax theorem (\cite{Kas-94} or \cite{Ras-Sep-10-ldp-}), and 
continue as follows:
\begin{align*}
&A\ge  \sup_{ \alpha\in\cM_2^{k}  }\;\inf_{h\in b\bigom_\ell} \big\{  E^{\alpha_1}[g]+E^{\alpha_2}[h]-E^{\alpha_1}[h]
		-H(\alpha\,\vert\,\alpha_1\times \hat p_\ell)\big\}\\
		&= \inf_{h\in b\bigom_\ell} \sup_{ \alpha\in\cM_2^{k}  }  \big\{  E^{\alpha_1}[g]+E^{\alpha_2}[h]-E^{\alpha_1}[h]
		-H(\alpha\,\vert\,\alpha_1\times \hat p_\ell)\big\}\\
&= \inf_{h\in b\bigom_\ell} \sup_{ \alpha\in\cM_2^{k}  }  \;
\sum_{z_{1,\ell}} \int_\Omega \P(d\w) \phi_{z_{1,\ell}}(\w) \\
&\qquad \times \Bigl\{\,  \sum_z q(\wz, \Sopr_z\wz) \bigl(  g(\wz) -h(\wz) 
+h(\Sopr_z\wz) \bigr) 
 -H\bigl(q(\wz, \,\cdot) \,\vert \,  \hat p_\ell(\wz, \,\cdot)\bigr)\Bigr\} \\
&= \inf_{h\in b\bigom_\ell} \sup_{ \alpha\in\cM_2^{k}  }  \;
\sum_{z_{1,\ell}} \int_\Omega \P(d\w) \phi_{z_{1,\ell}}(\w) \\
&\qquad \times \Bigl\{\,   \sum_z q(\wz, \Sopr_z\wz) 
 \E\bigl[  g(\wz) -h(\wz) +h(\Sopr_z\wz) \,\vert\, \kS_k \bigr] 
  -H\bigl(q(\wz, \,\cdot) \,\vert \,  \hat p_\ell(\wz, \,\cdot)\bigr)\Bigr\}\\
&= \inf_{h\in b\bigom_\ell} \sup_{ \mu\in\cM_1^{k}  }  \;
\sum_{z_{1,\ell}} \int_\Omega \P(d\w) \phi_{z_{1,\ell}}(\w)  
 \; \log \sum_z 
\tfrac1{|\range|}e^{\E[ g(\wz)-h(\wz)+h(\Sopr_z\wz)\,|\, \kS_k]}. 
\end{align*}
Above we introduced the densities $\phi_{z_{1,\ell}}(\w)$ and the kernel
$q$ that correspond to $\alpha\in\cM_2^k$,  used $\kS_k$-measurability
to take conditional expectation, and then took supremum over the kernels
$q$ with  the first marginal $\alpha_1=\mu$ fixed.   This supremum is a finite
case of the convex duality of relative entropy:  
\[  \sup_q\Bigl\{ \sum_z  q(z) v(z) - \sum_z q(z) \log \frac{q(z)}{p(z)} \Bigr\}
= \log \sum_x p(x)e^{v(x)}  \]
and  the maximizing  probability is    $q(z) = (\sum_x p(x)e^{v(x)})^{-1} p(z)e^{v(z)}$. 
  In our case  
$v(z)=\E[ g(\wz)-h(\wz)+h(\Sopr_z\wz)\,|\, \kS_k]$ so the maximizing kernel 
is $\kS_k$-measurable in $\w$ and thus admissible under the condition 
$ \alpha\in\cM_2^{k}$.  

Performing the last supremum over $\mu\in\cM_1^{k}$ gives 
 \[A\ge \inf_{h\in b\bigom_\ell} \max_{z_{1,\ell}\in\range^\ell} \; \P\text{-}\esssup_\w \Bigl\{ \log \sum_z 
\tfrac1{|\range|}e^{\E[ g(\wz)-h(\wz)+h(\Sopr_z\wz)\,|\, \kS_k]} \Bigr\}.\]
Consequently  for $\e>0$ and $k\ge1$ there exists a bounded measurable 
function $h_{k,\e}$ on $\bigom_\ell$ such that for all $z_{1,\ell}\in\range^\ell$ and $\P$-a.s.
\begin{align}
\label{F-bound}
A+\log|\range|+\e\ge
 \log\sum_z e^{\E[g(\wz)-h_{k,\e}(\wz)+h_{k,\e}(\Sopr_z\wz)\,|\,\kS_k]}\,.
 \end{align}

For integers $0\le i\le k$ define 
\begin{align}
\label{F}
F^{(i)}_{k,\e}(\wz,z)=\E[h_{k,\e}(\Sopr_z\wz)-h_{k,\e}(\wz)\,|\,\kS_{k-i}].
\end{align} 

We next extract a limit point in $\K$. The proof of the following lemma is given in Appendix \ref{app:pf lm class K}.

\begin{lemma}\label{F-lemma}
%Let $(\Omega,\kS,\P,\{T_z:z\in\gr\})$ be a measurable ergodic system.
%Assume $\range$ is finite and 
Assume $\kS$ is countably generated % and $\Omega$ is compact.
and $A<\infty$. Construct $F_{k,\e}^{(i)}$ as in \eqref{F}.
Fix $\ell\ge0$ and let $g(\w,z_{1,\ell})\in L^1(\P)$ for all $z_{1,\ell}\in\range^\ell$.
Fix $\e>0$. Then, as $k\to\infty$, along a subsequence that works simultaneously  for all $z_{1,\ell}\in\range^\ell$, $z\in\range$, and $i\ge0$, one can 
write \[F_{k,\e}^{(i)}=\Fhat^{(i)}_{k,\e}-R^{(i)}_{k,\e}\] with $\Fhat_{k,\e}^{(i)}(\wz,z)$ converging in weak $L^1(\P)$ %topology
to a limit $\Fhat_\e^{(i)}$ and the error terms $\w\mapsto R^{(i)}_{k,\e}(\w,z_{1,\ell},z)\ge0$ $\kS_{k-i}$-measurable and converging to $0$ $\P$-a.s. 
Furthermore, as $i\to\infty$, $\Fhat_\e^{(i)}$ converges strongly in $L^1(\P)$ to a limit $\Fhat_\e$, \[c(z)=\E[\Fhat_\e(\w,(z,z,\cdots,z),z)]\ge0\] for all $z\in\range$, and
$F_\e(\wz,z)=\Fhat_\e(\wz,z)-c(z_1)$ belongs to class $\K$.
\end{lemma}

\label{convex page}
Fix $i\ge0$ for the moment.   As a uniformly integrable  martingale,
$M_k(\wz)=\E[g(\wz)\,|\,\kS_{k-i}]$
 converges as $k\to\infty$ to   $g(\wz)$, both a.s.\ and  in $L^1(\P)$  
 for all $z_{1,\ell}$ and $z$. 
%Measure theory, Volume 2  By D. H. Fremlin 247C: on any measure space  relatively compact in weak L^1 topo iff unif int

Fix $z_{1,\ell}$ and $z$.
The weak-$L^1(\P)$  closure of the convex hull of $\{M_j+\Fhat^{(i)}_{j,\e}:j\ge k\}$ is equal to its strong
closure (Theorem 3.12 of \cite{Rud-91}) . Since $g(\wz)+\Fhat^{(i)}_\e(\wz,z)$ is in this closure, there exist finite convex combinations  
	\[\Ghat^{(i)}_{k,\e}=\sum_{j\ge k}\alpha_{j,k}(M_j+\Fhat^{(i)}_{j,\e})\]
%with $\{\alpha_{j,k}\in[0,1]:j\ge k\}$ having finite support, $\sum_{j\ge k}\alpha_{j,k}=1$, and 
such that 
\[\E\bigl\lvert g(\wz)+\Fhat^{(i)}_\e(\wz,z)-\Ghat^{(i)}_{k,\e}(\wz,z)\bigr\rvert \le1/k.\]
Along a subsequence (that we again index by $k$) $\Ghat^{(i)}_{k,\e}(\wz,z)$ converges 
$\P$-a.s.\ to  $g(\wz)+\Fhat^{(i)}_\e(\wz,z)$. Consequently also 
	\[G^{(i)}_{k,\e}=\sum_{j\ge k}\alpha_{j,k}(M_j+F^{(i)}_{j,\e})  \; \underset{k\to\infty}\longrightarrow \;
g+\Fhat^{(i)}_\e  \quad \text{$\P$-a.s.} 	\]
Along  a further subsequence this holds simultaneously for all $z_{1,\ell}$ and $z$.  

By \eqref{F-bound} and Jensen's inequality, 
we have for all $z_{1,\ell}\in\range^\ell$ and $\P$-a.s.
\begin{align*}
e^{A+\log|\range|+\e}
&\ge\sum_{z\in\range}\E\Big[e^{\E[g(\wz)-h_{k,\e}(\wz)+h_{k,\e}(\Sopr_z\wz)|\kS_k]}\,\Big|\,\kS_{k-i}\Big]\\
&\ge \sum_{z\in\range}e^{M_k(\wz,z)+F_{k,\e}^{(i)}(\wz,z)}.
\end{align*}
Since this is valid for all $k\ge i$, another application of Jensen's inequality gives
\begin{align*}
e^{A+\log|\range|+\e}\ge \sum_{z\in\range}e^{G_{k,\e}^{(i)}(\wz,z)}.
\end{align*}
Taking $k\to\infty$ implies, for $\P$-a.e.\ $\w$ and all $z_{1,\ell}\in\range^\ell$,
\[A+\e\ge g(\wz)+\log\sum_{z\in\range}\frac1{|\range|}e^{\Fhat^{(i)}_\e(\wz,z)}.\]
Taking $i\to\infty$ implies, for $\P$-a.e.\ $\w$ and all $z_{1,\ell}\in\range^\ell$,
\[A+\e\ge g(\wz)+\log\sum_{z\in\range}\tfrac1{|\range|}e^{\Fhat_\e(\wz,z)}.\]
Since $c(z_1)\ge0$ the above inequality still holds if $\Fhat_\e$ is replaced with $F_\e$.
Thus
\[A+\e\ge \inf_{F\in\K}
\max_{z_{1,\ell}\in\range^\ell}\P\text{-}\esssup_\wz\Big\{g(\wz)+\log\sum_{z\in\range}\tfrac1{|\range|}e^{F(\wz,z)}\Big\}.\]
Taking $\e\to0$ gives  $A\ge\unK_\ell(g)$. \eqref{mmx-goal} is verified 
and thereby the proof of Lemma \ref{lm:iv} is   complete.
\end{proof}

Next for technical purposes a  Fatou-type lemma for $\unK_\ell$.    

\begin{lemma}\label{K-Fatou}
%Let $(\Omega,\kS,\P,\{T_z:z\in\gr\})$ be a measurable ergodic system
%and assume $\range$ is finite. % and $\Omega$ is compact.
%Fix $\ell\ge0$. 
Let 
$g_k(\cdot\,,z_{1,\ell})\underset{k\to\infty}\longrightarrow g(\cdot\,,z_{1,\ell})$  in $L^1(\P)$ for each
 $z_{1,\ell}\in\range^\ell$. 
Then
\begin{align*}%\label{K-Fatou}
\unK_\ell(g)\le\varliminf_{k\to\infty}\unK_\ell(g_k).
\end{align*}
\end{lemma}

\begin{proof}%[Proof of Lemma \ref{K-Fatou}]
We can assume $\varliminf_{k\to\infty}\unK_\ell(g_k)=A<\infty$.
Fix $\e>0$.
There exists a subsequence, denoted again by $g_k$, such that $\unK_\ell(g_k)<A+\e$ for all $k$.
Pick $F_k\in\K$ such that 
	\beq
g_k(\wz)+\log\sum_{z\in\range} \tfrac1{|\range|} e^{F_k(\wz,z)}<A+\e\label{temp-8}\eeq
for all $k$, $z_{1,\ell}\in\range^\ell$, and $\P$-a.e.\ $\w$. 
Out of this we can produce 
an $F\in\K$ such that
\beq
	g(\wz)+\log\sum_{z\in\range} \tfrac1{|\range|} e^{F(\wz,z)}\le A+\e.
\label{temp-9}\eeq 
This implies  $\unK_\ell(g)\le A+\e$ and taking $\e\to0$ finishes the proof.

The construction of $F$ is a simplified version of the argument  to 
realize a limit point  in $\K$ in the proof of  Lemma  \ref{lm:iv}. 
We sketch  the steps.  The reader who aims to master the proof may find it
useful to fill in the details.  
 
For each $k$, $z_{1,\ell}$, and $z$,
	\[F_k(\wz,z)\le A+\e - g_k(\wz) + \log|\range|.\]
Thus $F_k^+$ is uniformly integrable. 
Controlling $F_k^-$ is indirect.  Set 
   $\wz_0=(\w,z_{1,\ell})$, $z_0=z$, $a_i=z_{i-1}$, and $\wz_i=\Sopr_{a_i}\wz_{i-1}$ for $i=1,\dotsc,\ell+1$.  By the    mean-zero property of $F_k$ (part (ii) in Definition \ref{cK-def}), 
 \[\E[F^-_k(\wz,z)]\le\sum_{i=0}^{\ell} \E[F_k^-(\wz_i,a_{i+1})]=\sum_{i=0}^{\ell} \E[F_k^+(\wz_i,a_{i+1})]\]
and so $\E[F_k^-]$ 
is bounded uniformly in $k$. Apply Lemma \ref{lm:raghu} to write $F_k^-=\Ftil_k+R_k$ such that along a subsequence $\Ftil_k$ is uniformly integrable and $R_k\ge0$ converges to $0$ in $\P$-probability, for each $z_{1,\ell}$.
Along a further subsequence $\Fhat_k\equiv F^+_k-\Ftil_k$ converges weakly in $L^1(\P)$ to a limit $\Fhat$ and the limits  $R_k\to 0$ and $g_k\to g$ hold almost surely. 

In \eqref{temp-8} write 
  $F_k=\Fhat_k-R_k$.   As done in the proof of  Lemma  \ref{lm:iv} take almost
  surely convergent 
  convex combinations of $\Fhat_k$, $R_k$ and $g_k$ and substitute these into
\eqref{temp-8}.   Taking the limit now yields \eqref{temp-9}  but with $\Fhat$ in place
of $F$.  
 
Almost sure convergence of convex combinations ensures  
  that
$\Fhat$ satisfies the closed-loop property.  But  it may fail the mean-zero property. To remedy this, 
let $c(z)=\E[\Fhat(\w,(z,z,\dotsc,z),z)]$.  By the weak convergence $c(z)$ is a 
  limit of $\E[\Fhat_k(\w,(z,z,\dotsc,z),z)]$, which is nonnegative due to $R_k\ge 0$ 
and  the mean-zero property for $F_k$.  Since   $c(z)\ge0$,
  \eqref{temp-9}  holds with $F(\wz,z)=\Fhat(\wz,z)-c(z_1)$.  
That $F$ satisfies both the mean-zero and the closed-loop property is verified
with the argument given between equations \eqref{temp-21} and \eqref{temp-22}
in Appendix \ref{app:pf lm class K}.  The point
is that the closed-loop property of $\Fhat$ allows us to define the path integral
$\fhat$ which is used in that argument.   This verifies that
$F\in\K$ and completes the proof.  
 \end{proof}

Next a large deviation lower bound lemma that gives us   (\ref{inequalities1}{\color{my-red}.iii}) and  serves again to prove  Theorem \ref{th:ldp RWRP} below.
 
\begin{lemma}\label{lower-bound-lemma}
%Let $(\Omega,\kS,\P,\{T_z:z\in\gr\})$ be a measurable ergodic system and 
%assume $\range$ is finite. % and $\Omega$ is compact.   Fix $\ell\ge0$ and 
Let   $g(\cdot\, ,z_{1,\ell})\in L^1(\P)$ be  bounded above. 
Then for $\P$-a.e.\ $\w$ 
  \begin{align}\label{Lambda>H*-1}
  \varliminf_{n\to\infty}n^{-1}\log E_0\Big[e^{n R_n^{\oneell}(g)}\Big]
  \ge \sup_\mu\{E^\mu[g]-\ratell(\mu)\}.
  \end{align}
Assume additionally that $\Omega$ is a separable  metric space.
Then for $\P$-a.e.\ $\w$ this lower bound holds
for all   open   $O\subset\measures(\bigom_\ell)$:  
  \begin{align}\label{Lambda>H*-2}
  \varliminf_{n\to\infty}n^{-1}\log E_0\Big[e^{n R_n^{\oneell}(g)}\one\{R_n^{\oneell}\in O\}\Big]
  \ge -\inf_{\mu\in O}\{\ratell(\mu)-E^\mu[g]\}.
  \end{align}
  %insert \one\{R_n\in B\}, change measure, use Jensen, use the ergodic theorem, then a convexity 
  %argument to remove positivity of kernel.
\end{lemma}

\begin{proof}%[Proof of Lemma \ref{lower-bound-lemma}]
This proof proceeds along the familiar lines of Markov chain lower bound arguments 
and we refer to  Section 4 of \cite{Ras-Sep-10-} for some further details.

Switch to the $\bigom_\ell$-valued  Markov chain 
$\wz_k=(T_{X_k}\w,Z_{k+1,k+\ell})$ with  transition kernel $\pell$ defined in \eqref{pell-def}.
Then   $R_n^{\oneell}$ is the 
 position level   empirical measure
	$L_n=n^{-1}\sum_{k=0}^{n-1}\delta_{\wz_k}. $  
%We can treat $R_n^{\oneell}$   as the position  level (level 2) empirical measure 
%  of a %Feller-continuous  Markov chain.  
Denote by $P_\wz$  (with expectation $E_\wz$)  the distribution
 of the Markov chain $(\wz_k)_{k\ge0}$  
with initial state $\wz$.
Starting at $\wz=(\w, z_{1,\ell}) $ is the same as conditioning   our original process
on $Z_{1,\ell}$:  
%\beq
\begin{align*}  
E_0\bigl[ G\bigl( (T_{X_k}\w,\,Z_{k+1,k+\ell})_{0\le k\le n} \bigr) \one\{ Z_{1,\ell}=z_{1,\ell}\}  \bigr]
= \tfrac1{|\range|^\ell} \,E_\wz[G(\wz_0, \dotsc, \wz_n)].  
\end{align*} %\label{two-chains}
%\eeq  
%So for large deviation asymptotics it is immaterial which process is used:   
Consequently   for any $z_{1,\ell}\in\range^\ell$  and  $\P$-almost every $\w$,
with    $\wz=(\w,z_{1,\ell})$, 
\[\varliminf_{n\to\infty}n^{-1}\log E_0\Big[e^{n R_n^{\oneell}(g)}\one\{R_n^{\oneell}\in O\}\Big]\ge
\varliminf_{n\to\infty}n^{-1}\log E_{\wz}\Big[e^{n L_n(g)}\one\{L_n\in O\}\Big].\]

Next we reduce the right-hand sides of \eqref{Lambda>H*-1} and \eqref{Lambda>H*-2}
to nice measures.  A convexity argument shows that 
the supremum/infimum is not altered by restricting it to measures $\mu$ 
with these properties:  
   $\mu_0\ll\P$ and 
there exists   a kernel $q\in \MC(\bigom_\ell)$ such that $\mu q=\mu$, 
$H(\mu\times q\,|\,\mu\times\pell)<\infty$,   $q(\wz,\cdot)$ is supported on  
shifts $\Sopr_z\wz$, and 
  $q(\wz,\Sopr_z\wz)>0$ for all $z\in\range$ and $\mu$-a.e.\ $\wz$. 
We omit this argument.  It can be patterned after the lower bound
	proof of Theorem 3.1 of \cite{Ras-Sep-10-} (page 224).  This step 
 needs the integrability of  $g(\w,z_{1,\ell})$ under $\P$.
  
These properties of $\mu$  imply  the equivalence $\mu_0\sim\P$
and  the  ergodicity 
of the Markov chain $Q_\wz$  with initial state $\wz$ and transition
kernel $q$   (Lemma 4.1 of \cite{Ras-Sep-10-}). 

  Next follows a standard 
change of measure argument.   Let $\sF_n$ be the $\sigma$-algebra generated by $(\wz_0,\dotsc,\wz_n)$. Then
	\begin{align*}
		&n^{-1}\log E_{\wz}\Big[e^{n L_n(g)}\one\{L_n\in O\}\Big]\\
		&\ge n^{-1}\log 
			\frac{E^{Q_\wz}\Big[\Big(\frac{d{Q_\wz}_{|\sF_{n-1}}}{d{P_\wz}_{|\sF_{n-1}}}\Big)^{-1}
				e^{n L_n(g)}\one\{L_n\in O\}\Big]}{Q_\wz\{L_n\in O\}}
			+n^{-1}\log Q_\wz\{L_n\in O\}\\
		%\ge since the Q measure might not be equivalent to P
		&\ge\frac{-n^{-1} E^{Q_{\wz}}\Big[\log\Big(\frac{d{Q_\wz}_{|\sF_{n-1}}}{{dP_\wz}|_{\sF_{n-1}}}\Big)\Big]}{Q_\wz\{L_n\in O\}}
		+\frac{E^{Q_{\wz}}[L_n(g)]}{Q_\wz\{L_n\in O\}}+n^{-1}\log Q_\wz\{L_n\in O\}\\
		&\quad+\frac{n^{-1} E^{Q_\wz}\Big[\log\Big(\frac{{dQ_\wz}_{|\sF_{n-1}}}{d{P_\wz}_{|\sF_{n-1}}}\Big)
				\one\{L_n\notin O\}\Big]}{Q_\wz\{L_n\in O\}}
		-\frac{ E^{Q_\wz}[L_n(g)\one\{L_n\notin O\}]}{Q_\wz\{L_n\in O\}}\\
		&=
		\frac{-n^{-1}H\Big({Q_\wz}_{|\sF_{n-1}}\,\Big|\,{P_\wz}_{|\sF_{n-1}}\Big)}
		{Q_\wz\{L_n\in O\}}+\frac{E^{Q_{\wz}}[L_n(g)]}{Q_\wz\{L_n\in O\}}
		+n^{-1}\log Q_\wz\{L_n\in O\}\\
			&\quad+\frac{n^{-1} E_\wz\Big[
			\frac{d{Q_\wz}_{|\sF_{n-1}}}{d{P_\wz}|_{\sF_{n-1}}}
			\log\Big(\frac{d{Q_\wz}_{|\sF_{n-1}}}{{dP_\wz}_{|\sF_{n-1}}}\Big)
				\one\{L_n\notin O\}\Big]}
			{Q_\wz\{L_n\in O\}}
		-\frac{E^{Q_\wz}[L_n(g)\one\{L_n\notin O\}]}{Q_\wz\{L_n\in O\}}\\
		&\ge
		-\frac{E^{Q_\wz}\Big[n^{-1}\sum_{k=0}^{n-1}F(\wz_k)\Big]}
		{Q_\wz\{L_n\in O\}}+\frac{E^{Q_{\wz}}[L_n(g)]}{Q_\wz\{L_n\in O\}}
		+n^{-1}\log Q_\wz\{L_n\in O\}\\
		&\quad
		-\frac{n^{-1}e^{-1}}{Q_\wz\{L_n\in O\}}
		-\frac{(\sup g) Q_\wz\{L_n\notin O\}}{Q_\wz\{L_n\in O\}},
	\end{align*}
	where we used $x\log x\ge-e^{-1}$ and
		\[F(\wz)=\sum_{z\in\range} q(\wz,\Sopr_z\wz)\,
			\log\frac{q(\wz,\Sopr_z\wz)}{\pell(\wz,\Sopr_z\wz)}\,.\]
	Since $F\ge0$ by Jensen's inequality and $g$ is bounded above, 
	ergodicity gives the limits   for $\mu_0$-a.e.\ $\w$:  
		\[\varliminf_{n\to\infty} n^{-1}\log E_0\Big[e^{n R_n^{\oneell}(g)}\one\{R_n^{\oneell}\in O\}\Big]
		\ge E^\mu[g]-H(\mu\times q\,|\,\mu\times\pell).\]
		(For the details of $Q_\wz\{L_n\in O\}\to1$ see the proof of Lemma 4.2 of \cite{Ras-Sep-10-}.) By $\mu_0\sim\P$   this also holds $\P$-a.s.
%See the proofs of Lemmas 3.2 and 4.2 of \cite{Ras-Sep-10-} for more detailed explanations of this computation.
	\end{proof}

We are ready for the proof of the theorem.

\begin{proof}[Proof  of Theorem \ref{Lambda=Hstar}]
Assume first that  $g\in\Ll$ is bounded above.  Then  
Lemmas \ref{lm-new-upper},   \ref{lm:iv} and   \ref{lower-bound-lemma}   give these  
$\P$-a.s.\  inequalities:  
\[  \overline\pres_\ell(g,\w)\le \unK_\ell(g)  \le \ratell^\#(g)\le  \underline\pres_\ell(g,\w). \]
 Existence of the limit $\pres_\ell(g)$  and 
$\pres_\ell(g)=\unK_\ell(g)=\ratell^\#(g)$ follow.  

Next, consider $g\in\Ll$. 
Lemma \ref{K-Fatou} implies that $\unK_\ell(g)=\sup_c\unK_\ell(\min(g,c))$. 
Existence of the limit $\pres_\ell(\min(g,c))$ combined with Lemma \ref{lm-new-upper}
implies 
	\begin{align*}
\unK_\ell(g)&=\sup_c\unK_\ell(\min(g,c))=\sup_c\pres_\ell(\min(g,c))=\sup_c\underline\pres_\ell(\min(g,c))\\
	&\le\underline\pres_\ell(g)\le\overline\pres_\ell(g)\le \unK_\ell(g) .\end{align*}
Existence of the limit and the equality 
 $\pres_\ell(g)=\unK_\ell(g)$   follow again.  
For the other variational formula write
	\begin{align*}
	\unK_\ell(g)&=\sup_c\unK_\ell(\min(g,c))\\
	&=\sup_c
	\sup_{\mu\in\measures(\bigom_\ell)}\{E^\mu[\min(g,c)]-\ratell(\mu)\}=\ratell^\#(g).\qedhere
\end{align*}
%
%Finally, let $G(\w)=\min_{z_{1,\ell}\in\range^\ell} g(\w,z_{1,\ell})\in L^1(\P)$ and observe that
%$\P\refp=\P$. Thus
%	\[\pres_\ell(g)\ge\pres_0(G)\ge H_{0,\P}^\#(G)\ge \E[G]-H(\P\times\refp\,|\,\P\times\refp)=\E[G]>-\infty.\qedhere\]
 \end{proof}

\section{Large deviations under quenched polymer measures}
\label{sec-ldp-pol}
As before, we continue to assume  that  $\range$ is finite and 
Ê$(\Omega,\kS,\P,\{T_z:z\in\gr\})$Ê
is a measurable ergodic  system where $\gr$ is the additive subgroup 
of $\Z^d$ generated by $\range$. 
Now  assume additionally that   $\Omega$ is a separable metric space and $\kS$ is 
its Borel $\sigma$-algebra.   

Since our limiting logarithmic moment generating functions $\pres_\ell(g)$ are defined only $\P$-a.s.\ we need
a separable function space  that generates the weak topology of probability measures.
%When  $\Omega$, and hence also $\bigom_\ell$,  is a separable metric space, we
Give $\bigom_\ell$   a totally bounded metric and let $\unif(\bigom_\ell)$ be the space of  uniformly
continuous functions under this metric.  These functions are bounded. 
 The space  $\unif(\bigom_\ell)$ is separable 
under the supremum norm and generates the same
topology on $\measures(\bigom_\ell)$ as does the space of bounded continuous functions.

Given a real-valued function $V$ on $\bigom_\ell$ define the quenched
 polymer measures %on paths by 	
	\[Q^{V,\w}_{n,0}(A) 
	=\frac1{Z_{n,0}^{V,\w}}E_0
\bigl[e^{-\sum_{k=0}^{n-1}V(T_{X_k}\w, Z_{k+1,k+\ell})}\one_A(\w, X_{0,\infty})\bigr],\]
where $A$ is an event on environments and paths and 
 	\[Z_{n,0}^{V,\w}=E_0\Big[e^{-\sum_{k=0}^{n-1}V(T_{X_k}\w, Z_{k+1,k+\ell})}\Big].\]
Theorem \ref{Lambda=Hstar} gives the a.s.\ limit  
$\pres_\ell(-V)=\lim n^{-1}\log Z_{n,0}^{V,\w}$. Next  we prove 
a LDP  for the quenched distributions
	$Q^{V,\w}_{n,0}\{R_n^{\oneell}\in \cdot\,\}$  of  the empirical  measure 
%	$Q^{V,\w}_{n}\{R_n^{\oneinfty}\in \cdot\,\}$  of  the empirical process 
%\beq R_n^{\oneinfty}=n^{-1}\sum_{k=0}^{n-1}\delta_{T_{X_k}\w,\,Z_{k+1,\infty}}. \label{empproc} 
%\eeq    
%These distributions  are   probability measures on  
%  $\measures(\Omega\times\range^\N)$,   the space of {\color{red}Borel} probability measures on
%  $\Omega\times\range^\N$ endowed with the weak topology generated by bounded
%  continuous functions.   This   process level LDP is proved 
%via a projective limit argument
% from  large deviation theory.   
%  The intermediate steps are   
%multivariate quenched level 2  LDPs.   
%define  the multivariate empirical measure 
	\begin{align*}
	R_n^{\oneell}=n^{-1}\sum_{k=0}^{n-1}\delta_{T_{X_k}\w,\,Z_{k+1,k+\ell}}.%\label{level2 empproc}
	\end{align*}
%$R_n^0$ is the empirical measure of the environment process $T_{X_k}\w$ without any steps.  
%$R_n^{\oneell}$  lives on  the space  $\bigom_\ell=\Omega\times\range^\ell$  whose generic element is   denoted by $\wz=(\w,z_{1,\ell})$.   
	 	
%Next  we state   the multivariate level 2 quenched LDP under $Q^{V,\w}_{n}$.

\begin{theorem}\label{th:ldp RWRP}
%	Let $(\Omega,\kS,\P,\{T_z:z\in\gr\})$ be a measurable ergodic system.  Assume $\range$ is finite and $\Omega$ is metric and separable.
Fix $\ell\ge 0$.  	Let  $V$ be a measurable function  
on $\bigom_\ell$,  $V\in \Ll$ and $\pres_\ell(-V)<\infty$. 
%	 \begin{itemize}
%	 \item[{\rm (a)}] 
   Then for $\P$-a.e.\ $\w$
	the weak large deviation principle holds for the sequence of probability distributions 
	$Q_{n,0}^{V,\w}\{R_n^{\oneell}\in\cdot\,\}$ on $\measures(\bigom_\ell)$ with convex rate function 
\beq	 \label{lower-rate-rwrerp1}		 
		\Iqell^V(\mu)=\sup_{g\in\unif(\bigom_\ell)}\{E^\mu[g]-\pres_\ell(g-V)\}+\pres_\ell(-V).
		\eeq
	Rate $\Iqell^V$ is also equal to the lower semicontinuous regularization of 
\beq	 \label{lower-rate-rwrerp2}
	\ratell^V(\mu)=\inf_{c<0}\{\ratell(\mu)+E^{\mu}[\max(V,c)]+\pres_\ell(-V)\}.
 \eeq
% 	\item[{\rm(b)}] For $\P$-a.e.\ $\w$, the weak large deviation principle holds for %the laws 
%	 $Q_{n,0}^{V,\w}\{R_n^{\oneinfty}\in\cdot\,\}$
%	with convex rate function 
%		\[\Iqr^V(\mu)=\sup_\ell \Iqell^V(\mu|_{\bigom_\ell}).\]
%	\end{itemize}
\end{theorem}

\begin{proof}[Proof of Theorem \ref{th:ldp RWRP}] 
We show an upper bound for compact sets $A$:
\beq
\varlimsup_{n\to\infty} n^{-1} \log Q_{n,0}^{V,\w}\{R_n^{\oneell}\in A\} 
 {\le} 
-\inf_{\mu\in A}\Iqell^V(\mu), 
\label{QUB}\eeq
a lower bound for open sets $G$: 
\beq
\varliminf_{n\to\infty} n^{-1} \log Q_{n,0}^{V,\w}\{R_n^{\oneell}\in G\} \ge -\inf_{\mu\in G}\ratell^V(\mu), 
\label{QLB}\eeq
and then match the rates.  
%The lower bound in (\ref{inequalities2}{\color{my-red}.i}) will again follow from a special
%case of  Lemma \ref{lower-bound-lemma} and the upper bound (\ref{inequalities2}{\color{my-red}.ii}) 
% is an instance of a general upper bound. The upper and lower bounds then match because $\pres_\ell(g)=\ratell^\#(g)$ implies that
% $\Iqell^V$ is the lower semicontinuous regularization of $\ratell^V$.

By Theorem \ref{Lambda=Hstar}  and separability of $\unif(\bigom_\ell)$ 
we have $\P$-a.s.\  these finite limits for all $g\in\unif(\bigom_\ell)$: 
\[\lim_{n\to\infty}  n^{-1}\log E^{Q^{V,\w}_{n,0}} [e^{nR^{\oneell}_n(g)} ]  =  \pres_\ell(g-V) - \pres_\ell(-V).\]
%This holds simultaneously for all $g\in because this space is separable.
%[we take a countable dense set in the space of bounded
%unif cont functions under a totally bounded metric on \bigom_\ell;  this space of functions is
%separable and is rich enough to generate the weak top of M_1.  I don't think we need to say this
%stuff in [ ]'s.   ]
 \eqref{QUB} follows  by a general convex duality 
 argument (see Theorem 4.5.3 in \cite{Dem-Zei-98} or  Theorem 5.24 in \cite{Ras-Sep-10-ldp-}).     
 
Lower bound \eqref{QLB} follows from   Lemma \ref{lower-bound-lemma} and a truncation: 
for   $-\infty<c<0$ 
\begin{align*}
n^{-1}\log Q_{n,0}^{V,\w}\{R_n^{\oneell}\in O\} &\ge 
n^{-1}\log E_0\bigl[e^{-n R_n^{\oneell}(\max(V,c))}\one\{R_n^{\oneell}\in O\}\bigr]\\
&\qquad - n^{-1}\log E_0\bigl[e^{-n R_n^{\oneell}(V)}\bigr]. 
\end{align*}  
\eqref{QLB}  continues
   to hold if $\ratell^V$ is replaced with its lower semicontinuous regularization
 $\ratell^{V, **}(\mu)=\sup_B \inf_{\nu\in B}\ratell^V(\nu)$   where the supremum is over
 open neighborhoods $B$ of $\mu$.   
 
   Theorem \ref{Lambda=Hstar} implies that for $g\in\unif(\bigom_\ell)$
	\begin{align*}
	\pres_\ell(g-V)-\pres_\ell(-V)
	&=\sup_{\mu\in\measures(\bigom_\ell),\,c>0}\{ E^\mu[\min(g-V,c)] - \ratell(\mu)-\pres_\ell(-V)\}\\
	&=\sup_{\mu\in\measures(\bigom_\ell),\,c<0}\{ E^\mu[g-\max(V,c)] - \ratell(\mu)-\pres_\ell(-V)\}\\
	&=\sup_{\mu\in\measures(\bigom_\ell)}\{E^\mu[g]-\ratell^V(\mu)\}. 
	\end{align*}
Another convex duality  gives $\Iqell^V(\mu)=\ratell^{V, **}(\mu)$ because
the  lower semicontinuous regularization
 $\ratell^{V, **}$ is also equal to the double convex dual of  
 $\ratell^{V}$.  
%the equality of \eqref{lower-rate-rwrerp1} and \eqref{lower-rate-rwrerp2}. 
%Part (a) is proved. Part (b) follows then from a projective limit step.
%This step is coded for example as Theorem~4.6.1 in \cite{Dem-Zei-98}.
\end{proof}

Next we record the LDP for 
  the quenched distributions of the empirical process $R_n^{\oneinfty}=n^{-1}\sum_{k=0}^{n-1}\delta_{T_{X_k}\w,\,Z_{k+1,\infty}}$. 
  
 \begin{theorem}  \label{lev3thm} 	Let  $V$ be a measurable function  
on some $\bigom_{\ell_0}$ with  $V\in \Ll$ and $\pres_{\ell_0}(-V)<\infty$. 
  Then for $\P$-a.e.\ $\w$  
	the weak large deviation principle  holds for the sequence of probability distributions 
	%$P_\wz\{L_n\in\cdot\,\}$ and  
	$Q^{V,\w}_{n,0}\{ R_n^{\oneinfty} \in\cdot\,\}$   on  $\measures(\Omega\times\range^\N)$  with convex rate function $\Iqr^V(\mu)=\sup_{\ell\ge \ell_0}  \Iqell^V(\mu|_{\bigom_\ell})$.
\end{theorem}
%OLD PROOF
%\begin{proof}  This comes from a projective limit.  Formula \eqref{lower-rate-rwrerp1} 
%shows that  $\Iqell^V(\mu\circ \gamma_{\ell+1,\ell}^{-1})\le \Iqella^V(\mu)$ for 
%$\mu\in\measures(\bigom_{\ell+1})$ where $\gamma_{\ell+1,\ell}: \bigom_{\ell+1}
%\to \bigom_{\ell}$ is the natural projection.  Under this condition a projective limit
%argument gives   the weak level 3 LDP.  This is not what is usually proved in textbooks,
%see  e.g.\ Theorem~4.6.1 in \cite{Dem-Zei-98},
%but it is elementary to  verify.
% \end{proof}
 \begin{proof}  This comes from a projective limit.  Formula \eqref{lower-rate-rwrerp1} 
shows that  $\Iqell^V(\mu\circ \gamma_{\ell+1,\ell}^{-1})\le \Iqella^V(\mu)$ for 
$\mu\in\measures(\bigom_{\ell+1})$ where $\gamma_{\ell+1,\ell}: \bigom_{\ell+1}
\to \bigom_{\ell}$ is the natural projection. Since  weak topology of 
$\measures(\Omega\times\range^\N)$ can be generated by uniformly
continuous functions, a base for the topology can be created from 
inverse images of open sets from the spaces $\measures(\bigom_{\ell})$. 
Apply Theorem \ref{proj-ldp-thm}. \end{proof}

In one of the most basic situations, namely for 
  strictly directed walks in i.i.d.\ environments, 
 we can upgrade the weak LDPs into
full LDPs. This means that the upper bound is valid for all closed sets.
 Strictly directed means that there is a vector $\uhat\in\R^d$ such
that $z\cdot\uhat>0$ for all $z\in\range$.  Equivalently,  $0$ does not lie
in the convex hull of $\range$.  

Here is the setting.  Let $\Gamma$ be a Polish space.  Set 
  $\Omega=\Gamma^{\Z^d}$ 
with generic elements $\w=(\w_x)_{x\in\Z^d}$ and shift maps 
$(T_x\w)_y=\w_{x+y}$. Assume 
that  the coordinates $\{\w_x\}$ are i.i.d.\ under $\P$.

\begin{theorem}   As described  above, let   $\P$ be an i.i.d.\ product
measure on a Polish product space $\Omega$. Assume that $0$ does not lie
in the convex hull of $\range$.    Let  $V$ be a measurable function  
on some $\bigom_{\ell}$,  $V\in \Ll$ and  assume that  
$\pres_{\ell}(-\beta V)<\infty$ for some $\beta>1$. 
  Then for $\P$-a.e.\ $\w$  
	the full LDP  holds for the sequence of probability distributions 
 	$Q^{V,\w}_{n,0}\{ R_n^{\oneinfty} \in\cdot\,\}$ 
  on  $\measures(\Omega\times\range^\N)$ 
  with convex rate function $\Iqr^V$ described in Theorem \ref{lev3thm}. 
\label{iid-lev3-thm}\end{theorem} 
 
\begin{proof}  $\pres_\ell(V)<\infty$ by Jensen's inequality. Due to Theorem \ref{lev3thm} it suffices to show that  the distributions  
$Q^{V,\w}_{n,0}\{ R_n^{\oneinfty} \in\cdot\,\}$  are exponentially tight
for $\P$-a.e.\ $\w$.  Suppose we can show that 
\beq 
\text{ distributions $P_0 \{ R_n^{\oneinfty} \in\cdot\,\}$  are exponentially tight
for $\P$-a.e.\ $\w$. } \label{iid-5} \eeq
 From the lower bound in \eqref{La-lb} and the hypotheses on $V$ we have  constants $0<c_0, c_1<\infty$ such that,
 for $\P$-a.e.\ $\w$, 
\[
E_0[e^{-nR_n^{\oneell}(V)}] \ge e^{-c_0n}  \quad\text{and}\quad  
E_0[e^{-nR_n^{\oneell}(\beta V)}] \le e^{c_1\beta n}   \]
for large enough $n$.  Fix  $\w$ so that these bounds and \eqref{iid-5} hold. 
Given $c<\infty$, pick a compact $A\subset\measures(\Omega\times\range^\N)$ 
such that $P_0 \{ R_n^{\oneinfty} \in A^c\} \le e^{-\beta(c_0+c_1+c)n/(\beta-1)}$
for large $n$.  
Then 
\begin{align*}
&Q^{V,\w}_{n,0}\{ R_n^{\oneinfty} \in A^c\}
\le  E_0[e^{-nR_n^{\oneell}(V)}]^{-1} E_0[e^{-nR_n^{\oneell}(\beta V)}]^{\beta^{-1}}
P_0 \{ R_n^{\oneinfty} \in A^c\}^{1- \beta^{-1}} \le e^{-cn}. 
\end{align*}

Thus it suffices to check \eqref{iid-5}.  Next observe from 
\[
\P\{ \w:  P_0( R_n^{\oneinfty}   \in A^c) \ge e^{-cn}\} \le  e^{cn}
  \bar P( R_n^{\oneinfty}  \in A^c) \]
and the Borel-Cantelli lemma that we only need exponential tightness under
the averaged measure $\bar P=\P\otimes P_0$.  As the last reduction, 
note that by the compactness of $\range^\N$ it is enough to have the
exponential tightness of the $\bar P$-distributions of 
$R_n^0=n^{-1}\sum_{k=0}^{n-1}\delta_{T_{X_k}\w}$.

The exponential tightness that is part of 
Sanov's theorem   gives 
compact sets $\{U_{m,x}: m\in\N, x\in\Z^d\}$ in the state space $\Gamma$ 
of the $\w_x$ such that 
\[   \P\Bigl\{  n^{-1}\sum_{k=0}^{n-1} \one_{U_{m,x}^c}(\w_{y_k}) > e^{-m-\abs{x}} \Bigr\} 
\le e^{-n(m+\abs{x})}.  \] 
Here $\{y_k\}$ are any distinct sites.  
Define 
\[  H_m =\{ Q\in\measures(\Omega): \forall x\in\Z^d \;  Q\{\w: \w_x\notin U_{m,x} \} \le 
e^{-(m+\abs{x})} \}  \]
and compact sets
\[  K_b = \bigcap_{m\ge \ell(b)}  H_m \]
where $\ell=\ell(b)$ is chosen for $b\in\N$  so that 
\[  \sum_{m\ge \ell-b}\sum_x  e^{-(m+\abs{x})}  \le 1.  \]
Now 
\begin{align*}
\bar P( R_n^0 \in K_b^c) &\le \sum_{m\ge \ell(b)} \bar P( R_n^0 \in H_m^c) 
\le \sum_{m\ge \ell(b)}  \sum_x \bar P\bigl( R_n^0\{\w_x\notin U_{m,x} \} > e^{-m-\abs{x}} \bigr)\\
&\le \sum_{m\ge \ell(b)}  \sum_x  
 \bar P\Bigl\{  n^{-1}\sum_{k=0}^{n-1} \one_{U_{m,x}^c}(\w_{x+X_k}) > e^{-m-\abs{x}} \Bigr\} 
 \le e^{-bn}.  
\end{align*}
The crucial point used above was that 
under the assumption on $\range$ 
the points $\{X_n\}$ of the walk are distinct (Corollary \ref{cor:loops}), 
and so the variables  $\{\w_{x+X_n}\}$ are i.i.d.\ under $\bar P$. 
This gives the exponential tightness of the $\bar P$-distributions of 
$R_n^0=n^{-1}\sum_{k=0}^{n-1}\delta_{T_{X_k}\w}$. 
\end{proof}

%suppose \w_x  IID and the walk is space-time.   then  \w_{x_k}  are  IID under the averaged measure.
%Consequently the averaged laws of   of the marginal of the empirical measure, that is,
%L_n = n^{-1}\sum_{k=0}^{n-1} \delta_{\w_{x_k}}, are tight, because it's Sanov.
%But then also quenched exp tightness follows from Chebyshev:
%\P\{    P^\w(   L_n \notin K )  \ge  e^{-nc}    \}   \le   e^{nc}   P(   L_n \notin K )  \le   e^{nc}  e^{-na}
%and now just pick K so that  a can be made as large as needed.
%
%Assuming \La(-V) finite and \La(-2V) finite (does the factor 2 need a separate assumption??) 
%%maybe in the theorem. but not in practice. V in L or L^p implies the same for 2V.
%and also, if V is bounded below, we can shift to have V\ge0 and then 2V\ge V.
%then 
%Q( R_n \in A^c )  =    \frac{  E^\w( e^{nR_n(-V)} ,  R_n \in A^c )  }  {  E^\w( e^{nR_n(-V)}  )  }  
%\le     \frac{  [ E^\w( e^{nR_n(-2V)}  ]^{1/2}    [  P^\w( R_n \in A^c ) ]^{1/2} }  {  E^\w( e^{nR_n(-V)}  )  }     
%now  the  E^\w( e^{nR_n(-V)}  )  type things are bdd above and below by exp(n) - type things, 
%while the probability satisfies exp tightness by what was said before.  
%So the Q-ldp for also for Polish \Omega under strictly directed IID. 

\begin{remark} 
For exponential tightness the theorem above is in some sense best possible.  
Theorem \ref{iid-lev3-thm} can fail if $0$ lies in the convex hull of $\range$. Then a loop is 
possible (Corollary \ref{cor:loops}).
Suppose the distribution of $\w_0$ is not supported on any compact set.
Then, given any  compact set $U$  in $\Gamma$, wait until the walk
finds an environment $\w_x\notin U$, and then forever after  execute a loop at $x$. 
%This way  the empirical measure will give at least a fixed 
%positive probability  to $U^c$.  
%However, we do not know if the LDP of Theorem \ref{th:ldp RWRP} holds in general  for
%all closed sets with a rate function that does not have compact level sets.  
 \end{remark} 
% \begin{remark}
%	Then, by Lemma 6.1 of \cite{Ras-Sep-10-} we have $\sup_\ell\Iqell(\mu|_{\bigom_\ell})=\Iqr(\mu)$ and thus
%	if $V$ is bounded and continuous, then 
%			\[\Iqr^V(\mu)=\Iqr(\mu)-E^{\mu}[V]+\Lambda_\ell(V).\] 
%\end{remark}

%As a corollary of Theorem \ref{th:ldp RWRP} we state the quenched LDP for the position of the walk. 
%The limiting logarithmic moment generating function is defined by
%\beq
%\lambda^V(t)= \lim_{n\to\infty} \frac1n \log Q_{n,0}^{V,\w}[e^{\,t\cdot X_n}], \quad
%t\in\R^d
%\label{level1lmgf|RWRP} \eeq
%assuming that the limit exists, and its convex conjugate by 
%\[ (\lambda^V)^*(\zeta)=\sup_{t\in\R^d}\{\zeta\cdot t-\lambda^V(t)\},\quad \zeta\in\R^d. \]  

%\begin{theorem}\label{level1 RWRP}   
%	Let $(\Omega,\kS,\P,\{T_z:z\in\gr\})$ be a measurable ergodic system.
%	 Assume $\range$ is finite and $\Omega$ is compact.
%	 Assume $V(\w,z_{1,\ell})\in\Ll$ for each $z_{1,\ell}\in\range$.
%	 Then, for $\P$-a.e.\ $\w$,  the limit in \eqref{level1lmgf|RWRP}
%exists for all $t\in\R^d$  and equals $\pres_\ell(-V+t\cdot z_1)-\pres_\ell(-V)$.
% The large deviation principle holds 
%for the distributions  $Q_{n,0}^{V,\w}\{n^{-1}X_n\in\cdot\}$
%	with convex rate function  $I_{q,1}^V=(\lambda^V)^*$. 
%\end{theorem}

%Of course, $I_{q,1}^V$ is a contraction of $\Iqr^V$ (or $\Iqell^V$) and, by  Theorem \ref{Lambda=Hstar}, $\lambda^V(t)$ has two variational formulas.

\section{Large deviations for random walk in random environment}
This final section before the appendices is a remark about adapting the 
results of Section \ref{sec-ldp-pol} to RWRE  
 %random walk in random environment.  
  described in Example \ref{ex:RWRE}.   
  Continue with the assumptions on Ê$(\Omega,\kS,\P,\{T_z:z\in\gr\})$Ê
from Section \ref{sec-ldp-pol}.
  Fix any $\ell\ge 1$ and 
let  $V(\w,z_{1,\ell})=-\log\pi_{0,z_1}(\w)$ to put RWRE in the polymer framework.
Then $\pres_1(-V)=-\log\abs{\range}$.
The necessary  assumption is now 
\beq   \text{  $\abs{\log\pi_{0,z}}\in\Ll$ for each $z\in\range$. }  \label{rwreL}\eeq
The commonly used RWRE assumption of
{\sl uniform ellipticity}, namely the existence of $\kappa>0$ such that
 $\P\{\pi_{0,z}\ge\kappa\}=1$ for  
$z\in\range$,  
  implies \eqref{rwreL}. 

Under assumption \eqref{rwreL}  Theorems  \ref{th:ldp RWRP} and 
\ref{lev3thm} are valid for RWRE and give quenched weak LDPs for the distributions   
$P^\w_0\{ R_n^{\oneell} \in\cdot\,\}$  and 
	$P^\w_0\{ R_n^{\oneinfty} \in\cdot\,\}$.   Note
 though that for $\ell\ge 2$,  $Q_{n,0}^{V,\w}\{R_n^{\oneell}\in B\}$  is  not exactly equal to 
$P^\w_0\{ R_n^{\oneell} \in B\}$  because under  $Q_{n,0}^{V,\w}$ steps 
$Z_k$ for $k>n$ are taken from kernel $\refp$.  This difference vanishes in the limit
due to $\log\pi_{0,z}(\w)\in L^1(\P)$. 
These  LDPs take care  of  cases
of RWRE not covered by \cite{Ras-Sep-10-}, namely those walks for which
$0$ does not lie in the relative interior of the convex hull $\Uset$  of $\range$.

 For RWRE 
the rate function  $\Iqell^V$ in Theorem \ref{th:ldp RWRP}   can be expressed 
directly as the lower semicontinuous regularization of 
an entropy. Indeed, let $\overline V(\w,z_{1,\ell})=-\log\pi_{0,z_\ell}(T_{x_{\ell-1}}\w)$. The difference between using potential $\overline V$ and potential $V$
is only in finitely many terms in the exponent. Thus
%$F(\wz,z)+\log\pi_{0,z_\ell}(T_{x_{\ell-1}}\w) - \log\pi_{0,z_1}(\w)$ is in class $\K$ if, and only if, $F\in\K$,
%we have that $\pres_\ell(g-V)=\unK_\ell(g-V)=\unK_\ell(g-\overline V)=\pres_\ell(g-\overline V)$ for all $g\in\unif(\bigom_\ell)$. 
$\pres_\ell(g-V)=\pres_\ell(g-\overline V)$ for all $g\in\unif(\bigom_\ell)$. 
Then \eqref{lower-rate-rwrerp1}
shows that  $\Iqell^V=\Iqell^{\overline V}$. The latter rate is the lower semicontinuous regularization of $\ratell^{\overline V}$ in \eqref{lower-rate-rwrerp2},
which itself equals  $\ratell$ from \eqref{Helldef} with $\hat p_\ell$ replaced with the kernel $\pr(\wz,\Sopr_z\wz)=\pi_{0,z}(T_{x_\ell}\w)$ of the Markov chain 
$(T_{X_k}\w,Z_{k+1,k+\ell})$ under $P^\w_0$. 
%Indeed,  in \eqref{Helldef} replace $\hat p_\ell$ with  the kernel  
%$\pr(\wz,\Sopr_z\wz)=\pi_{0,z}(T_{x_\ell}\w)$ of the Markov chain 
%$(T_{X_k}\w,Z_{k+1,k+\ell})$ under $P^\w_0$. This introduces in \eqref{lower-rate-rwrerp2}
%the term $\log|\range|+\int \log\pi_{0,z_\ell}(T_{x_{\ell-1}}\w)\,\mu(d\wz)$.
%Since $F(\wz,z)= \log\pi_{0,z}(T_{x_\ell}\w) - \log\pi_{0,z_1}(\w)$ is in class $\K$, one can show that
%this extra term equals $\log|\range|-E^\mu[V]$, which cancels out in \eqref{lower-rate-rwrerp2} leaving just an entropy, but
%now relative to the new kernel $\pr$.
%Can also see this directly, because using $V$ in \eqref{lower-rate-rwrerp2} directly gives $\ratell(\mu)$ plus the difference
%\int \log\pi_{0,z_\ell}(T_{x_{\ell-1}}\w) d\mu(\wz) - \int\log\pi_{z_1}(\w) d\mu(\wz) which vanishes if $\mu$ has finite entropy 
%all limits of R_n must clearly satisfy this, and this is a closed set. so measures outside this set have infinite rate.
%also, this just says that \mu has equal marginals: \int F(\w,z_1) \mu(d\wz) = \int F(T_{x_i}\w,z_{i+1}) \mu(d\wz) for i=0,...,\ell-1. if not, var characterization shows \mu has
%infinite entropy
By Lemma 6.1 of \cite{Ras-Sep-10-} same is true of the level 3 rate
$\Iqr^V$ under the additional assumption that $\Omega$ is a compact space. 
We refer to \cite{Ras-Sep-10-} for this and 
some  other  properties of  $\Iqr^V$.

If $\Omega$ is compact, these weak LDPs are of course
  full LDPs, that is, the upper bound holds
 for all closed sets.  For  RWRE with finite $\range$ the natural canonical 
 choice of $\Omega$ is compact:  in   the setting of Example \ref{ex:RWRE} 
take  $\Omega=\cP^{\Z^d}$ with generic elements $\w=(\w_x)_{x\in\Z^d}$
and  
$p(\w)=\w_0$ projection at the origin.  

If $\Omega$ is compact  we can project
the LDP of Theorem  \ref{th:ldp RWRP}  to the level of the walk to obtain the following statements.
The limiting logarithmic moment generating function 
\beq
\lambda(t)= \lim_{n\to\infty} \frac1n \log E_0^\w[e^{\,t\cdot X_n}], \quad
t\in\R^d
\label{level1lmgf} \eeq
  exists a.s.  Its convex conjugate  
\[ \lambda^ *(\zeta)=\sup_{t\in\R^d}\{\zeta\cdot t-\lambda(t)\},\quad \zeta\in\R^d,  \]
is the rate function for the LDP of the 
distributions  $P_0^\w\{n^{-1}X_n\in\cdot\}$ on $\R^d$.  
For walks without ellipticity, in particular for walks with $0\notin\Uset$,
even this quenched position-level LDP is new.  It has been proved in the past only in a
neighborhood of the limiting velocity \cite{Yil-09-aop}.  

\bigskip
 
%%      ---------------------------------------------------------------------
%%      ------------------------- APPENDIX (OPTIONAL) -----------------------
%%      ---------------------------------------------------------------------
        
%%      If you have one appendix, uncomment the line \appendix and add
%%      a \section{ *** APPENDIX TITLE ***}. If you have more than
%%      one, uncomment the line \appendices and add a \section{ ***
%%      APPENDIX TITLE ***} command for each appendix title.

\appendix
\appendices

%%      Type body of appendix/-ices here.
%It remains to prove the remaining lemmas from Section \ref{} in the appendices.  
In the following appendices we invoke the ergodic theorem a few times.
By that we mean the multidimensional ergodic theorem; see for example Theorem 14.A8 in \cite{Geo-88}.

\section{Some auxiliary lemmas } 
\label{app-pressure} 

In this appendix   $\range$ is a finite subset of $\Z^d$, 
 $\gr$   the additive subgroup of $\Z^d$ generated by  $\range$, and $\Uset$  the convex hull of $\range$  in $\R^d$. 

\begin{lemma}\label{lm-convex}   Let $\xi\in\Q^d\cap\Uset$.
Then there exist rational  coefficients $\alpha_z\ge 0$
such that  $\sum_{z\in\range} \alpha_z=1$ and  $\xi=\sum_{z\in\range} \alpha_z z$. 
\end{lemma}

\begin{proof}
Suppose first that $\range=\{\zhat_0,\dotsc, \zhat_n\}$ for   affinely independent points 
$\zhat_0,\dotsc, \zhat_n$. This means that the vectors $\zhat_1-\zhat_0,\dotsc, \zhat_n-\zhat_0$ are 
linearly independent in $\R^d$, and then necessarily $n\le d$.  
Augment this set to a basis $\{b_1=\zhat_1-\zhat_0,\dotsc, b_n=\zhat_n-\zhat_0, b_{n+1},\dotsc, b_d\}$ of $\R^d$ 
where $b_{n+1},\dotsc, b_d$ are  also integer vectors (for example, by including a suitable 
set of  $d-n$  standard basis vectors).   Let $A$ be the unique invertible linear transformation 
such that $Ab_i=e_i$ for $1\le i\le d$.  In the standard basis the matrix of $A$ is the inverse of the matrix
$B=[b_1,\dotsc,b_d]$, hence this matrix has rational entries. %as can be seen from $B^{-1}= (\det B)^{-1}B^T$. 

Now let $\xi=\sum_{i=0}^n \alpha_i \zhat_i$ be a representation of $\xi$ as a  convex 
combination of $\zhat_0,\dotsc,\zhat_n$.    Then   $\xi-\zhat_0=\sum_{i=1}^n \alpha_i (\zhat_i-\zhat_0)$,
and after an application of $A$,  $A\xi-A\zhat_0=\sum_{i=1}^n \alpha_i e_i$.  The vector on the
left has rational coordinates by the assumptions and by what was just said about $A$. 
The vector on the right is $[\alpha_1,\dotsc, \alpha_n,0,\dotsc,0]^T$.  Hence the coefficients 
$\alpha_1,\dotsc, \alpha_n$ are rational, and so is also $\alpha_0=1-\sum_{i=1}^n \alpha_i$. 

Now consider the case of a general $\range$.  
By Carath\'eodory's theorem,  every point in the convex hull of $\range$ is a convex
combination of $d+1$ or fewer affinely independent points of $\range$ \cite[Corollary~17.1.1]{Roc-70}. 
Thus the argument given above covers the general case. 
\end{proof}

The next  simple corollary characterizes the existence of a loop.

\begin{corollary}\label{cor:loops}
The existence of a loop (i.e.\ $z_{1,m}\in\range^m$ with
$z_1+\cdots+z_m=0$) is equivalent to $0\in\Uset$.
\end{corollary}

%\begin{proof}
%If a loop exists, then $0$ is written as a convex combination of elements of $\range$ and thus belongs to $\Uset$.   On the other hand, if $0\in\Uset$ then it can be written as a rational convex combination of elements of $\range$ and thus a loop exists.
%\end{proof}

This corollary expresses  the irreducibility assumption used in \cite{Ras-Sep-10-} in
terms of  the convex hull of $\range$.  

\begin{corollary}\label{cor:nestling}
  There is a  path from $0$ to each $y\in\gr$ with steps from $\range$ 
if and only if    $0$ is in the relative interior of $\Uset$.   
 \end{corollary}

\begin{proof}
Each $y\in\gr$  is reachable from $0$ if and only if  $-x$ is reachable from $0$ for each 
$x\in\range$.  This   is equivalent to the existence of an identity 
 $0=x_1+\dotsm+x_m$ where each $x_i$ is in $\range$ and 
each $z\in\range$ appears
 at least once among the $x_i$'s.   Equivalently,   we can write  $0$   as a convex 
combination of   $\range$ so that each $z\in\range$ has a  positive 
rational coefficient.   Using Lemma \ref{lm-convex}, this in turn is equivalent to the following statement:
for each $z\in\range$,  $-\e z\in\Uset$ for small enough $\e>0$.  
By Theorem.~6.4 in \cite{Roc-70} this is the same as $0\in\ri\Uset$. 
 \end{proof}

% Earlier proof:  Suppose  $0\in \inter \Uset$ and pick any $z\in\range$. For some $t\in(0,1)\cap\Q$ the point $\xi=-(t^{-1}-1)z$ lies also in  $\inter \Uset$. Then   $0=t\xi+(1-t)z$.  Apply Lemma \ref{lm-convex} to $\xi$ to represent  $0$   as a rational convex combination   that gives $z$ positive weight. Combining such expressions we can write  $0$   as a convex  combination of   $\range$ so that each $z\in\range$ has a  positive rational coefficient.   Multiplying away the denominators produces a   loop $x+z_1+\cdots+z_m=0$.  
 
This lemma  gives sufficient conditions for membership in class $\Ll$ of
Definition \ref{cL-def}.

\begin{lemma}\label{lm-L}  Let    $(\Omega,\kS,\P,\{T_x:x\in\gr\})$
be a measurable ergodic  dynamical system.   Let $0\le g\in L^1(\P)$.
  Assume one
of the conditions  {\rm(a)}--{\rm(d)} below.  
\begin{itemize}
\item[{\rm(a)}] $g$ is bounded. 
\item[{\rm(b)}] $d=1$. 
%\item[{\rm(c)}] $d=2$, $G_x$ is i.i.d., and $E[G_0^2]<\infty$.
\item[{\rm(c)}] $d\ge2$. There exist $r\in(0,\infty)$ and $p>d$ such that $\E[g^p]<\infty$ and $\{g\circ T_{x_i}: i=1,\dotsc,m\}$ are i.i.d.\ whenever $\abs{x_i-x_j}\ge r$ 
for all $i\ne j$. 
\item[{\rm(d)}]  $d\ge 2$.   There exist $a>d$ and $p>ad/(a-d)$ such that $\E[g^p]<\infty$ and for each $z\in\range\setminus\{0\}$ and large $k\in\N$
\begin{align}\label{cond:strong-mixing}
\sup_{\substack{A\in\sigma(g\circ T_x:\,x\cdot z\le0)\\[2pt]B\in\sigma(g\circ T_x:\,x\cdot z\ge k)}}|\P(A\cap B)-\P(A)\P(B)|\le k^{-a}.
\end{align}  
\end{itemize}
  Then, for each $z\in\range\setminus\{0\}$ 
 \beq%\begin{align*}
\lim_{\e\to 0}\;\varlimsup_{n\to\infty} \; \max_{x\in\gr: \abs{x}\le n}\; \frac1n \sum_{i=0}^{\e n} g\circ T_{x+iz} =0  \quad\text{$\P$-a.s.}
%\le \e\rho_z(\w) \quad\text{for $\P$-a.e.~$\w$.}
%\end{align*} 
\label{appA-9}\eeq
\end{lemma}

\begin{proof} Part (a) is immediate.   
%For the rest it suffices to prove that $\max_{|x|\le n} \frac1n \big|\sum_{i=0}^{n} \Gbar_{x+iz}\big|$ converges to $0$ almost surely, where $\Gbar_x=G_x-\E[G_x]$.

\medskip

For (b) let $s\in\N$ be such that $\gr=\{ns: n\in\Z\}$.  Fix $z=as$ and let 
$\bar g=g-\E(g\,\vert\,\mathcal I_z)$ where $\mathcal I_z$ is the $\sigma$-algebra of
events invariant under $T_z$.  By $T_{as}$-invariance 
\[  \max_{-n\le j\le n} \frac1n \sum_{i=0}^{\e n} \E(g\,\vert\,\mathcal I_z)\circ T_{js+ias}
\le  (\e+\tfrac1n) \max_{0\le j< a} \E(g\,\vert\,\mathcal I_z)\circ T_{js} \quad\text{$\P$-a.s.}\]
By the ergodic theorem 
\[ \varlimsup_{n\to\infty} \; \max_{\abs{\ell}\le n}\;  \biggl\lvert  \frac1n \sum_{i=0}^{ n} \bar g\circ T_{\ell z+iz}   \biggr\rvert = 0 \quad\text{$\P$-a.s.}\]
This limit is not changed by taking  a finite maximum over the shifts by $T_{js}$, $0\le j< a$. 

Part (c) follows from part (d).

Fix $z$ for part (d).   First two reductions.  
 (i) The maximum over $x$ in  \eqref{appA-9} can be restricted to
a set $A_n$ of size $\abs{A_n}\le Cn^{d-1}\e^{-1}$, at the expense of 
doubling $\e$ in the upper summation limit. The reason is that $g\ge0$ and if $x'=x+jz$  for 
 some $1\le j< n\e/2$, then the $2n\e$-sum started at $x$ covers the 
$n\e$-sum started at $x'$.  

(ii) It suffices to consider a subsequence $n_m=m^\gamma$ for any
fixed  $\gamma>0$ because $n_{m+1}/n_m\to 1$ and $g\ge 0$.  

Since constants satisfy  \eqref{appA-9} we can   replace $g$ with  $\bar g=g-\E[ g]$.
Let 
$S^x_n=\sum_{i=0}^n \bar g\circ T_{x+i z}$.  
Equation \eqref{cond:strong-mixing} and the translation invariance of $\P$ imply strong mixing as defined by \cite{Rio-95}.
Then applying Theorem 6 therein  with $u=n^{-b}$, $r$ large enough, and $t=\delta n/(cr)$
we get a generalization of the Fuk-Nagaev inequality to square-integrable, mean-zero strongly mixing random variables.
%(Theorem 6 in \cite{Rio-95}), 
This implies that for fixed $\e,\delta>0$,  
$  \P\{ \absa{S^x_{n\e}} > n\delta\} \le C(\e,\delta)n^{1-b}$ with $b=ap/(a+p)>d$.
By  a straightforward union bound 
\begin{align*}
\P\Bigl\{ \max_{x\in A_n} \Bigl\lvert  \sum_{i=0}^{n\e} \bar g\circ T_{x+i z} \Bigr\rvert
 > n\delta \Bigr\}
\le    Cn^{d-1}\e^{-1}  \P\{ \absa{S^0_{n\e}} > n\delta\}  \le  C(\e,\delta)n^{d-b}.
\end{align*}
Along the subsequence $n_m=m^\gamma$ for $\gamma>(b-d)^{-1}$ the last
bound is summable.  We get $\P$-a.s.\ convergence to $0$ for each fixed $\e>0$ 
by the   Borel-Cantelli lemma. 
\end{proof}  

For a general ergodic system (a) cannot be improved.  For example, take $d=2$,
an i.i.d.\ sequence  $\{\w_{i,0}\}_{i\in\Z}$, and then set $\w_{i,j}=\w_{i,0}$.  
For $z=e_2$,  $n^{-1}\sum_{j=0}^{n\e} \abs{\w_{x+(0,j)}}$ $\ge  \e \abs{\w_x}$ and consequently
the limit in $n$ in    \eqref{appA-9}  blows up unless $\w_{i,j}$ is a bounded process. 

If the mixing in part (d) above is faster than any polynomial, then we can take $a\to\infty$ and the condition becomes  $p>d$. 
%In particular, this is  the case if there exists $r\in(0,\infty)$ such that  $\{g\circ T_{x_i}: i=1,\dotsc,m\}$ are i.i.d.\ whenever $\abs{x_i-x_j}\ge r$ 
%for all $i\ne j$. 
%However, if mixing is polynomial at rate $k^{-a}$ then there is a trade-off between $p$ and $a$.
%In this i.i.d.\ case 
Part (c)   is close to optimal.  If $\E[g^d]=\infty$ then 
$n^{-1}\max_{\abs{x}\le n} g\circ T_{x}$ blows up by the second  Borel-Cantelli lemma. 
Currently we do not know if $p\ge d$ is sufficient in (c).

\section{Weak LDP through a projective limit}
We describe a small alteration of  the projective limit LDP.  
Let   $\mathcal X$ and   $\mathcal X_j$, $j\in\mathbb N$, be metric spaces
with  
continuous maps  $g_j:\mathcal X\to\mathcal X_j$ and 
$g_{j,i}:\mathcal X_j\to\mathcal X_i$ for $i<j$ such that 
$g_i=g_{j,i}\circ g_j$ and $g_{k,i}=g_{j,i}\circ g_{k,j}$.  Let $\{\mu_n\}$ be a sequence of Borel probability
measures on $\mathcal X$, and define $\mu^j_n=\mu_n\circ g_j^{-1}$ on $\mathcal X_j$. 
Let $I_j:\mathcal X_j\to[0,\infty]$ be lower semicontinuous.   Define 
$I(x)=\sup_j I_j(g_j(x))$ for $x\in\mathcal X$.

\begin{theorem}  {\rm (i)}  Suppose that  for all $j$, $I_j\circ g_{j+1,j}\le I_{j+1}$ and  $I_j$ satisfies
 the   large deviation upper  bound  for compact sets in $\mathcal X_j$. 
 Then $I$ satisfies
 the   large deviation upper  bound  for compact sets in $\mathcal X$. 
 
 {\rm (ii)}  Assume
that 
$\mathcal U=\{ g_j^{-1}(U_j):  j\in\mathbb N,  U_j\subseteq\mathcal X_j  \text{ open} \}  $
is a base for the topology of $\mathcal X$. 
Suppose that  for all $j$,    $I_j$ satisfies
 the   large deviation lower  bound  for open sets in $\mathcal X_j$. 
 Then $I$ satisfies
 the   large deviation lower  bound  for open  sets in $\mathcal X$. 
 
\label{proj-ldp-thm}
\end{theorem}

 \begin{proof}   Part (ii) is straightforward.  We prove 
part  (i).  Let $A\subseteq\mathcal X$ be compact. 
 Since $g_j(A)$ is compact in $\mathcal X_j$  and  $g_j^{-1}(g_j(A))\supseteq A$,  
\begin{align*}
\varlimsup\, n^{-1}\log\mu_n(A) &\le \varlimsup\, n^{-1}\log\mu^j_n(g_j(A)) \le 
-\inf_{y\in g_j(A)} I_j(y)\\
&=-\inf_{x\in A} I_j(g_j(x))
\end{align*} 
from which 
\begin{align*}
\varlimsup\, n^{-1}\log\mu_n(A) \le   -\sup_j \inf_{x\in A} I_j(g_j(x)). 
\end{align*} 

Next we claim a minimax property from the assumption of monotonicity:  
\begin{equation}
\sup_j \inf_{x\in A} I_j(g_j(x)) = \inf_{x\in A} \sup_j  I_j(g_j(x)) \equiv \inf_{x\in A}  I(x). 
\label{pr-lim:minimax}\end{equation}
Inequality $\le$ is  obviously true. To show $\ge$, let 
$c< \inf_{x\in A} \sup_j  I_j(g_j(x))$.   Then each $x\in A$ has an index $j(x)$ such
that  $I_{j(x)}(g_{j(x)}(x))>c$.   The set
$D_x=\{ z\in\mathcal X:  I_{j(x)}(g_{j(x)}(z))>c\}$ is open by the continuity of $g_j$ and 
lower semicontinuity of $I_j$.  Cover $A$ with finitely many:
$A\subseteq D_{x_1}\cup \dotsm\cup D_{x_k}$.  Fix $j\ge j(x_1)\vee\dotsm\vee j(x_k)$.
Then if $x\in A$ pick $\ell$ such that $x\in D_{x_\ell}$, and we have 
\[   I_j(g_j(x)) \ge   I_{j(x_\ell)} \bigl( g_{j,\, j(x_\ell)}(g_j(x))\bigr) 
=  I_{j(x_\ell)} \bigl( g_{j(x_\ell)}(x)\bigr) > c.    \]
Thus  $\inf_{x\in A} I_j(g_j(x))\ge c$. We have proved \eqref{pr-lim:minimax} and 
thereby the upper large deviation bound for $A$. 
\end{proof}

\section{Proofs of Lemmas %\ref{L1 limit},  
\ref{CLASS K} %,  \ref{lm:Sk}, 
and \ref{F-lemma}}\label{app:pf lm class K}

Standing assumptions in this section are the same as in Section \ref{sec:var rep}:
$(\Omega,\kS,\P,\{T_z:z\in\gr\})$ is a measurable ergodic dynamical system and,  as throughout the paper,  $\range$ is
an arbitrary finite subset of $\Z^d$ that generates the additive  group $\gr$.
Throughout this section  $\ell\ge0$ is a fixed integer. 
  $C$ denotes a chameleon 
constant that  can change   from term to term and  only depends on
$\range$, $\ell$, and $d$. 
 In order to avoid working on a sublattice, we will assume throughout this appendix that $\range$ generates $\Z^d$ as a group.
This does not cause any loss of generality. The additive  group  $\gr$
generated by $\range$ is  linearly  isomorphic to $\Z^{d'}$ for some $d'\le d$ \cite[p.~65-66]{Spi-76} and  we can transport the model to   $\Z^{d'}$.  

A crucial tool will be the path integral of a function $F\in\K$. The main idea is that due to the closed loop property
these functions are gradient-like.

For $\ell$-tuples $\ztil_{1,\ell}, \zbar_{1,\ell}\in\range^\ell$ we 
write  $\xtil_\ell=\ztil_1+\dotsc+\ztil_\ell$ and
 $\xbar_\ell=\zbar_1+\dotsc+\zbar_\ell$.
%Since $\range$ generates $\Z^d$ as a group, there always exists
We say that there exists a path from $(y,\tilde z_{1,\ell})$ to $(x,z_{1,\ell})$ 
when there exist 
$a_1,\dotsc,a_{m}\in\range$ such that 
%	\[  y+\xtil_\ell+ a_1+\dotsm+a_{m-\ell}=x. \]
%The definition is independent of $z_{1,\ell}$ but for symmetry of language it seems
%sensible to keep it in the statement. The case $m=\ell$ is admissible also and then
%$y+\xtil_\ell=x$.   Then if we set $a_{m-\ell+1, m}=z_{1,\ell}$,
the composition $\Sopr_{a_m}\circ\dotsm\circ\Sopr_{a_1}$ takes 
$(T_y\w, \ztil_{1,\ell})$ to $(T_x\w,z_{1,\ell})$ for all $\w\in\Omega$.  
This is equivalent to the pair of equations 
\[ 
    y+\xtil_\ell+a_1+\dotsm+a_{m-\ell}=x  \quad\text{and}\quad  a_{m-\ell+1,m}=z_{1,\ell}. 
 \]

%Paths can be concatenated.  If there is a path from $(y,\ztil_{1,\ell})$ to $(x,z_{1,\ell})$
%and from $(u, \zbar_{1,\ell})$ to $(y,\ztil_{1,\ell})$, then we have 
%	\[  y+\xtil_\ell+ a_1+\dotsm+a_{m-\ell}=x 
%		\quad\text{and}\quad u+\xbar_\ell+ b_1+\dotsm+b_{n-\ell}=y. \]
%Taking $b_{n-\ell+1,n}=\tilde z_{1,\ell}$ we then have 
%	\[   u+\xbar_\ell+ b_1+\dotsm+b_{n} + a_1+\dotsm+a_{m-\ell}=x \]
%and there is a path from $(u, \zbar_{1,\ell})$ to $(x,z_{1,\ell})$. 

For any two points $(x,z_{1,\ell})$ and $(\xbar,\zbar_{1,\ell})$ and any $\ztil_{1,\ell}$
there exists a point $y\in\Z^d$ such that from $(y,\ztil_{1,\ell})$  there is a path to  both $(x,z_{1,\ell})$ 
and $(\xbar,\zbar_{1,\ell})$.  For this, find  first 
$\abar_1,\dotsc,\abar_{m-\ell}$ and $a_1,\dotsc,a_{n-\ell}\in\range$ such that 
	\[  \xbar- x = (\abar_1+\dotsm +\abar_{m-\ell})-(  a_1+\dotsm+a_{n-\ell}) \]
so that 
	\[y'=\xbar-  (\abar_1+\dotsm +\abar_{m-\ell}) = x  -(  a_1+\dotsm+a_{n-\ell})\]
and then take $y= y'-\xtil_\ell$.   By induction, for any finite number of points there is a common starting point from
which there exists a path to each of the chosen points.

Now fix  a   measurable function $F:\bigom_\ell\times\range\to\R$ 
that satisfies   the closed
loop property {\rm (iii)} of Definition \ref{cK-def}.  
If there is a path $(a_i)_{i=1}^m$ from $(y,\ztil_{1,\ell})$ to $(x,z_{1,\ell})$, set $\wz_0=(T_y\w, \ztil_{1,\ell})$, 
$\wz_i=\Sopr_{a_i}\wz_{i-1}$ for $i=1,\dotsc, m$ so that $\wz_m=(T_x\w, z_{1,\ell})$,
and then   
	\begin{align}
  		L(\w, (y,\ztil_{1,\ell}), (x,z_{1,\ell}))= \sum_{i=0}^{m-1} F(\wz_i, a_{i+1}).  \label{defL}
	\end{align}
By the closed loop property $L(\w, (y,\ztil_{1,\ell}), (x,z_{1,\ell}))$ is independent of the
path chosen. 
We also admit   an empty path that gives 
\[  L(\w, (x,z_{1,\ell}), (x,z_{1,\ell})) = 0.  \]
If $a_1,\dotsc,a_{m}$  work for  $(y,\ztil_{1,\ell})$ and  $(x,z_{1,\ell})$,
then these steps work also for  $(y+u,\ztil_{1,\ell})$ and  $(x+u,z_{1,\ell})$. 
The effect on the right-hand side of \eqref{defL} is  to shift $\w$ by $u$, and
consequently 
	\begin{align}  
	L(T_u\w, (y,\ztil_{1,\ell}), (x,z_{1,\ell}))=L(\w, (y+u,\ztil_{1,\ell}), (x+u,z_{1,\ell})). 
	\label{shiftL}
	\end{align}

Next define $f:\Omega\times\range^{2\ell}\times\Z^d\to\R$ by
	\begin{align}   
		f(\w, z_{1,\ell}, \zbar_{1,\ell}, x)= L(\w, (y,\ztil_{1,\ell}), (x,\zbar_{1,\ell}))
			- L(\w, (y,\ztil_{1,\ell}), (0,z_{1,\ell}))  \label{deff}
	\end{align}
for any $(y,\ztil_{1,\ell})$ with  a path to  both $(0,z_{1,\ell})$ and $(x,\zbar_{1,\ell})$.   
This definition is  independent of the choice of  $(y,\ztil_{1,\ell})$, again by the closed loop property.
%For suppose $(y',\tilde z'_{1,\ell})$ is an alternative choice.  
%Pick $(y'',\tilde z''_{1,\ell})$ with a path to both $(y,\tilde z_{1,\ell})$ and
% $(y',\tilde z'_{1,\ell})$.  Then the closed loop property implies 
%\begin{align*} 
%&L(\w, (y'',\tilde z''_{1,\ell}), (y,\tilde z_{1,\ell})) + L(\w, (y,\tilde z_{1,\ell}), (0,  z_{1,\ell})) \\
%&\qquad =  L(\w, (y'',\tilde z''_{1,\ell}), (y',\tilde z'_{1,\ell})) + L(\w, (y',\tilde z'_{1,\ell}), (0,  z_{1,\ell})) 
%\intertext{and}
%&L(\w, (y'',\tilde z''_{1,\ell}), (y,\tilde z_{1,\ell})) + L(\w, (y,\tilde z_{1,\ell}), (x, \bar z_{1,\ell})) \\
%&\qquad 
%=  L(\w, (y'',\tilde z''_{1,\ell}), (y',\tilde z'_{1,\ell})) + L(\w, (y',\tilde z'_{1,\ell}), (x, \bar z_{1,\ell})) 
%\end{align*} 
%from which follows 
%\begin{align*}  &L(\w, (y,\tilde z_{1,\ell}), (x,\bar z_{1,\ell}))
%- L(\w, (y,\tilde z_{1,\ell}), (0,z_{1,\ell})) \\
%&\qquad = L(\w, (y',\tilde z'_{1,\ell}), (x,\bar z_{1,\ell}))
%- L(\w, (y',\tilde z'_{1,\ell}), (0,z_{1,\ell})).  \end{align*}
%Thus  definition \eqref{deff}  is independent of the choice
%of  $(y,\tilde z_{1,\ell})$. 

Here are some basic properties of $f$.

\begin{lemma}\label{f properties}  Let $F(\cdot\,, z_{1,\ell}, z)\in L^1(\P)$ for each 
$(z_{1,\ell}, z)$ and satisfy the closed
loop property {\rm (iii)} of Definition \ref{cK-def}.  
\begin{itemize}
\item[{\rm (a)}]  There exists a constant $C$ depending only on $d$, $\ell$, and $R=\max\{|z|:z\in\range\}$ such that
for all
$z_{1,\ell},\zbar_{1,\ell}\in\range^\ell$, $x\in\Z^d$, and $\P$-a.e.\ $\w$
 	\[|f(\w,z_{1,\ell},\zbar_{1,\ell},x)|\le\sum_{b:|b|\le C(|x|+1)}\max_{\ztil_{1,\ell}\in\range^\ell}\max_{z\in\range}|F(T_b\w,\ztil_{1,\ell},z)|.\]
	 In particular,  $f\in L^1(\P)$ for all $(z_{1,\ell}, \bar z_{1,\ell}, x)$.  
\item[{\rm (b)}]  For  $z_{1,\ell},\zbar_{1,\ell},\ztil_{1,\ell}\in\range^\ell$, $x,\xbar\in\Z^d$, and $\P$-a.e.\ $\w$, 
 \begin{align*} 
f(\w,z_{1,\ell},\ztil_{1,\ell},\xbar)=f(\w,z_{1,\ell},\zbar_{1,\ell},x)+f(T_x\w,\zbar_{1,\ell},\ztil_{1,\ell},\xbar-x). 
\end{align*}
\item[{\rm (c)}]   Assume additionally that $F$ satisfies the mean zero property  {\rm (ii)} of Definition \ref{cK-def}.   Then for any $\zbar_{1,\ell}\in\range^\ell$ and $x\in\Z^d$, $\E[f(\w, \zbar_{1,\ell},  \zbar_{1,\ell}, x)]=0.$
\end{itemize}
\end{lemma}

\begin{proof} %[Proof of Lemma \ref{f properties}]  
Let $e_1,\dotsc,e_d$ be the canonical basis of $\R^d$.
For each $1\le i\le d$, there exist nonnegative integers $n_i^\pm$ and $(a_{i,j}^\pm)_{j=1}^{n_i^\pm}$ from $\range$ such that
	\[e_i=a^+_{i,1}+\cdots+a^+_{i,n^+_i}-a^-_{i,1}-\cdots-a^-_{i,n^-_i}.\]
Write $x=\sum_{i=1}^d b_i \epsilon_i e_i$ with $b_i\ge0$ and $\epsilon_i\in\{-1,+1\}$. Then, 
	\[x=\sum_{i=1}^d\sum_{j=1}^{n_i^{\epsilon_i}}b_i a^{\epsilon_i}_{i,j}-\sum_{i=1}^d\sum_{j=1}^{n^{-\epsilon_i}_i}b_ia^{-\epsilon_i}_{i,j}.\]
One can thus find a $y$ that has paths to both $0$ and $x$ that stay inside a ball of radius $C(|x|+1)$. This proves (a).
  
To prove (b),  
let  $(y,\zhat_{1,\ell})$ have paths to $(-x,z_{1,\ell})$, $(0,\zbar_{1,\ell})$, and $(\xbar-x,\ztil_{1,\ell})$.
Use the definition of $f$ \eqref{deff} and the shift property \eqref{shiftL} to write
	\begin{align*}
	&f(T_x\w,\zbar_{1,\ell},\ztil_{1,\ell},\xbar-x)\\
	&\qquad=L(T_x\w,(y,\zhat_{1,\ell}),(\xbar-x,\ztil_{1,\ell}))-L(T_x\w,(y,\zhat_{1,\ell}),(0,\zbar_{1,\ell}))\\
	&\qquad=L(\w,(y+x,\zhat_{1,\ell}),(\xbar,\ztil_{1,\ell}))-L(\w,(y+x,\zhat_{1,\ell}),(x,\zbar_{1,\ell}))\\
	&\qquad=\big[L(\w,(y+x,\zhat_{1,\ell}),(\xbar,\ztil_{1,\ell}))-L(\w,(y+x,\zhat_{1,\ell}),(0,z_{1,\ell}))\big]\\
	&\qquad\qquad-\big[L(\w,(y+x,\zhat_{1,\ell}),(x,\zbar_{1,\ell}))-L(\w,(y+x,\zhat_{1,\ell}),(0,z_{1,\ell}))\big]\\
	&\qquad=f(\w,z_{1,\ell},\ztil_{1,\ell},\xbar)-f(\w,z_{1,\ell},\zbar_{1,\ell},x).
	\end{align*}

For (c), let $y$ be so that 
from $(y,\bar z_{1,\ell})$  there is a path to  both $(x,\bar z_{1,\ell})$ 
and $(0,\bar z_{1,\ell})$.  Then 
\[   f(\w, \bar z_{1,\ell}, \bar z_{1,\ell}, x)= L(\w, (y,\bar z_{1,\ell}), (x,\bar z_{1,\ell}))
- L(\w, (y,\bar z_{1,\ell}), (0, \bar z_{1,\ell})).  \]
Both $L$-terms above equal sums 
$\sum_{i=0}^{m-1} F(\wz_i, a_{i+1})$ where $\wz_0=(T_y\w, \bar z_{1,\ell})$
and $\wz_m=(T_u\w, \bar z_{1,\ell})$ with $u=x$ or $u=0$.  Both have 
zero $\E$-mean by property (ii) of Definition \ref{cK-def}. 
\end{proof}

\begin{remark}
Part (b)  above shows that $f$ is a path integral of $F$ or, alternatively, that $F$ is a gradient of $f$. More precisely,
\[F(\w,z_{1,\ell},z)=f(\w,\zbar_{1,\ell},\Sopr_z z_{1,\ell},z_1)-f(\w,\zbar_{1,\ell},z_{1,\ell},0),\]
for all $z_{1,\ell}$, $\zbar_{1,\ell}\in\range^\ell$, $z\in\range$, and $\P$-a.e.\ $\w$. ($\Sopr_z$ acts on $\range^\ell$ in the obvious way.)
\end{remark}

\begin{lemma}\label{L1 limit} 
Let $F\in\K$. Then, there exists a sequence of bounded
measurable functions $h_k:\bigom_\ell\to\R$ such that $\E[|h_k(\Sopr_z\wz)-h_k(\wz)-F(\wz,z)|]\to0$
for all $z_{1,\ell}\in\range^\ell$ and $z\in\range$.
\end{lemma}

\begin{proof}[Proof of Lemma \ref{L1 limit}]
Starting with $F$,  denote its path integral by $f$ as above. Define
	\[g_n(\w,z_{1,\ell})=-|\range|^{-\ell}(2n+1)^{-d}\sum_{\zbar_{1,\ell}\in\range^\ell}\sum_{|x|\le n}f(\w,z_{1,\ell},\zbar_{1,\ell},x).\]
By part  (b) of Lemma \ref{f properties} 
\begin{align*}
&f(\w,z_{1,\ell},\zbar_{1,\ell},x)+f(T_x\w,\zbar_{1,\ell},\zbar_{1,\ell},z_1)=f(\w,z_{1,\ell},\zbar_{1,\ell},x+z_1)\\
&\qquad=f(\w,z_{1,\ell},\Sopr_z z_{1,\ell},z_1)+f(T_{z_1}\w,\Sopr_z z_{1,\ell},\zbar_{1,\ell},x).
\end{align*}
Consequently,  from the closed loop 
property alone,  
	\begin{align}
	&g_n(\Sopr_z\wz)-g_n(\wz)\\
	&\qquad=|\range|^{-\ell}(2n+1)^{-d}\sum_{\zbar_{1,\ell}\in\range^\ell}\sum_{|x|\le n}\big[f(\w,z_{1,\ell},\zbar_{1,\ell},x)\nn\\
	&\qquad\qquad\qquad\qquad\qquad\qquad\qquad\qquad-f(T_{z_1}\w,\Sopr_z z_{1,\ell},\zbar_{1,\ell},x)\big]\nn\\
	&\qquad=|\range|^{-\ell}(2n+1)^{-d}\sum_{\zbar_{1,\ell}\in\range^\ell}\sum_{|x|\le n}\big[f(\w,z_{1,\ell},\Sopr_z z_{1,\ell},z_1)\nn\\
	&\qquad\qquad\qquad\qquad\qquad\qquad\qquad\qquad-f(T_x\w,\zbar_{1,\ell},\zbar_{1,\ell},z_1)\big]\nn\\
	&\qquad=F(\wz,z)-|\range|^{-\ell}(2n+1)^{-d}\sum_{\zbar_{1,\ell}\in\range^\ell}\sum_{|x|\le n}f(T_x\w,\zbar_{1,\ell},\zbar_{1,\ell},z_1).\label{F-c}
	\end{align}
%just use (b) of Lemma B.3.  let \ztil=S_z z_{1,\ell}.
%
%f(\w,z_{1,\ell},\zbar,x)+f(T_x\w,\zbar,\zbar,z_1)
%f(\w,z_{1,\ell},\ztil,z_1)+f(T_{z_1}\w,\ztil,\zbar,x)
%are both equal to f(\w,z_{1,\ell},\zbar,x+z_1).
%
%the other one is similar (and uses the fact that F(\wz,z)=f(\w,z_{1,\ell},\ztil,z_1) )

	By parts (a) and (c) of Lemma \ref{f properties} and by the ergodic theorem we see that $F(\wz,z)$ is the $L^1(\P)$-limit of 
	$g_n(\Sopr_z\wz)-g_n(\wz)$ for each $z_{1,\ell}\in\range^\ell$ and $z\in\range$. 
	Finally, approximate  the
integrable $g_n$ with a bounded $h_n$ in $L^1(\P)$.
%	Although functions $g_n$ are not necessarily bounded we can find $a_n$ such tha $\E[|g_n|\one\{|g_n|>a_n\}]\le n^{-1}$ for all $z_{1,\ell}\in\range^\ell$. Then, functions $h_n=g_n\one\{|g_n|\le a_n\}$ do the job:	 	\[\E[|h_n(\Sopr_z\wz)-h_n(\wz)-F(\wz,z)|]\le\E[|g_n(\Sopr_z\wz)-g_n(\wz)-F(\wz,z)|]+2n^{-1}\mathop{\longrightarrow}_{n\to\infty}0.\qedhere\]
\end{proof}

\begin{proof}[Proof of Lemma \ref{CLASS K}]
The $L^1(\P)$ convergence to $0$ follows from Lemma \ref{L1 limit}.
Next, observe that for any $a_{1,n}$ that satisfies the properties in braces in the statement of the lemma,
the $F$-sum satisfies
\[\sum_{i=0}^{n-1} F(\wz_i, a_{i+1}) = f(\w, z_{1,\ell}, \zbar_{1,\ell}, \xhat_n(\xi)).\]  
Consequently the task is to show that $n^{-1}f(\w,z_{1,\ell},\zbar_{1,\ell},\xhat_n(\xi))$ has a limit $\P$-a.s.

%For $\xi\in\Q^d\cap\Uset$ and 
Recall that the definition of the path $\xhat_\centerdot(\xi)$ given above Lemma 
\ref{lm-pressure2} involved an integer $b=b(\xi)$ such that 
  $b\xi\in\Z^d$ and  $\xhat_{mb}(\xi)=mb\xi$ for all $m$. 
Using (b) of Lemma \ref{f properties} we have 
	\begin{align*}
	(mb)^{-1}f(\w,z_{1,\ell},\zbar_{1,\ell},mb\xi)&= (mb)^{-1}f(\w,z_{1,\ell},\zbar_{1,\ell},0)\\
	&\qquad\qquad+(mb)^{-1}\sum_{j=0}^{m-1}f(T_{jb\xi}\w,\zbar_{1,\ell},\zbar_{1,\ell},b\xi)
	\end{align*}
and by the ergodic theorem, the right-hand side has a $\P$-almost sure limit.
%But lemma \ref{L1 limit} implies that the left-hand side converges to $0$ strongly in $L^1(\P)$.
%Consequently, $(mb)^{-1}f(\w,z_{1,\ell},\zbar_{1,\ell},mb\xi)$ converges $\P$-almost surely to $0$ as $m\to\infty$.

Given $n$ choose $m_n$ so that $m_nb\le n<(m_n+1)b$.
By (b) and (a) of Lemma \ref{f properties} 
	\begin{align*}
	&n^{-1}|f(\w,z_{1,\ell},\zbar_{1,\ell},\xhat_n(\xi))-f(\w,z_{1,\ell},\zbar_{1,\ell},m_nb\xi)|\\ 
	&\qquad= n^{-1}|f(T_{m_nb\xi}\w,\zbar_{1,\ell},\zbar_{1,\ell},\xhat_n(\xi)-m_nb\xi)|\\
	 &\qquad\le (m_n b)^{-1} G(T_{m_nb\xi}\w)\mathop{\longrightarrow}_{n\to\infty}0\quad\text{$\P$-a.s.},
	 \end{align*}
where $G(\w)=\sum_{x:\abs{x}\le C(b|\xi|+1)}\max_{\ztil_{1,\ell}\in\range^\ell}\max_{z\in\range}|F(T_x\w,\ztil_{1,\ell},z)|$ is in $L^1(\P)$. 
\end{proof}

\begin{proof}[Proof of Lemma \ref{F-lemma}]
Fix $\e>0$ for the rest of the proof.
From \eqref{F-bound} we have that for $\P$-a.e.\ $\w$ and for all $z_{1,\ell}\in\range^\ell$ and $z\in\range$
	\[F^{(0)}_{k,\e}(\wz,z)\le C-\E[g(\wz)\,|\,\kS_{k}].\] 
By the fact that 
$g(\w,z_{1,\ell})\in L^1(\P)$ for all $z_{1,\ell}\in\range^\ell$
the right-hand side is uniformly integrable. Thus so is $F_{k,\e}^{(0),+}=\max(F^{(0)}_{k,\e},0)$.

Let $F^{(0),-}_{k,\e}=\max(-F^{(0)}_{k,\e},0)$. 
Observe that by the $T_z$-invariance of $\P$
	\begin{align*}%\label{mean zero k}
	\E[F^{(0)}_{k,\e}(\w,z_{1,\ell},z)]&=\E[h_{k,\e}(T_{z_1}\w,\Sopr_{z}z_{1,\ell})-h_{k,\e}(\w,z_{1,\ell})]\\
	&=\E[h_{k,\e}(\w,\Sopr_{z}z_{1,\ell})]-\E[h_{k,\e}(\w,z_{1,\ell})].
	\end{align*} 
Thus $F^{(0)}_{k,\e}$ satisfies the mean-zero property (ii) in Definition \ref{cK-def}. 
Letting $\wz_0=(\w,z_{1,\ell})$, $z_0=z$, $a_i=z_{i-1}$, and $\wz_i=\Sopr_{a_i}\wz_{i-1}$ for $i=1,\dotsc,\ell+1$,  
one has that
\[\E[F^{(0),-}_{k,\e}(\wz,z)]\le\sum_{i=0}^{\ell} \E[F_{k,\e}^{(0),-}(\wz_i,a_{i+1})]=\sum_{i=0}^{\ell} \E[F_{k,\e}^{(0),+}(\wz_i,a_{i+1})]\]
is bounded uniformly in $k$. We apply the following lemma to extract a uniformly integrable part leaving a small error.

\begin{lemma}[Lemma 4.3 of \cite{Kos-Var-08}]\label{lm:raghu} 
Let $\{g_n\}_{n\ge1}$  be a sequence of nonnegative functions such that 
$\sup_n E[g_n]\le C$. Then there is a subsequence 
$\{n_j\}_{j\ge1}$ and an increasing sequence $a_j\nearrow\infty$ such that $g_{n_j}\one\{g_{n_j}\le a_j\}$ is uniformly integrable
and $g_{n_j}\one\{g_{n_j}>a_j\}$ converges to $0$ in probability.
\end{lemma}

By the above lemma we can write $F^{(0),-}_{k,\e}=\Ftil^{(0)}_{k,\e}+R^{(0)}_{k,\e}$
such that along a subsequence $\Ftil^{(0)}_{k,\e}$ is uniformly integrable and $R_{k,\e}^{(0)}\ge0$ is $\kS_k$-measurable 
and converges to $0$ in $\P$-probability. One can then take a further subsequence along which $\Fhat^{(0)}_{k,\e}=F_{k,\e}^{(0),+}-\Ftil^{(0)}_{k,\e}$ converges in weak $L^1(\P)$   to   $\Fhat^{(0)}_\e$ and $R^{(0)}_{k,\e}\to0$  $\P$-a.s. 
(Uniform integrability gives sequential compactness in weak $L^1$; see Theorem 9 on page 292 of \cite{Dun-Sch-58}.)
We will always keep indexing subsequences by $k$.
Now we have the decomposition 
\begin{align}\label{F-decomposition}
F^{(0)}_{k,\e}=\Fhat^{(0)}_{k,\e} - R^{(0)}_{k,\e}. 
\end{align}

An attempt to check  that the limit $\Fhat^{(0)}_\e$ satisfies the closed loop property
runs into difficulty because we have very weak control of the errors  $R^{(0)}_{k,\e}$
and the conditioning  in   definition \eqref{F} damages the closed loop property of
the function $h_{k,\e}(\Sopr_z\wz)-h_{k,\e}(\wz)$.   To get around this   we defined
the family indexed by $i$ in \eqref{F}. In the next lemma we develop a  hierarchy of
errors obtained by successive application of Lemma \ref{lm:raghu}.   We give the proof after the current proof is done. 
Recall that given $z_{1,j}\in\range^j$, $x_j=z_1+\cdots+z_j$.
We will use the notation $\varnothing$ for a path of length $j=0$ and then $x_0=0$.

\begin{lemma}\label{lm:construction}
There exist nonnegative random variables on $\bigom_\ell\times\range$, denoted by $\Rtil_{k,\e}^{(i,j,z_{1,j})}$, $\Rhat_{k,\e}^{(i,j,z_{1,j+1})}$, and 
$R_{k,\e}^{(i,j,z_{1,j})}$, with $0\le j\le i\le k$ and $z_{1,j+1}\in\range^{j+1}$, such that the following properties are satisfied. 
\begin{itemize}
\item[{\rm(a)}] $R_{k,\e}^{(0,0,\varnothing)}=R_{k,\e}^{(0)}$.
\item[{\rm(b)}] $\Rtil_{k,\e}^{(i,j,z_{1,j})}$, $\Rhat_{k,\e}^{(i,j,z_{1,j+1})}$,
and $R_{k,\e}^{(i,j,z_{1,j})}$ are $T_{x_j}\kS_{k-i}$-measurable.
\item[{\rm(c)}] $\E[R_{k,\e}^{(i,0,\varnothing)}\,|\,\kS_{k-i-1}]=\Rtil_{k,\e}^{(i+1,0,\varnothing)}+R_{k,\e}^{(i+1,0,\varnothing)}$ for all $i\ge0$. 
\item[{\rm(d)}] $\E[R_{k,\e}^{(i,j,z_{1,j})}\,|\,T_{x_{j-1}}\kS_{k-i-1}]=\Rhat_{k,\e}^{(i+1,j-1,z_{1,j})}+R_{k,\e}^{(i+1,j-1,z_{1,j-1})}$ for all $i\ge j\ge1$ and $z_{1,j}\in\range^j$.
\item[{\rm(e)}] $\E[R_{k,\e}^{(i,j,z_{1,j})}\,|\,T_{x_{j+1}}\kS_{k-i-1}]=\Rtil_{k,\e}^{(i+1,j+1,z_{1,j+1})}+R_{k,\e}^{(i+1,j+1,z_{1,j+1})}$ for all $i\ge j\ge0$ and $z_{1,j+1}\in\range^{j+1}$.
\item[{\rm(f)}] As $k\to\infty$, $\Rtil_{k,\e}^{(i,j,z_{1,j})}$ and $\Rhat_{k,\e}^{(i,j,z_{1,j+1})}$ are uniformly integrable and converge in weak $L^1(\P)$  to a limit
$\Rtil_{\e}^{(i,j,z_{1,j})}$ and $\Rhat_{\e}^{(i,j,z_{1,j+1})}$, respectively.
\item[{\rm(g)}] $R_{k,\e}^{(i,j,z_{1,j})}$ converges to $0$ $\P$-a.s.\ as $k\to\infty$.
\item[{\rm(h)}] One has for $j\ge0$, $z_{1,j+1}\in\range^{j+1}$, and $s\ge1$
	\begin{align*}
	&\Rtil_{\e}^{(s,0,\varnothing)}+\Rtil_{\e}^{(s+1,0,\varnothing)}+\Rtil_{\e}^{(s+2,1,z_1)}+\cdots+\Rtil_{\e}^{(j+s,j-1,z_{1,j-1})}
	+\Rtil_{\e}^{(j+s+1,j,z_{1,j})}\\
	&\qquad\qquad\qquad=\Rtil_{\e}^{(s,1,z_1)}+\cdots+\Rtil_{\e}^{(j+s,j+1,z_{1,j+1})}+\Rhat_{\e}^{(j+s+1,j,z_{1,j+1})}.	
	\end{align*}
\item[{\rm(i)}] For any fixed $j\ge0$ and $z_{1,j+1}\in\range^{j+1}$ both $\Rtil_{\e}^{(i,j,z_{1,j})}$ and $\Rhat_{\e}^{(i,j,z_{1,j+1})}$ converge to $0$ strongly in $L^1(\P)$ as $i\to\infty$.
\end{itemize}
The limits as $k\to\infty$ are to
be understood in the sense that there exists one subsequence along which all the
countably many claimed
limits hold simultaneously.
\end{lemma}

Fix $i\ge0$ and let $k\ge i$.
Starting with \eqref{F-decomposition}, using (a) of the above lemma, and applying (c) repeatedly we have the decomposition $F^{(i)}_{k,\e}=\Fhat^{(i)}_{k,\e}-R^{(i)}_{k,\e}$ with 
	\[\Fhat^{(i)}_{k,\e}=\E[\Fhat^{(0)}_{k,\e}-\Rtil^{(1,0,\varnothing)}_{k,\e}-\cdots-\Rtil^{(i,0,\varnothing)}_{k,\e}\,|\,\kS_{k-i}]\quad\text{and}\quad
	R^{(i)}_{k,\e}=R_{k,\e}^{(i,0,\varnothing)}.\]
$R^{(i)}_{k,\e}$ is $\kS_{k-i}$-measurable and $\Fhat^{(i)}_{k,\e}$  uniformly integrable. (The proof of Theorem 5.1 in Chapter 4 of \cite{Dur-96} applies to the uniformly
integrable sequence $\Fhat^{(0)}_{k,\e}-\Rtil^{(1,0,\varnothing)}_{k,\e}-\cdots-\Rtil^{(i,0,\varnothing)}_{k,\e}$.)
One can check by a standard $\pi$-$\lambda$ or monotone class argument  that any
weak $L^1(\P)$ limit coincides with the weak limit without the conditioning, namely  
$\Fhat^{(i)}_\e=\Fhat^{(0)}_\e-\Rtil^{(1,0,\varnothing)}_\e-\cdots-\Rtil^{(i,0,\varnothing)}_\e.$
Furthermore, since $\E[R^{(0)}_{k,\e}]$ is uniformly bounded in $k$ we have 
\[\E[\Rtil^{(1,0,\varnothing)}_{k,\e}+\cdots+\Rtil^{(i,0,\varnothing)}_{k,\e}]\le \E[R^{(0)}_{k,\e}]\le C.\]
Taking $i\to\infty$ we see that $\Fhat^{(i)}_\e$ decreases converging strongly in $L^1(\P)$ to 
\[\Fhat_\e=\Fhat_\e^{(0)}-\sum_{i\ge1}\Rtil^{(i,0,\varnothing)}_\e.\]
Then, $\Fhat_\e\in L^1(\P)$ satisfies (i) of Definition \ref{cK-def}. 

Fix  a path $x_{0,\infty}$ with increments in $\range$. Fix integers $k\ge b\ge j\ge0$. 
%Observe next that for any bounded function $g(\w)$ we have $\E[ g \,|\, \kS_k ]\circ T_x = \E[ g\circ T_x  \,|\, T_{-x}\kS_k]$. 
Recall that $T_{\pm z}\kS_{s-1}\subset\kS_s$ for all $z\in\range$ and $s\ge1$. 
In particular, $\kS_{k-b+j}\supset T_{x_1}\kS_{k-b+j-1}\supset\cdots\supset T_{x_j}\kS_{k-b}$. Applying (e) and (b) of Lemma \ref{lm:construction} repeatedly one
has
	\begin{align}\label{whatever}
	\E\Big[R_{k,\e}^{(b-j)}\,\Big|\,T_{x_j}\kS_{k-b}\Big]=\E\Big[\sum_{s=1}^j\Rtil_{k,\e}^{(b-j+s,s,z_{1,s})}\,\Big|\,T_{x_j}\kS_{k-b}\Big]+R_{k,\e}^{(b,j,z_{1,j})}.
	\end{align}
Thus,
\begin{align}
&\E\Big[h_{k,\e}(T_{x_{j+1}}\w,z_{j+2,j+\ell+1})-h_{k,\e}(T_{x_j}\w,z_{j+1,j+\ell})\,\Big|\,\kS_{k-b}\Big]\label{h diff}\\
&\quad=\E\Big[\E[h_{k,\e}(T_{x_{j+1}}\w,z_{j+2,j+\ell+1})\nn\\
&\qquad\qquad\qquad\qquad-h_{k,\e}(T_{x_j}\w,z_{j+1,j+\ell})\,|\,T_{-x_j}\kS_{k-b+j}]\,\Big|\,\kS_{k-b}\Big]\nn\\
&\quad=\E\Big[F^{(b-j)}_{k,\e}(T_{x_j}\w,z_{j+1,j+\ell},z_{j+\ell+1})\,\Big|\,\kS_{k-b}\Big]\nn\\
&\quad=\E\Big[\Fhat^{(b-j)}_{k,\e}(T_{x_j}\w,z_{j+1,j+\ell},z_{j+\ell+1})\,\Big|\,\kS_{k-b}\Big]\nn\\
&\quad\quad\quad-\E\Big[R^{(b-j)}_{k,\e}(T_{x_j}\w,z_{j+1,j+\ell},z_{j+\ell+1})\,\Big|\,\kS_{k-b}\Big]\nn\\
&\quad=\E\Big[\Fhat^{(b-j)}_{k,\e}(T_{x_j}\w,z_{j+1,j+\ell},z_{j+\ell+1})\,\Big|\,\kS_{k-b}\Big]\label{ui part1}\\
&\quad\quad\quad-\E\Big[\Rtil^{(b-j+1,1,z_1)}_{k,\e}(\,\cdot\,,z_{j+1,j+\ell},z_{j+\ell+1})\nn\\
&\qquad\qquad\qquad+\cdots+\Rtil_{k,\e}^{(b,j,z_{1,j})}(\,\cdot\,,z_{j+1,j+\ell},z_{j+\ell+1})\,|\,T_{x_j}\kS_{k-b}\Big](T_{x_j}\w)\label{ui part2}\\
&\quad\quad\quad-R^{(b,j,z_{1,j})}_{k,\e}(T_{x_j}\w,z_{j+1,j+\ell},z_{j+\ell+1}).\nn%\label{error part}
%&\qquad=\E\Big[\Fhat^{(0)}_{k,\e}(T_{x_j}\w,z_{j+1,j+\ell},z_{j+\ell+1})\,\Big|\,\kS_{k-b}\Big]-\E\Big[R^{(0)}_{k,\e}(\w,z_{j+1,j+\ell},z_{j+\ell+1})\,\Big|\,T_{x_j}\kS_{k-b}\Big](T_{x_j}\w).\nn%\label{Fhat}
%&\qquad=\E\Big[\Fhat^{(0)}_{k,\e}(T_{x_j}\w,z_{j+1,j+\ell},z_{j+\ell+1})\,\Big|\,\kS_{k-b}\Big]\label{Fhat diff}\\
%&\qquad\qquad\qquad-\E\Big[(\Rtil^{(1)}_{k,\e}+\cdots+\Rtil^{(b+j)}_{k,\e})(\w,z_{j+1,j+\ell},z_{j+\ell+1})\,\Big|\,\kS_{k-b-j}\Big](T_{x_j}\w)\label{Rtil diff}\\
%&\qquad\qquad\qquad-I_k(T_{x_j}\w),\nn
\end{align}
The last equality used \eqref{whatever} and the formula $\E[ g \,|\, \kS_k ]\circ T_x = \E[ g\circ T_x  \,|\, T_{-x}\kS_k]$.
The two sequences in \eqref{ui part1} and \eqref{ui part2} are uniformly integrable and converge weakly in $L^1(\P)$ (along a subsequence) 
to $\Fhat^{(b-j)}_\e(T_{x_j}\w,z_{j+1,j+\ell},z_{j+\ell+1})$ and 
\[(\Rtil_\e^{(b-j+1,1,z_{1,j})}+\cdots+\Rtil^{(b,j,z_{1,j})}_\e)(T_{x_j}\w,z_{j+1,j+\ell},z_{j+\ell+1}),\] respectively. 
%(For the first one the proof of Theorem 5.1 in Chapter 4 of \cite{Dur-96} applies to the uniformly
%integrable sequences $\Fhat_{k,\e}^{(b-j)}\circ T_{x_j}$.)
%One can check by a standard $\pi$-$\lambda$ or monotone class argument  that the
%weak limit of the sequence on the left in \eqref{ui part} coincides with the weak limits without the conditioning, namely $\Fhat^{(b-j)}_\e(T_{x_j}\w,z_{j+1,j+\ell},z_{j+\ell+1})$. 
	
For any two paths 
$\{\wz_i\}_{i=0}^n$ and $\{\bar\wz_j\}_{j=0}^m$ 
as in (iii) of Definition \ref{cK-def}
\begin{align*}
\sum_{i=0}^{n-1}\E\Big[ h_{k,\e}(\Sopr_{a_{i+1}}\wz_i)-h_{k,\e}(\wz_i)\,\Big|\,\kS_{k-b}\Big]=\sum_{j=0}^{m-1} \E\Big[h_{k,\e}(\Sopr_{\bar a_{j+1}}\bar\wz_j)-h_{k,\e}(\bar\wz_j) \,\Big|\,\kS_{k-b}\Big].
\end{align*}

Now, taking $b>\max(m,n)$ and further subsequences of %convex combinations of 
\eqref{h diff} %, as was done on page \pageref{convex page}, 
we arrive at
\begin{align*}
&\sum_{i=0}^{n-1}\Big(\Fhat^{(b-i)}_\e(\wz_i,a_{i+1})-\sum_{s=1}^i\Rtil_{\e}^{(b-i+s,s,z_{1,s})}(\wz_i,a_{i+1})\Big)\\
&\qquad\qquad\qquad\qquad=\sum_{j=0}^{m-1}\Big(\Fhat^{(b-j)}_{\e}(\bar\wz_j,\bar a_{j+1})-\sum_{s=1}^j\Rtil_{\e}^{(b-j+s,s,z'_{1,s})}(\bar\wz_j,\bar a_{j+1})\Big).
\end{align*}
Here, $z_{1,n}$ and $z'_{1,m}$ denote the steps of the two paths corresponding to $\{\wz_i\}$ and $\{\bar\wz_j\}$.
Taking $b\to\infty$ and applying part (i) of Lemma \ref{lm:construction} we conclude that  $\Fhat_\e$ satisfies the closed loop property (iii) 
of Definition \ref{cK-def}.
Next, we work on the mean zero property. 

Abbreviate $\zhat_{1,\ell}=(z,\dotsc,z)\in\range^\ell$. Then,
	\begin{align*}
	c(z)&=\E[\Fhat_\e(\w,\zhat_{1,\ell},z)]
	=\inf_i\E[\Fhat^{(i)}_\e(\w,\zhat_{1,\ell},z)]=\inf_i\lim_{k\to\infty}\E[\Fhat^{(i)}_{k,\e}(\w,\zhat_{1,\ell},z)]\\
	&\ge\inf_i\lim_{k\to\infty}\E[F^{(i)}_{k,\e}(\w,\zhat_{1,\ell},z)]
	=\lim_{k\to\infty}\E[h_{k,\e}(T_z\w,\zhat_{1,\ell})-h_{k,\e}(\w,\zhat_{1,\ell})]=0.
	\end{align*}
Since $\Fhat_\e$ satisfies the closed loop property, one can define its path integral $\fhat_\e$ as above and use \eqref{deff} to write
\beq	\begin{aligned}
	\fhat_\e(\w,\zbar_{1,\ell},\zbar_{1,\ell},z)
	&=\Fhat_\e(T_{-\ell z}\w,\zhat_{1,\ell},z)\\
	&\qquad+\Fhat_\e(T_{-(\ell-1)z}\w,\zhat_{1,\ell},\zbar_1)
	+\cdots+\Fhat_\e(\w,(\zhat_\ell,\zbar_{1,\ell-1}),\zbar_\ell)\\
	&\qquad-\Fhat_\e(T_{-\ell z}\w,\zhat_{1,\ell},\zbar_1)
	-\cdots-\Fhat_\e(T_{-z}\w,(\zhat_\ell,\zbar_{1,\ell-1}),\zbar_\ell).
	\end{aligned} \label{temp-21}\eeq
Thus, we have $c(z)=\E[\fhat_\e(\w,\zbar_{1,\ell},\zbar_{1,\ell},z)]$ for all $\zbar_{1,\ell}\in\range^\ell$ and $z\in\range$. Hence
	\[c(z)=|\range|^{-\ell}\sum_{\zbar_{1,\ell}\in\range^\ell}\E[\fhat_\e(\w,\zbar_{1,\ell},\zbar_{1,\ell},z)].\]
Integrating \eqref{F-c} out (with $F=\Fhat_\e$) one sees that \[\E[g_n(\Sopr_z\wz)-g_n(\wz)]=\E[\Fhat_\e(\wz,z)-c(z_1)].\] Since $g_n(\Sopr_z\wz)-g_n(\wz)$ has the 
mean zero property (ii) of Definition \ref{cK-def}, we conclude that $\Fhat_\e(\wz,z)-c(z_1)$ does too. 
Let   $\zbar_{1,\ell}\in \range^\ell$,  $z_{1,n}\in\range^n$, $x=z_1+\cdots+z_n$,  and apply the mean zero property of
$\Fhat_\e(\wz,z)-c(z_1)$  to the path  that takes  steps   $(z_{1,n},\zbar_{1,\ell})$
to go from $(0,\zbar_{1,\ell})$  to $(x+\xbar_\ell,\zbar_{1,\ell})$.
This gives  
\beq
\E[\fhat_\e(\w,\zbar_{1,\ell},\zbar_{1,\ell},x+\xbar_\ell)]=c(\zbar_1)+\cdots+c(\zbar_\ell)+\sum_{i=1}^n c(z_i).\label{temp-22}\eeq
Since the left-hand side does not depend on $z_{1,n}$ as long as the increments add up to $x$ we see that $c(z_1)+\cdots+c(z_n)$ only depends
on $z_1+\cdots+z_n$.
Consequently, $\Fhat_\e(\wz,z)-c(z_1)$ also
has the closed loop property and thus belongs to $\K$. This completes the proof of Lemma \ref{F-lemma}.
\end{proof}

\begin{proof}[Proof of Lemma \ref{lm:construction}]
In what follows, decomposing a sequence $R_k\ge0$ means applying Lemma \ref{lm:raghu} to it. The leftmost term
in the decomposition is the one that 
converges in weak $L^1(\P)$ topology along a subsequence.  Its limit is denoted by the same symbol, with $k$ omitted. 
The rightmost term in the decomposition is the one converging to $0$ $\P$-a.s. 
Subsequences are chosen to work for all $z_{1,j}$, $j\ge1$, at once, and are
still indexed by $k$. Once a subsequence has been given to suit a decomposition, subsequent decompositions go along this subsequence, and so on. 
Induction will be repeatedly used in our
proof and once an induction is complete, the diagonal trick is used to obtain one subsequence that works for all the terms simultaneously. 
Recall that $R^{(0)}_{k,\e}\ge0$ and $\E[R^{(0)}_{k,\e}]$ is bounded uniformly in $k$.

The following diagram may be instructive to the reader during the course of the proof. Index the columns from left to right by $i=0,1,\dotsc,k$
and the rows from top to bottom by $j=0,1,\dotsc,k$.
\begin{align*}
\begin{matrix}
\kS_k&\kS_{k-1}&\kS_{k-2}&\cdots&\cdots&\kS_{k-i}&\cdots&\kS_0\\
&T_{x_1}\kS_{k-1}&T_{x_1}\kS_{k-2}&\cdots&\cdots&T_{x_1}\kS_{k-i}&\cdots&T_{x_1}\kS_0\\
&&T_{x_2}\kS_{k-2}&&&\vdots&&&\\
&&&\ddots&&T_{x_j}\kS_{k-i}&&\vdots\\
&&&&\ddots&\vdots&&\vdots\\
&&&&&T_{x_{i}}\kS_{k-i}&&\\
&&&&&&\ddots&&\\
&&&&&&&T_{x_{k}}\kS_0
\end{matrix}
\end{align*}
The algebra on row $j$ and column $i\ge j$ is $T_{x_j}\kS_{k-i}$. Each algebra on the diagram includes the one down and right of it, the one up and 
right of it, and the one to the right of it. The decomposition in (d) corresponds
to a step up and right in the diagram, while the decomposition in (e) corresponds
to a step down and right.

We will define $\Rtil_{k,\e}^{(i,j,z_{1,j})}$, $\Rhat_{k,\e}^{(i,j,z_{1,j+1})}$, and $R_{k,\e}^{(i,j,z_{1,j})}$ by induction on $s=i-j\ge0$. On the above diagram, this 
corresponds to the $s$-th diagonal starting at $\kS_{k-s}$ and going down to $T_{x_1}\kS_{k-s-1}$ and so on. We check property (i) after the whole induction 
process is complete.\medskip

{\sl Induction assumption for $s$:}  there exist nonnegative random variables on $\bigom_\ell\times\range$, denoted by $\Rtil_{k,\e}^{(i,j,z_{1,j})}$, 
$\Rhat_{k,\e}^{(i,j,z_{1,j+1})}$, and 
$R_{k,\e}^{(i,j,z_{1,j})}$, with $0\le j\le i\le k$, $i-j\le s$, and $z_{1,j+1}\in\range^{j+1}$, such that properties (a-h) are satisfied (whenever the terms 
involved have already been defined).\medskip

Set $R_{k,\e}^{(0,0,\varnothing)}=R_{k,\e}^{(0)}$ and $\Rtil_{k,\e}^{(0,0,\varnothing)}=0$.
For $k>j\ge0$, $z_{1,j+1}\in\range^{j+1}$ observe that $T_{x_{j+1}}\kS_{k-j-1}\subset T_{x_j}\kS_{k-j}$ and decompose inductively
\[\E[R_{k,\e}^{(j,j,z_{1,j})}\,|\,T_{x_{j+1}}\kS_{k-j-1}]=\Rtil_{k,\e}^{(j+1,j+1,z_{1,j+1})}+R_{k,\e}^{(j+1,j+1,z_{1,j+1})}.\]
For $j\ge0$ and $z_{1,j+1}\in\range^{j+1}$ set $\Rhat_{k,\e}^{(j,j,z_{1,j+1})}=0$. These are actually never used in properties (a-i) of the lemma.
This settles the case $s=0$.

Next, for $k>0$, decompose \[\E[R_{k,\e}^{(0,0,\varnothing)}\,|\,\kS_{k-1}]=\Rtil_{k,\e}^{(1,0,\varnothing)}+R_{k,\e}^{(1,0,\varnothing)}\]
and for $k>j\ge1$ decompose inductively 
\[\E[R_{k,\e}^{(j+1,j,z_{1,j})}\,|\,T_{x_{j+1}}\kS_{k-j-2}]=\Rtil_{k,\e}^{(j+2,j+1,z_{1,j+1})}+R_{k,\e}^{(j+2,j+1,z_{1,j+1})}.\]
Set $\Rhat_{k,\e}^{(j+1,j,z_{1,j+1})}=0$ for all $k>j\ge0$ and $z_{1,j+1}\in\range^{j+1}$. These are again not used in properties (a-i) of the lemma. This settles the 
case $s=1$.

Now fix $s\ge1$ and assume the induction assumption for this $s$. 
We will define $\Rtil_{k,\e}^{(j+s+1,j,z_{1,j})}$, $\Rhat_{k,\e}^{(j+s+1,j,z_{1,j+1})}$, and $R_{k,\e}^{(j+s+1,j,z_{1,j})}$
by induction on $j\ge0$. On the above diagram, this corresponds to going along the fixed $s$-th diagonal.\medskip

{\sl Induction assumption for $j$ with $s\ge1$ fixed:}  we have defined $\Rtil_{k,\e}^{(j+s+1,j,z_{1,j})}$, $\Rhat_{k,\e}^{(j+s+1,j,z_{1,j+1})}$, and 
$R_{k,\e}^{(j+s+1,j,z_{1,j})}$, 
such that properties (a-h) are satisfied (whenever the terms 
involved have already been defined).\medskip

Observe that $\kS_{k-s-1}\subset T_{z_1}\kS_{k-s}$ and temporarily decompose 
	\begin{align*}
	&\E[R_{k,\e}^{(s,0,\varnothing)}\,|\,\kS_{k-s-1}]=\Rtil_{k,\e}^{(s+1,0,\varnothing)}+R_{k,\e}^{(s+1,0,\varnothing)}\quad\text{and}\\
	&\E[R_{k,\e}^{(s,1,z_1)}\,|\,\kS_{k-s-1}]=\Rhat_{k,\e}^{(s+1,0,z_1)}+{\Rbar}_{k,\e}^{(s+1,0,z_1)}.
	\end{align*}
Let $R$ be the smallest of ${\Rbar}_{k,\e}^{(s+1,0,z_1)}$, $z_1\in\range$, and $R_{k,\e}^{(s+1,0,\varnothing)}$.
Use (c) and (e) with $i=s-1$ and $j=0$ to write
	\begin{align*}
	&\E[\Rtil_{k,\e}^{(s,0,\varnothing)}\,|\,\kS_{k-s-1}]+\Rtil_{k,\e}^{(s+1,0,\varnothing)}+R_{k,\e}^{(s+1,0,\varnothing)}\\
	&\qquad=\E[R_{k,\e}^{(s-1,0,\varnothing)}\,|\,\kS_{k-s-1}]
	=\E[\Rtil_{k,\e}^{(s,1,z_1)}\,|\,\kS_{k-s-1}]+\Rhat_{k,\e}^{(s+1,0,z_1)}+{\Rbar}_{k,\e}^{(s+1,0,z_1)}.
	\end{align*} 
The above display shows that the differences ${\Rbar}_{k,\e}^{(s+1,0,z_1)}-R$ and $R_{k,\e}^{(s+1,0,\varnothing)}-R$ are uniformly integrable.
Redefine all the terms ${\Rbar}_{k,\e}^{(s+1,0,z_1)}$, $z_1\in\range$, and $R_{k,\e}^{(s+1,0,\varnothing)}$
to equal $R$ and redefine $\Rtil_{k,\e}^{(s+1,0,\varnothing)}$ to equal 
$\Rtil_{k,\e}^{(s+1,0,\varnothing)}+R_{k,\e}^{(s+1,0,\varnothing)}-R$ and $\Rhat_{k,\e}^{(s+1,0,z_1)}$ to equal 
$\Rhat_{k,\e}^{(s+1,0,z_1)}+{\Rbar}_{k,\e}^{(s+1,0,z_1)}-R$. 
The upshot is that one can assume that ${\Rbar}_{k,\e}^{(s+1,0,z_1)}=R_{k,\e}^{(s+1,0,\varnothing)}$ for all $z_1\in\range$.
Taking $k\to\infty$ in the above display verifies (h) for $j=0$. This starts the induction at $j=0$. 
  
Now we go from $j$ to $j+1$. Temporarily decompose 
	\begin{align}
	&\E[R_{k,\e}^{(j+s+1,j,z_{1,j})}\,|\,T_{x_{j+1}}\kS_{k-j-s-2}]=\Rtil_{k,\e}^{(j+s+2,j+1,z_{1,j+1})}+R_{k,\e}^{(j+s+2,j+1,z_{1,j+1})}\quad\text{and}\nn\\
	&\E[R_{k,\e}^{(j+s+1,j+2,z_{1,j+2})}\,|\,T_{x_{j+1}}\kS_{k-j-s-2}]=\Rhat_{k,\e}^{(j+s+2,j+1,z_{1,j+2})}+{\Rbar}_{k,\e}^{(j+s+2,j+1,z_{1,j+2})}.\label{property:h}
	\end{align}
Then, one has
	\begin{align*}%\label{key}
	%\begin{split}
	&\E[\Rtil_{k,\e}^{(s,0,\varnothing)}+\Rtil_{k,\e}^{(s+1,0,\varnothing)}+\Rtil_{k,\e}^{(s+2,1,z_1)}+\cdots+\Rtil_{k,\e}^{(j+s+1,j,z_{1,j})}\,|\,T_{x_{j+1}}\kS_{k-j-s-2}]\\
	&\qquad\qquad\qquad\qquad+\Rtil_{k,\e}^{(j+s+2,j+1,z_{1,j+1})}+R_{k,\e}^{(j+s+2,j+1,z_{1,j+1})}\\
	&\qquad\qquad=\E[R_{k,\e}^{(s-1,0,\varnothing)}\,|\,T_{x_{j+1}}\kS_{k-j-s-2}]\\
	&\qquad\qquad=\E[\Rtil_{k,\e}^{(s,1,z_1)}+\cdots+\Rtil_{k,\e}^{(j+s+1,j+2,z_{1,j+2})}\,|\,T_{x_{j+1}}\kS_{k-j-s-2}]\\
	&\qquad\qquad\qquad\qquad+\Rhat_{k,\e}^{(j+s+2,j+1,z_{1,j+2})}+{\Rbar}_{k,\e}^{(j+s+2,j+1,z_{1,j+2})}.	
	%\end{split}
	\end{align*}
Here is an explanation: use (c) first twice to condition $R^{(s-1,0,\varnothing)}$ on $\kS_{k-s}$ and then on $\kS_{k-s-1}$.
Next, use (e) conditioning $R^{(s+1,0,\varnothing)}$ on $T_{x_1}\kS_{k-s-2}$ then 
$R^{(s+2,1,z_1)}$ on $T_{x_2}\kS_{k-s-3}$ and so on until conditioning $R^{(s+j+1,j,z_{1,j})}$ on $T_{x_{j+1}}\kS_{k-s-j-2}$. 
On the other side use (e) conditioning $R^{(s-1,0,\varnothing)}$ on $T_{x_1}\kS_{k-s}$ then 
$R^{(s,1,x_1)}$ conditioned on $T_{x_2}\kS_{k-s-1}$ and so on until $R^{(s+j,j+1,z_{1,j+1})}$ is conditioned on $T_{x_{j+2}}\kS_{k-s-j-1}$.
Then use \eqref{property:h} and condition $R^{(s+j+1,j+2,z_{1,j+2})}$ on $T_{x_{j+1}}\kS_{k-s-j-2}$.

Now, repeating what we have done for the case $j=0$, we can assume that ${\Rbar}_{k,\e}^{(j+s+2,j+1,z_{1,j+2})}=R_{k,\e}^{(j+s+2,j+1,z_{1,j+1})}$ for all 
$z_{1,j+2}\in\range^{j+2}$. Taking $k\to\infty$ verifies (h).

We have achieved the induction on $j$ and thus also the induction on $s$. Our construction is thus complete once we prove it satisfies (i).
Using (h), this follows easily by induction on $j\ge0$ once one shows that $\Rtil_\e^{s,0,\varnothing}\to0$ strongly in $L^1(\P)$, which itself follows from
the fact that $\E[\Rtil_{k,\e}^{(1,0,\varnothing)}+\cdots+\Rtil_{k,\e}^{(s,0,\varnothing)}]\le\E[R_{k,\e}^{(0)}]$ is uniformly bounded in $k$.
The lemma is proved.
\end{proof}

%%      ---------------------------------------------------------------------
%%      ---------------------------ACKNOWLEDGMENTS (OPTIONAL) ---------------
%%      ---------------------------------------------------------------------

%% ***** UNCOMMENT THE FOLLOWING LINE TO ADD ACKNOWLEDGMENTS.

 \ack 

%%      Type acknowledgments here.

%F.\ Rassoul-Agha
We  thank Elena Kosygina and Fraydoun Rezakhanlou  for valuable discussions.
Rassoul-Agha was supported in part by NSF Grant DMS-0747758. Sepp\"al\"ainen was 
supported in part by NSF Grant %DMS-0701091 and 
DMS-1003651 and by the Wisconsin Alumni Research Foundation.

%%      ---------------------------------------------------------------------
%%      --------------------------- BIBLIOGRAPHY ----------------------------
%%      ---------------------------------------------------------------------

\frenchspacing
\bibliographystyle{cpam}
\bibliography{firasbib2010}

\begin{thebibliography}{10}
\providecommand{\url}[1]{\texttt{#1}}
\providecommand{\urlprefix}{URL }
\providecommand{\eprint}[2][]{\url{#2}}

\bibitem{Arm-Sou-11-}
Armstrong, S.~N.; Souganidis, P.~E. Stochastic homogenization of
  Hamilton-Jacobi and degenerate Bellman equations in unbounded environments
  (2011). \urlprefix\url{http://arxiv.org/abs/1103.2814v4}.

\bibitem{Ave-Hol-Red-10}
Avena, L.; den Hollander, F.; Redig, F. Large deviation principle for
  one-dimensional random walk in dynamic random environment: attractive
  spin-flips and simple symmetric exclusion. \emph{Markov Process. Related
  Fields} \textbf{16} (2010), no.~1, 139--168.

\bibitem{Bol-89}
Bolthausen, E. A note on the diffusion of directed polymers in a random
  environment. \emph{Comm. Math. Phys.} \textbf{123} (1989), no.~4, 529--534.

\bibitem{Bol-Szn-02-dmv}
Bolthausen, E.; Sznitman, A.-S. \emph{Ten lectures on random media}, \emph{DMV
  Seminar}, vol.~32. Birkh{\"a}user Verlag, Basel, 2002.

\bibitem{Car-Hu-04}
Carmona, P.; Hu, Y. Fluctuation exponents and large deviations for directed
  polymers in a random environment. \emph{Stochastic Process. Appl.}
  \textbf{112} (2004), no.~2, 285--308.

\bibitem{Che-67-eng}
Chernov, A.~A. Replication of a multicomponent chain, by the ``lightning
  mechanism''. \emph{Biophysics.} \textbf{12} (1967), 336--341.

\bibitem{Com-Gan-Zei-00}
Comets, F.; Gantert, N.; Zeitouni, O. Quenched, annealed and functional large
  deviations for one-dimensional random walk in random environment.
  \emph{Probab. Theory Related Fields} \textbf{118} (2000), no.~1, 65--114.

\bibitem{Com-Shi-Yos-03}
Comets, F.; Shiga, T.; Yoshida, N. Directed polymers in a random environment:
  path localization and strong disorder. \emph{Bernoulli} \textbf{9} (2003),
  no.~4, 705--723.

\bibitem{Com-Shi-Yos-04}
Comets, F.; Shiga, T.; Yoshida, N. Probabilistic analysis of directed polymers
  in a random environment: a review. In \emph{Stochastic analysis on large
  scale interacting systems}, \emph{Adv. Stud. Pure Math.}, vol.~39, pp.
  115--142, Math. Soc. Japan, Tokyo, 2004.

\bibitem{Cox-Gan-Gri-93}
Cox, J.~T.; Gandolfi, A.; Griffin, P.~S.; Kesten, H. Greedy lattice animals.
  {I}. {U}pper bounds. \emph{Ann. Appl. Probab.} \textbf{3} (1993), no.~4,
  1151--1169.

\bibitem{Dem-Zei-98}
Dembo, A.; Zeitouni, O. \emph{Large deviations techniques and applications},
  \emph{Applications of Mathematics (New York)}, vol.~38. Springer-Verlag, New
  York, 2nd ed., 1998.

\bibitem{Hol-00}
den Hollander, F. \emph{Large deviations}, \emph{Fields Institute Monographs},
  vol.~14. American Mathematical Society, Providence, RI, 2000.

\bibitem{Hol-09}
den Hollander, F. \emph{Random polymers}, \emph{Lecture Notes in Mathematics},
  vol. 1974. Springer-Verlag, Berlin, 2009. Lectures from the 37th Probability
  Summer School held in Saint-Flour, 2007.

\bibitem{Deu-Str-89}
Deuschel, J.-D.; Stroock, D.~W. \emph{Large deviations}, \emph{Pure and Applied
  Mathematics}, vol. 137. Academic Press Inc., Boston, MA, 1989.

\bibitem{Dre-etal-10-}
Drewitz, A.; G{\"a}rtner, J.; Ram{\'\i}rez, A.~F.; Sun, R. Survival Probability
  of a Random Walk Among a Poisson System of Moving Traps  (2010).
  \urlprefix\url{http://arxiv.org/abs/1010.3958}.

\bibitem{Dun-Sch-58}
Dunford, N.; Schwartz, J.~T. \emph{Linear {O}perators. {I}. {G}eneral
  {T}heory}. With the assistance of W. G. Bade and R. G. Bartle. Pure and
  Applied Mathematics, Vol. 7, Interscience Publishers, Inc., New York, 1958.

\bibitem{Dur-96}
Durrett, R. \emph{Probability: theory and examples}. Duxbury Press, Belmont,
  CA, 2nd ed., 1996.

\bibitem{E-Weh-Xin-08}
E, W.; Wehr, J.; Xin, J. Breakdown of homogenization for the random
  {H}amilton-{J}acobi equations. \emph{Commun. Math. Sci.} \textbf{6} (2008),
  no.~1, 189--197.

\bibitem{Gan-Kes-94}
Gandolfi, A.; Kesten, H. Greedy lattice animals. {II}. {L}inear growth.
  \emph{Ann. Appl. Probab.} \textbf{4} (1994), no.~1, 76--107.

\bibitem{Geo-88}
Georgii, H.-O. \emph{Gibbs measures and phase transitions}, \emph{de Gruyter
  Studies in Mathematics}, vol.~9. Walter de Gruyter \& Co., Berlin, 1988.

\bibitem{Gia-07}
Giacomin, G. \emph{Random polymer models}. Imperial College Press, London,
  2007.

\bibitem{Gre-Hol-94}
Greven, A.; den Hollander, F. Large deviations for a random walk in random
  environment. \emph{Ann. Probab.} \textbf{22} (1994), no.~3, 1381--1428.

\bibitem{Hav-Ben-02}
Havlin, S.; Ben-Avraham, D. {Diffusion in disordered media}. \emph{{Advances in
  Physics}} \textbf{{51}} ({2002}), no.~{1}, {187--292}.

\bibitem{Hus-Hen-85}
Huse, D.~A.; Henley, C.~L. Pinning and Roughening of Domain Walls in Ising
  Systems Due to Random Impurities. \emph{Phys. Rev. Lett.} \textbf{54} (1985),
  no.~25, 2708--2711.

\bibitem{Imb-Spe-88}
Imbrie, J.~Z.; Spencer, T. Diffusion of directed polymers in a random
  environment. \emph{J. Statist. Phys.} \textbf{52} (1988), no. 3-4, 609--626.

\bibitem{Kas-94}
Kassay, G. A simple proof for {K}{\"o}nig's minimax theorem. \emph{Acta Math.
  Hungar.} \textbf{63} (1994), no.~4, 371--374.

\bibitem{Kos-Rez-Var-06}
Kosygina, E.; Rezakhanlou, F.; Varadhan, S. R.~S. Stochastic homogenization of
  {H}amilton-{J}acobi-{B}ellman equations. \emph{Comm. Pure Appl. Math.}
  \textbf{59} (2006), no.~10, 1489--1521.

\bibitem{Kos-Var-08}
Kosygina, E.; Varadhan, S. R.~S. Homogenization of
  {H}amilton-{J}acobi-{B}ellman equations with respect to time-space shifts in
  a stationary ergodic medium. \emph{Comm. Pure Appl. Math.} \textbf{61}
  (2008), no.~6, 816--847.

\bibitem{Lio-Sou-05}
Lions, P.-L.; Souganidis, P.~E. Homogenization of ``viscous''
  {H}amilton-{J}acobi equations in stationary ergodic media. \emph{Comm.
  Partial Differential Equations} \textbf{30} (2005), no. 1-3, 335--375.

\bibitem{Mar-02}
Martin, J.~B. Linear growth for greedy lattice animals. \emph{Stochastic
  Process. Appl.} \textbf{98} (2002), no.~1, 43--66.

\bibitem{Ras-Sep-10-ldp-}
Rassoul-Agha, F.; Sepp\"al\"ainen, T. \emph{A course on large deviation theory
  with an introduction to {G}ibbs measures}. 2010. Preprint,
  \urlprefix\url{http://www.math.utah.edu/~firas/Papers/rassoul-seppalainen-ldp.pdf}.

\bibitem{Ras-Sep-10-}
Rassoul-Agha, F.; Sepp\"al\"ainen, T. Process-level quenched large deviations
  for random walk in random environment. \emph{Ann. Inst. H. Poincar\'e Probab.
  Statist.} \textbf{47} (2011), no.~1, 214--242.

\bibitem{Rio-95}
Rio, E. The functional law of the iterated logarithm for stationary strongly
  mixing sequences. \emph{Ann. Probab.} \textbf{23} (1995), no.~3, 1188--1203.

\bibitem{Roc-70}
Rockafellar, R.~T. \emph{Convex analysis}. Princeton Mathematical Series, No.
  28, Princeton University Press, Princeton, N.J., 1970.

\bibitem{Ros-06}
Rosenbluth, J.~M. \emph{Quenched large deviation for multidimensional random
  walk in random environment: {A} variational formula}. ProQuest LLC, Ann
  Arbor, MI, 2006. Thesis (Ph.D.)--New York University,
  \urlprefix\url{http://arxiv.org/abs/0804.1444}.

\bibitem{Rub-Col-03}
Rubinstein, M.; Colby, R.~H. \emph{Polymer physics}. Oxford Univ. Press, 2003.

\bibitem{Rud-91}
Rudin, W. \emph{Functional analysis}. International Series in Pure and Applied
  Mathematics, McGraw-Hill Inc., New York, 2nd ed., 1991.

\bibitem{Sch-88}
Schroeder, C. Green's functions for the {S}chr{\"o}dinger operator with
  periodic potential. \emph{J. Funct. Anal.} \textbf{77} (1988), no.~1, 60--87.

\bibitem{Sol-75}
Solomon, F. Random walks in a random environment. \emph{Ann. Probability}
  \textbf{3} (1975), 1--31.

\bibitem{Spi-76}
Spitzer, F. \emph{Principles of random walks}. Springer-Verlag, New York, 2nd
  ed., 1976. Graduate Texts in Mathematics, Vol. 34.

\bibitem{Szn-94}
Sznitman, A.-S. Shape theorem, {L}yapounov exponents, and large deviations for
  {B}rownian motion in a {P}oissonian potential. \emph{Comm. Pure Appl. Math.}
  \textbf{47} (1994), no.~12, 1655--1688.

\bibitem{Szn-98}
Sznitman, A.-S. \emph{Brownian motion, obstacles and random media}. Springer
  Monographs in Mathematics, Springer-Verlag, Berlin, 1998.

\bibitem{Szn-04}
Sznitman, A.-S. Topics in random walks in random environment. In \emph{School
  and {C}onference on {P}robability {T}heory}, ICTP Lect. Notes, XVII, pp.
  203--266 (electronic), Abdus Salam Int. Cent. Theoret. Phys., Trieste, 2004.

\bibitem{Tem-72-eng}
Temkin, D.~E. One-dimensional random walks in a two-component chain.
  \emph{Soviet Math. Dokl.} \textbf{13} (1972), 1172--1176.

\bibitem{Var-84}
Varadhan, S. R.~S. \emph{Large deviations and applications}, \emph{CBMS-NSF
  Regional Conference Series in Applied Mathematics}, vol.~46. Society for
  Industrial and Applied Mathematics (SIAM), Philadelphia, PA, 1984.

\bibitem{Var-03-cpam}
Varadhan, S. R.~S. Large deviations for random walks in a random environment.
  \emph{Comm. Pure Appl. Math.} \textbf{56} (2003), no.~8, 1222--1245.
  Dedicated to the memory of J{{\"u}}rgen K. Moser.

\bibitem{Var-07}
Vargas, V. Strong localization and macroscopic atoms for directed polymers.
  \emph{Probab. Theory Related Fields} \textbf{138} (2007), no. 3-4, 391--410.

\bibitem{Yil-09-aop}
Yilmaz, A. Large deviations for random walk in a space-time product
  environment. \emph{Ann. Probab.} \textbf{37} (2009), no.~1, 189--205.

\bibitem{Yil-09-cpam}
Yilmaz, A. Quenched large deviations for random walk in a random environment.
  \emph{Comm. Pure Appl. Math.} \textbf{62} (2009), no.~8, 1033--1075.

\bibitem{Zei-04}
Zeitouni, O. Random walks in random environment. In \emph{Lectures on
  probability theory and statistics}, \emph{Lecture Notes in Math.}, vol. 1837,
  pp. 189--312, Springer, Berlin, 2004.

\bibitem{Zer-98-aap}
Zerner, M. P.~W. Directional decay of the {G}reen's function for a random
  nonnegative potential on {${\bf Z}^d$}. \emph{Ann. Appl. Probab.} \textbf{8}
  (1998), no.~1, 246--280.

\bibitem{Zer-98-aop}
Zerner, M. P.~W. Lyapounov exponents and quenched large deviations for
  multidimensional random walk in random environment. \emph{Ann. Probab.}
  \textbf{26} (1998), no.~4, 1446--1476.

\end{thebibliography}

\bigskip

\parindent=0pt

\begin{minipage}{2.4in}
{\sc Firas Rassoul-Agha}

Department of Mathematics

University of Utah

155 South 1400 East

Salt Lake City, UT 84109

%USA

E-mail: {\tt firas@math.utah.edu}

URL: \href{http://www.math.utah.edu/~firas}{\tt www.math.utah.edu/$\sim$firas}
\end{minipage}
\hfill
\begin{minipage}{2.55in}
{\sc Timo Sepp\"al\"ainen}

Department of Mathematics

University of Wisconsin-Madison

419 Van Vleck Hall

Madison, WI 53706

%USA

E-mail: {\tt seppalai@math.wisc.edu}

URL: \href{http://www.math.wisc.edu/~seppalai}{\tt www.math.wisc.edu/$\sim$seppalai}
\end{minipage}
\bigskip

\begin{minipage}{3.5in}
{\sc Atilla Yilmaz}

Department of Mathematics

Bo\u{g}azi\c{c}i University

TR-34342 Bebek

Istanbul, TURKEY

%USA

E-mail: {\tt  atilla.yilmaz@boun.edu.tr}

URL: \href{http://www.math.boun.edu.tr/instructors/yilmaz/}{\tt http://www.math.boun.edu.tr/instructors/yilmaz/}
\end{minipage}\bigskip

\end{document}